\documentclass[12pt, a4paper]{article}

\usepackage{hyperref}
\usepackage{amssymb,amsmath,amsthm}
\usepackage[all]{xy}

\theoremstyle{plain} 
\newtheorem{prop0}{Proposition}[section]
\newtheorem{lem0}[prop0]{Lemma}
\newtheorem{thm0}[prop0]{Theorem}
\newtheorem{conj0}[prop0]{Conjecture}
\theoremstyle{definition}
\newtheorem{rem0}[prop0]{Remark}
\theoremstyle{plain} 
\newtheorem{prop}{Proposition}[subsection]
\newtheorem{conj}[prop]{Conjecture}
\newtheorem{lem}[prop]{Lemma}
\newtheorem{thm}[prop]{Theorem}
\newtheorem{cor}[prop]{Corollary}
\theoremstyle{definition}
\newtheorem{definit}[prop]{Definition}
\newtheorem{ex}[prop]{Example}
\newtheorem{rem}[prop]{Remark}

\def\ilim#1{\displaystyle \lim_{\stackrel{\longrightarrow}{#1}}}
\def\plim#1{\displaystyle \lim_{\stackrel{\longleftarrow}{#1}}}
\def\varddots{\mathinner{\raise7pt\vbox{\kern3pt\hbox{.}}\mkern1mu\smash{\raise4pt\hbox{.}}\mkern1mu\smash{\raise1pt\hbox{.}}}}

\DeclareMathOperator{\Spec}{Spec}
\DeclareMathOperator{\Frob}{Frob}

\DeclareMathOperator{\Ker}{Ker}

\DeclareMathOperator{\ord}{ord}
\DeclareMathOperator{\Ord}{Ord}

\DeclareMathOperator{\alg}{alg}

\DeclareMathOperator{\soc}{soc}
\DeclareMathOperator{\diag}{diag}
\DeclareMathOperator{\End}{End}
\DeclareMathOperator{\Ind}{Ind}
\DeclareMathOperator{\ind}{c-Ind}
\DeclareMathOperator{\Hom}{Hom}
\DeclareMathOperator{\Gal}{Gal}

\DeclareMathOperator{\GL}{GL}
\DeclareMathOperator{\SL}{SL}
\DeclareMathOperator{\nr}{\mathrm{nr}}

\DeclareMathOperator{\ab}{\mathrm{ab}}

\newcommand{\LL}{L^{\!\otimes}}
\newcommand{\LLbar}{\overline L^{\otimes}}
\newcommand{\Lbar}{\overline L}

\newcommand{\oE}{{\mathcal O}_E}
\newcommand{\oFF}{{\mathcal O}_{\!F^+}}
\newcommand{\oFFv}{{\mathcal O}_{\!F_v^+}}
\newcommand{\oFFp}{{\mathcal O}_{F^+\!,p}}
\newcommand{\oF}{{\mathcal O}_{\!F}}
\newcommand{\oFw}{{\mathcal O}_{\!F_w}}
\newcommand{\pE}{\varpi_E}
\newcommand{\Z}{{\mathbb Z}}
\newcommand{\Q}{{\mathbb Q}}
\newcommand{\R}{{\mathbb R}}
\newcommand{\Qp}{\Q_{p}}

\newcommand{\Zp}{\Z_{p}}
\newcommand{\Fp}{{\mathbb F}_{p}}

\newcommand{\Qpbar}{{\overline{\mathbb Q}}_p}

\newcommand{\rhobar}{\overline\rho}
\newcommand{\rbar}{\overline{r}}

\newcommand{\Gp}{{\mathrm{GL}}_2(\Qp)}

\newcommand{\Av}{{\mathbb A}_{F^+}^{\infty,v}}
\newcommand{\Ap}{{\mathbb A}_{F^+}^{\infty,p}}
\newcommand{\A}{{\mathbb A}_{F^+}^{\infty}}
\newcommand{\GAv}{G(\Av)}
\newcommand{\GAp}{G(\Ap)}
\newcommand{\GA}{G(\A)}

\newcommand{\gp}{{\Gal}(\Qpbar/\Qp)}

\newcommand{\gF}{{\Gal}(\overline F/F)}

\newcommand{\TT}{{\mathbb T}}
\newcommand{\sss}{{\rm ss}}
\newcommand{\m}{{\mathfrak m}}
\newcommand{\p}{{\mathfrak p}}

\newcommand{\ip}{{\Gal}(\Qpbar/\Qp^{\rm nr})}

\newcommand{\smat}[1]{\left( \begin{smallmatrix} #1 \end{smallmatrix} \right)}

\renewcommand{\emptyset}{\varnothing}
\newcommand{\congto}{\xrightarrow{\,\sim\,}}
\newcommand{\bs}{\backslash}
\renewcommand{\o}[1]{\overline{#1}}

\newcommand{\wh}[1]{\widehat{#1}}

\renewcommand{\(}{\textup{(}}
\renewcommand{\)}{\textup{)}}

\title{Ordinary representations of $G(\Qp)$ and fundamental algebraic representations}

\author{Christophe Breuil\footnote{C.\ B.\ was supported by the CNRS and the project Th\'eHopaD ANR-2011-BS01-005.}\\
C.N.R.S.\\
Universit\'e Paris-Sud\\
B\^atiment 425\\
91405 Orsay Cedex\\
France
\and
Florian Herzig\footnote{F.\ H.\ was partially supported by NSF grant DMS-0902044, NSERC grant RGPIN 402885, and a Sloan Fellowship.}\\
Dept.\ of Mathematics\\
University of Toronto\\
40 St.\ George St., BA6290\\
Toronto, ON M5S 2E4\\
Canada}
\date{ }

\begin{document} 

\maketitle

\setcounter{tocdepth}{2}


\tableofcontents

\section{Introduction}

The $p$-adic Langlands programme for the group $\Gp$ is now well understood, both from a local and a global point of view (\cite{Co}, \cite{Pa2}, \cite{CDP}, \cite{Em4}, see \cite{Br3} for an overview). \ In \ particular, \ to \ (essentially) \ any \ continuous \ representation $\rho:\gp\rightarrow \GL_2(E)$ (where $E$ is a finite extension of $\Qp$) one can associate a unitary continuous representation $\Pi(\rho)$ of $\Gp$ on a $p$-adic Banach space over $E$. Likewise, to (essentially) any continuous representation $\rhobar:\gp\rightarrow \GL_2(k_E)$ (where $k_E$ is the residue field of $E$), one can associate a smooth representation $\Pi(\rhobar)$ of $\Gp$ over $k_E$. It is moreover expected that the $p$-adic Langlands correspondence -- if there is one -- for any other group beyond $\Gp$ (e.g.\ $\GL_2(L)$ or $\GL_3(\Qp)$ or ${\rm GSp}_{4}(\Qp)$) will be significantly more involved than for $\Gp$ (see e.g.\ \cite[\S3, \S4]{Br3} or \cite{Sc}). The aim of the present work is nevertheless to start to investigate the possible shape of the representation(s) $\Pi(\rho)$ and $\Pi(\rhobar)$ when $\rho$, $\rhobar$ take values in split reductive groups other than $\GL_2$.

Let $G$ be a split connected reductive algebraic group over $\Qp$ with dual $\widehat G$, $E$ a finite extension of $\Qp$ and $\rho:\gp\rightarrow \widehat G(E)$ a continuous representation. Assume that both $G$ and $\widehat G$ have a connected centre and that (up to conjugation) $\rho$ takes values in a Borel subgroup $\widehat B(E)$ of $\widehat G(E)$. In this setting, at least when $\rho$ is sufficiently generic, we define a unitary continuous representation $\Pi(\rho)^{\rm ord}$ that we expect to be the maximal closed subrepresentation of $\Pi(\rho)$ whose constituents are subquotients of unitary continuous principal series of $G(\Qp)$ over $E$. We define an analogous smooth representation $\Pi(\rhobar)^{\rm ord}$ when $G$ is a split connected reductive algebraic group over $\Zp$ and $\rhobar:\gp\rightarrow \widehat B(k_E)$. When $G=\GL_n$ and $\rhobar$ comes from some (automorphic) global Galois representation $\rbar$, we moreover prove using results of Gee and Geraghty (\cite{Ge}, \cite{GG}) that (under suitable assumptions) the $\GL_n(\Qp)$-representation $\Pi(\rhobar)^{\rm ord}$ occurs in the $\rbar$-part of spaces of mod $p$ automorphic forms for certain definite unitary groups which are outer forms of $\GL_n$. We only consider split reductive $G$ in this paper because, when $G$ is not split, e.g.\ $G={\rm Res}_{L/\Qp}\GL_2$, the results of \cite{BP} (see the remarks at the end of section 19 of {\it loc.\ cit.}) as well as \cite[Thm.\ 1.2]{Ha1} suggest that the representations $\Pi(\rho)^{\rm ord}$, $\Pi(\rhobar)^{\rm ord}$ are generically semi-simple and thus not very interesting.

We now try to motivate the idea underlying the construction of $\Pi(\rho)^{\rm ord}$ and $\Pi(\rhobar)^{\rm ord}$. One crucial ingredient in the $p$-adic Langlands correspondence for $\Gp$ is the construction by Colmez (\cite{Co}) of a covariant exact functor $F$ from unitary continuous (resp.\ smooth) representations of $\Gp$ over $E$ (resp.\ $k_E$) with suitable properties to finite-dimensional representations of $\gp$ over $E$ (resp.\ $k_E$) sending $\Pi(\rho)$ (resp.\ $\Pi(\rhobar)$) to $\rho$ (resp.\ $\rhobar$), see \cite[\S3.4]{Em4} for more details in the ordinary case. One may hope that $F$ has an analogue when $\GL_2$ is replaced by $G$ as above, see \cite{SV} for one tentative construction (see also the recent \cite {Br7}). Denoting by $L^{\otimes}\circ\rho$ the tensor product of all the fundamental algebraic representations of $\widehat G(E)$ composed with $\rho$, evidence coming from various sources (explicit computations with Serre weights, locally analytic vectors, the results of \cite{Ha1}, \cite{Ha2}, etc.) suggest that the internal structure of the representation $\Pi(\rho)$ (its constituents and socle filtration) should somehow ``reflect'' the internal structure of the Galois representation $L^{\otimes}\circ\rho$, and that likewise $\Pi(\rho)^{\rm ord}$ should reflect the structure of an explicit subrepresentation $(L^{\otimes})^{\rm ord}\circ\rho$ of $L^{\otimes}\circ\rho$ (see below). Note that if $\rho$ is the restriction of a global automorphic Galois representation $r$, it is natural to expect that the $r$-part of spaces of $p$-adic/mod $p$ automorphic forms is a direct sum of finitely many copies of $\Pi(\rho)$. (These globally defined spaces are denoted by $\widehat S(U^p,E)[\p^{\Sigma}]$ and $\widehat S(U^p,E)[\m^{\Sigma}]$ in section \ref{globconj}.) This would be a (weak) analogue of Emerton's local-global compatibility when $n=2$ (\cite{Em4}). Therefore, although we don't have $\Pi(\rho)$ and $\Pi(\rhobar)$, if a functor $F$ exists when $n>2$, one may ask whether $F(\Pi(\rho))=L^{\otimes}\circ\rho$ and $F(\Pi(\rhobar))=L^{\otimes}\circ\rhobar$ (see \S\ref{ques} in the text for more details). Let us just emphasize that, although we don't know $\Pi(\rho)$ or $\Pi(\rhobar)$, we can {\it use} $L^{\otimes}\circ\rho$ as a guide to make {\it predictions} about $\Pi(\rho)$ or $\Pi(\rhobar)$ and then test these predictions on cohomology. This is what we start to do in this paper in the case where $\rho$ (resp.\ $\rhobar$) takes values in $\widehat B(E)$ (resp.\ $\widehat B(k_E)$) and is sufficiently generic. But we do hope that the consideration of $L^{\otimes}\circ\rho$ will also be of importance for Galois representations which are not Borel-valued.

In light of the expected local-global compatibility for ${\rm GL}_n$, our main theorems suggest that $\Pi(\rho)^{\rm ord}$ is contained in $\Pi(\rho)$ and $\Pi(\rhobar)^{\rm ord}$ is contained in $\Pi(\rhobar)$, as expected. Moreover, the expected local-global compatibility also gives evidence for our expectation that $\Pi(\rho)^{\rm ord}$ is distinct from $\Pi(\rho)$ and $\Pi(\rhobar)^{\rm ord}$ is distinct from $\Pi(\rhobar)$ for $n>2$. In the mod $p$ case, if $\rhobar$ is semisimple and sufficiently generic, then $\Pi(\rhobar)^{\rm ord}$ contains $n!$ irreducible ${\rm GL}_n(\Zp)$-subrepresentations up to isomorphism, whereas $\Pi(\rhobar)$ should contain a lot more by the Serre-type conjecture of \cite{He1} (see also \cite{EGH}). In the $p$-adic case, suppose for instance that $n=3$ and that $\rho$ is crystabelian as in the second example of \cite[\S6.2]{Br6} (see the case $a_1a_2a_3\ne 0$ and $a_1a_2=a_3$ of {\it loc.\ cit.}). Then \cite[Thms.\ 9.3, 9.10]{Br5} and \cite[Lem.\ 8.8]{Br4} show that the locally analytic representation $C(s_\alpha,s_\alpha s_\beta)$ of \cite[\S6.2]{Br6} should occur as a subrepresentation of $\Pi(\rho)$ but cannot occur as a subrepresentation of $\Pi(\rho)^{\rm ord}$ (see also the forthcoming \cite{HM} for an analogous result in the mod $p$ case).

Let us now describe the main results of the paper.

In section \ref{LL}, for any closed subgroup $\widehat B_C$ of $\widehat B$ containing the maximal torus, we completely describe the maximal $\widehat B_C$-subrepresentation $(L^{\otimes}\vert_{\widehat B_C})^{\rm ord}$ of $L^{\otimes}\vert_{\widehat B_C}$ such that all its weights are in the orbit under the Weyl group $W$ of the highest weight $\lambda$ of $L^{\otimes}$. One finds a direct sum of indecomposable $\widehat B_C$-representations:
$$(\LL\vert_{\widehat B_C})^{\rm ord}=\oplus_{w\in W_C}\LL_{C,w},$$
where $W_C:=\{w\in W : \dot{w}^{-1}\widehat B_C\dot{w}\subseteq \widehat B\}$.

Let $\rho:\gp\rightarrow \widehat B(E)$ and $\widehat\chi_{\rho}:\gp\buildrel\rho\over \rightarrow \widehat B(E)\twoheadrightarrow \widehat T(E)$. In section \ref{ord}, we use this description to associate to any sufficiently generic such $\rho$ an admissible unitary continuous representation $\Pi(\rho)^{\rm ord}$ of $G(\Qp)$ over $E$ which is a successive extension of finitely many unitary continuous principal series. First, we associate to $\rho$ a closed subgroup $\widehat B_{C_\rho}$ of $\widehat B$ (defined as the smallest Zariski closed subgroup of $\widehat B$ such that $\rho$ takes values in $\widehat B_{C_\rho}(E)$). We may assume (after conjugation) that $\widehat B_{C_{\rho}}$ is minimal among all $\widehat B(E)$-conjugates of $\rho$. Then we define $\Pi(\rho)^{\rm ord}$ as
$$\Pi(\rho)^{\rm ord}:=\oplus_{w\in W_{C_\rho}}\Pi(\rho)_{C_\rho,w},$$
where each $\Pi(\rho)_{C_\rho,w}$ is a successive extension of unitary continuous principal series of $G(\Qp)$ over $E$.
More precisely, $\Pi(\rho)_{C_\rho,w}$ mimics the structure of $\LL_{C_{\rho},w}$ in the following way: each time the weight $w'(\lambda)$ appears in $(L^{\otimes}\vert_{\widehat B_{C_\rho}})^{\rm ord}$ (for some $w'\in W$), the continuous principal series $I(\rho)_{w'}:=\big(\Ind_{B^-(\Qp)}^{G(\Qp)}{w'}^{-1}(\chi_\rho)\cdot(\varepsilon^{-1}\circ\theta)\big)^{{\mathcal C}^0}$ appears ``at the same place'' in $\Pi(\rho)_{C_\rho,w}$, where $B^-\subseteq G$ is the Borel opposite to the dual of $\widehat B$, $\chi_\rho: B^-(\Qp)\twoheadrightarrow T(\Qp)\rightarrow E^\times$ corresponds to $\widehat\chi_{\rho}$ by the local correspondence for tori and $\varepsilon^{-1}\circ\theta$ is a certain twist. The minimality assumption on $\widehat B_{C_{\rho}}$ guarantees that $\Pi(\rho)^{\rm ord}$ only depends on the $\widehat G(E)$-conjugacy class of $\rho$.
This construction also works in characteristic $p$ for a generic $\rhobar:\gp\rightarrow \widehat B(k_E)$ (\S\ref{variant1}, \S\ref{variant2}) and produces an admissible smooth $G(\Qp)$-representation $\Pi(\rhobar)^{\rm ord}=\oplus_{w\in W_{C_{\rhobar}}}\Pi(\rhobar)_{C_{\rhobar},w}$, where each $\Pi(\rhobar)_{C_{\rhobar},w}$ is indecomposable and a successive extension of finitely many smooth principal series.

In section \ref{mainsection}, we prove that the representations $\Pi(\rhobar)_{C_{\rhobar},w}$ all occur in some spaces of mod $p$ automorphic forms, and, when $W_{C_\rho}=\{1\}$, that the representation $\Pi(\rho)^{\rm ord}$ occurs in some spaces of $p$-adic automorphic forms. More precisely, let $F^+$ be a totally real field and $G/F^+$ be a totally definite unitary group which is an outer form of $\GL_n/F^+$ that splits over a totally imaginary quadratic extension $F$ of $F^+$ (and is quasi-split at all finite places). We assume that $F/F^+$ is unramified at all finite places and that $p$ splits completely in $F$. If $U^p\subseteq \GAp$ is a compact open subgroup, we have the $k_E$-vector space
$$S(U^p,k_E):=\{f:G(F^+)\backslash \GA/U^p\rightarrow k_E,\ f\ {\rm locally\ constant}\}$$
of mod $p$ automorphic forms of level $U^p$ which is equipped with a smooth action of $G(F^+\otimes_{\Q}\Qp)\cong \prod_{v\vert p}\GL_n(\Qp)$ and with a commuting action of a certain Hecke algebra $\TT$. If $\rbar:\gF\rightarrow {\GL}_n(k_E)$ is a continuous absolutely irreducible representation, we can associate to $\rbar$ a maximal ideal $\m$ of $\TT$ with residue field $k_E$. We denote by $S(U^p,k_E)[\m]\subseteq S(U^p,k_E)$ the corresponding eigenspace and by $S(U^p,k_E)[\m]^{\ord}\subseteq S(U^p,k_E)[\m]$ the maximal $G(F^+\otimes_{\Q}\Qp)$-subrepresentation all of whose constituents are subquotients of principal series. Following \cite[\S6]{GG} we say that $\rbar$ is {\it modular and ordinary} if there exist $U^p$ (ramified only at places of $F^+$ that split in $F$) and an irreducible representation $\sigma$ of $G(\oFF\otimes_{\Z}\Zp)\cong \prod_{v\vert p}\GL_n(\Zp)$ over $k_E$ (a Serre weight) such that the action of the Hecke algebra ${\mathcal H}(\sigma)\cong \otimes_{v\vert p}k_E[T_{1,v},\dots,T_{n-1,v},T_{n,v}^{\pm 1}]$ on $\Hom_{G(\oFF\otimes_{\Z}\Zp)}\big(\sigma,S(U^p,k_E)[\m]\big)$ has a nonzero eigenvector (after possibly extending $k_E$) on which {\it each} $T_{i,v}$ is nonzero. Let us choose a place $\tilde v$ of $F$ above each $v\vert p$ in $F^+$ (this choice won't matter) and denote by $\rbar_{\tilde v}$ the restriction of $\rbar$ to a decomposition subgroup at $\tilde v$. If $\rbar$ is modular and ordinary, then $\rbar_{\tilde v}$ takes values in a Borel subgroup of ${\GL}_n(k_E)$ so that the $\GL_n(\Qp)$-representations $\Pi(\rbar_{\tilde v})^{\rm ord}$ are defined (if $\rbar_{\tilde v}$ is generic). Our main local-global compatibility result is (cf.\ Theorem \ref{ouf!}):

\begin{thm0}\label{thmintro}
Let $\rbar:\gF\rightarrow {\GL}_n(k_E)$ be continuous absolutely irreducible and assume:
\begin{enumerate}
\item $\rbar$ is modular and ordinary;
\item for each $v\vert p$ the restriction of $\rbar_{\tilde v}$ to inertia is generic;
\item $\rbar|_{{\Gal}(\overline F/F(\zeta_p ))}$ is absolutely irreducible, $p > 2n+2$ and $\zeta_p \not\in F$.
\end{enumerate}
Then there is an compact open subgroup $U^p\subseteq \GAp$ and integers $d_w>0$ for each $w=(w_v)\in \prod_{v\vert p}W_{C_{\rbar_{\tilde v}}}$ such that we have an essential injection of admissible smooth representations of $G(F^+\otimes_{\Q}\Qp)$ over $k_E$,
$$\bigoplus_{w=(w_v)} \bigg({\bigotimes_{v\vert p}}\Big(\Pi(\rbar_{\tilde v})_{C_{\rbar_{\tilde v}},w_v}\otimes \omega^{n-1}\circ\det\Big)\bigg)^{\oplus d_{w}}\hookrightarrow S(U^p,k_E)[\m]^{\ord},$$
where $\omega$ is the mod $p$ cyclotomic character.
\end{thm0}

We refer to \S\ref{theresult} for a more precise statement, and to Theorem \ref{ouf!2} for a $p$-adic version assuming that all $W_{C_{\rbar_{\tilde v}}}$ are trivial and replacing $S(U^p,k_E)$ by a suitable $p$-adically completed $E$-vector space of automorphic forms. If all $d_w$ are equal (e.g.\ if all $W_{C_{\rbar_{\tilde v}}}$ are trivial), note that the left-hand side of the injection in Theorem \ref{thmintro} is exactly (up to twist) a direct sum of copies of $\otimes_{v\vert p}\Pi(\rbar_{\tilde v})^{\rm ord}$. The assumption (ii) that $\rbar_{\tilde v}$ is generic on inertia (which implies e.g.\ $p>2n$) shouldn't be crucial and one should be able to replace it by just $\rbar_{\tilde v}$ generic (not necessarily on inertia) and $p>3$. Following our ``philosophy'' that $\Pi(\rhobar)^{\rm ord}$ should be the maximal subrepresentation of $\Pi(\rhobar)$ built out of principal series, we also conjecture (cf.\ Conjecture \ref{theconjbar}):

\begin{conj0}
Let $\rbar:\gF\rightarrow {\GL}_n(k_E)$ be continuous absolutely irreducible and assume:
\begin{enumerate}
\item $\rbar$ is modular;
\item $\rbar_{\tilde v}$ is upper triangular and generic for each $v\vert p$.
\end{enumerate}
Then for any compact open subgroup $U^p\subseteq \GAp$ such that $S(U^p,k_E)[\m]\ne 0$ there is an integer $d > 0$ depending only on $U^p$ and $\rbar$ such that we have an isomorphism of admissible smooth representations of $G(F^+\otimes_{\Q}\Qp)$ over $k_E$,
$$\bigg({\bigotimes_{v\vert p}}\Big(\Pi(\rbar_{\tilde v})^{\ord}\otimes \omega^{n-1}\circ\det\Big)\bigg)^{\oplus d}\congto S(U^p,k_E)[\m]^{\ord}.$$
\end{conj0}

We have an analogous conjecture in characteristic $0$ (cf.\ Conjecture \ref{theconj}).

The proof of Theorem \ref{thmintro} goes as follows: using results of Gee and Geraghty on ordinary Serre weights (\cite{GG}, \cite{Ge}, these results require the small technical assumptions alluded to in (iii)) together with Frobenius reciprocity, we first deduce that for each $w=(w_v)\in \prod_{v\vert p}W_{C_{\rbar_{\tilde v}}}$ the $G(F^+\otimes_{\Q}\Qp)$-socle $\otimes_{v\vert p}I(\rbar_{\tilde v})_{w_v}$ of ${\otimes_{v\vert p}}\big(\Pi(\rbar_{\tilde v})_{C_{\rbar_{\tilde v}},w}\otimes \omega^{n-1}\circ\det\big)$ occurs in some $S(U^p,k_E)[\m]$ with a certain positive multiplicity $d_w$. We thus have an injection as in Theorem \ref{thmintro} but with $\Pi(\rbar_{\tilde v})_{C_{\rbar_{\tilde v}},w_v}\otimes \omega^{n-1}\circ\det$ replaced by its socle $I(\rbar_{\tilde v})_{w_v}$. Moreover, combining the results of \cite{GG} with \cite[Cor.\ 9.13]{He2}, one easily checks that this injection is essential. The whole point is thus to prove that it extends to ${\otimes_{v\vert p}}\big(\Pi(\rbar_{\tilde v})_{C_{\rbar_{\tilde v}},w}\otimes \omega^{n-1}\circ\det\big)$. This follows from a key local theorem (cf.\ Theorem \ref{main}) which states that, under certain conditions on a smooth $G(\Qp)$-representation $\Pi$ over $k_E$, restriction to the $G(\Qp)$-socle induces an isomorphism $\Hom_{G(\Qp)}(\Pi(\rhobar)_{C_{\rhobar},w}\otimes \omega^{n-1}\circ\det,\Pi)\buildrel\sim\over\rightarrow \Hom_{G(\Qp)}(I(\rhobar)_w,\Pi)$. Here, we use an important improvement due to Pa{\v{s}}k{\=u}nas in the proof of this local theorem.

Since the first version of this article, several results have been proven that confirm some of our speculations and questions. In \cite{Br7} the first-named author constructed a functor from smooth representations of $G(\Qp)$ over $\oE/\pE^m$ to \'etale pro-$(\varphi,\Gamma)$-modules over $\oE/\pE^m$ for any $m>0$ which satisfies the mod $p$ analogue of properties (i)--(iii) in \S\ref{ques} (see \cite[Cor.\ 9.5]{Br7}). Hauseux \cite{Ha1}, \cite{Ha2} gave a positive answer to our question on extensions between principal series representations (see Theorem \ref{hau}) and proved our Conjecture \ref{conjloc} on the unicity of $\Pi(\rho)^{\rm ord}$ (assuming the irreducibility of the principal series (\ref{const})).

We now finish this introduction with some notation.

In the whole text $E$ is a finite extension of $\Qp$ (the coefficient field) with ring of integers $\oE$ and residue field $k_E$. We denote by $\pE$ a uniformizer in $\oE$. We denote by $\varepsilon:\gp\rightarrow \Zp^{\times}\hookrightarrow E^{\times}$ the $p$-adic cyclotomic character and by $\omega$ its reduction mod $p$. If $F$ is a number field and $w$ a finite place of $F$, $\Frob_w$ is a geometric Frobenius at $w$. We normalize the reciprocity map of local class field theory so that uniformizers correspond to geometric Frobenius elements. We denote by $\nr(u)$ either the unramified character of $\Qp^{\times}$ sending $p$ to $u$ or the unramified character of $\gp$ sending a geometric Frobenius element to $u$.

If $H$ is a $p$-adic Lie group, e.g.\ $H=G(\Qp)$ where $G$ is an algebraic group over $\Qp$, we call a \emph{continuous representation} of $H$ over $E$ any $p$-adic Banach space $\Pi$ over $E$ endowed with an $E$-linear action of $H$ such that the action map $H\times \Pi\rightarrow \Pi$ is continuous. If $G$ is a connected reductive algebraic group over $\Zp$, we call a \emph{Serre weight} for $G(\Fp)$ any irreducible representation of $G(\Fp)$ (or equivalently any irreducible smooth representation of $G(\Zp)$) over $k_E$. A Serre weight is in fact absolutely irreducible and defined over $\Fp$ when $G$ is split. The other notation will be introduced in the body of the text.

We would like to thank L.\ Clozel, M.\ Emerton, W.\ T.\ Gan, G.\ Henniart, S.\ Morra, S.\ W.\ Shin, and P.-J.\ White for discussions related to this work. We would also like to thank especially T.\ Gee and D.\ Geraghty for patiently answering our questions. The second author would like to thank the University of Paris 11, where some of this work was carried out. Finally, we are deeply indebted to V.\ Pa{\v{s}}k{\=u}nas for finding a much better proof of our main local result (Theorem \ref{main}) which easily extends to characteristic $0$ (Corollary \ref{mainpadic}) and for allowing us to reproduce it here (our initial proof in the first version of that paper was long, laborious and less general).

\section{The algebraic representation $\LL$}\label{LL}

We define the algebraic representation $\LL$ and study its restriction to certain subgroups of the Borel subgroup.

\subsection{Definition of the representation $\LL$}\label{fundamental}

We define the algebraic representation $\LL$.

Let $H/E$ be a split connected reductive algebraic group. Let $T\subseteq H$ be a split maximal torus (over $E$), $X(T)=\Hom_{\rm alg}(T,{\mathbb G}_m)$ the group of characters of $T$ and $X^{\vee}(T)=\Hom_{\rm alg}({\mathbb G}_m,T)$ its group of cocharacters. We let $\big(X(T),R,X^{\vee}(T),R^{\vee}\big)$ be the root datum of $H$, where $R\subseteq X(T)$ (resp.\ $R^{\vee}\subseteq X^{\vee}(T)$) is the set of roots (resp.\ coroots). For $\alpha\in R$, we let $s_{\alpha}$ be the reflection on $X(T)$ associated to $\alpha$ and recall that $s_{\alpha}(\lambda)=\lambda-\langle \lambda,\alpha^{\vee}\rangle\alpha$ for $\lambda\in X(T)$. We let $W$ be the Weyl group, that is, the subgroup of automorphisms of $X(T)$ generated by the $s_{\alpha}$ for $\alpha\in R$.

We fix a choice of simple roots $S\subseteq R$ and denote by $R^+\subseteq R$ the positive roots, i.e.\ the roots that are in $\oplus_{\alpha\in S}\Z_{\geq 0}\alpha$. We finally let $H^{\rm der}$ be the derived algebraic subgroup of $H$ and $\widehat{H}$ the dual algebraic group of $H$. We recall the following standard proposition (see e.g.\ the proof of \cite[Prop.\ 5.23]{DL} where the fact that the base field is algebraically closed of char.\ $p$ is here irrelevant).

\begin{prop}\label{sc}
The following conditions are equivalent:
\begin{enumerate}
\item the dual group $\widehat{H}$ has connected centre;
\item the derived subgroup $H^{\rm der}$ is \(semi-simple\) simply connected;
\item there exists $(\lambda_{\alpha})_{\alpha\in S}\in X(T)^{|S|}$ such that for any $\beta\in S$: $$\langle\lambda_{\alpha},{\beta}^{\vee}\rangle=\left\{\begin{array}{ccc}1&\ {\rm if}\ &\alpha=\beta\\ 0&\ {\rm if}\ &\alpha \ne \beta. \end{array}\right.$$
\end{enumerate}
\end{prop}

We assume from now on that $H$ satisfies the equivalent conditions of Proposition \ref{sc}. We recall that $\lambda\in X(T)$ is said to be dominant if $\langle \lambda,\alpha^{\vee}\rangle\geq 0$ for all $\alpha\in R^+$. If $\lambda\in X(T)$ is a dominant weight, we denote by $L(\lambda)$ the irreducible algebraic representation of $H$ over $E$ with highest weight $\lambda$. Let $X^0(T):=\{\lambda\in X(T) : \langle\lambda,\beta^{\vee}\rangle=0\ \forall\beta^{\vee}\in R^{\vee}\}$. Then $L(\lambda)$ has dimension $1$ if and only if $\lambda\in X^0(T)$. The weights $\lambda_{\alpha}$ of Proposition \ref{sc} for $\alpha\in S$ are clearly dominant and are called {\it fundamental} weights. They are uniquely defined in $X(T)/X^0(T)$, that is, one can obviously replace $\lambda_{\alpha}$ by $\lambda_{\alpha}+\lambda_0$ for any $\lambda_0\in \Hom_{\rm alg}(H,{\mathbb G}_{\rm m})\cong X^0(T)$. However this only changes $L(\lambda_{\alpha})$ by a twist by a $1$-dimensional $L(\lambda_0)$.

\begin{definit}\label{fund}
Assume $\widehat{H}$ has connected centre or equivalently $H^{\rm der}$ is simply connected. The \emph{fundamental algebraic representations} of $H$ are the irreducible representations $(L(\lambda_{\alpha}))_{\alpha\in S}$ (defined up to twist).
\end{definit}

\begin{ex}\label{glgsp}
We briefly give the examples of $H={\GL}_{n}$ and $H={\rm GSp}_{2n}$ ($n\geq 2$).\\
When $H={\GL}_{n}$, $T$ is the torus of diagonal matrices and $B$ the upper triangular matrices, we have $X(T)=\Z e_1\oplus \cdots \oplus \Z e_n$, where $e_i$ is the character sending $\diag(x_1,\ldots, x_n)$ to $x_i$ and $S=\{e_i-e_{i+1} : 1\leq i\leq n-1\}$. Then $\lambda_{e_i-e_{i+1}}=e_1+e_2+\cdots +e_i$ (up to an element of $X^0(T)=\Z(e_1+\cdots +e_n)$) and we have
$$L(\lambda_{e_i-e_{i+1}})=\Lambda_E^i({\rm Std})\ \ 1\leq i\leq n-1\ \ \ {\rm (up\ to\ a\ twist)},$$
where Std is the standard $n$-dimensional algebraic representation of $\GL_n$ over $E$.\\
When $H={\rm GSp}_{2n}=\{A\in \GL_{2n} : \tau(A)\smat{0& J_n\\-J_n& 0}A=\nu \smat{0& J_n\\-J_n& 0}\ {\rm for\ some\ }\nu\in {\mathbb G}_{\rm m}\}$ ($\tau$ is the transpose in ${\GL_n}$, $J_n$ is the anti-diagonal ``identity'' matrix of size $n$), $T$ is the torus of diagonal matrices in ${\rm GSp}_{2n}$ and $B$ the upper triangular matrices in ${\rm GSp}_{2n}$, we have $X(T)=\Z e_1\oplus \cdots \oplus \Z e_n\oplus \Z e$, where $e_i$ (resp. $e$) is the character sending $\diag(x_1,\dots,x_n,\nu x_n^{-1},\dots,\nu x_1^{-1})$ to $x_i$ (resp. $\nu$) and $S=\{e_i-e_{i+1} : 1\leq i\leq n-1,2e_n-e\}$. Then $\lambda_{e_i-e_{i+1}}=e_1+\cdots +e_i$, $1\leq i\leq n-1$ and $\lambda_{2e_{n}-e}=e_1+\cdots +e_n$ (up to an element of $X^0(T)=\Z e$). If Std is the standard $2n$-dimensional algebraic representation of ${\rm GSp}_{2n}$ over $E$, there is for $2\leq i\leq n$ a surjection $\psi^i:\Lambda_E^i({\rm Std})\twoheadrightarrow \Lambda_E^{i-2}({\rm Std})(\nu)$ and we have $L(\lambda_{e_1-e_{2}})={\rm Std}$, $L(\lambda_{2e_n-e})=\Ker(\psi^n)$ and
$$L(\lambda_{e_i-e_{i+1}})=\Ker(\psi^i),\ \ 2\leq i\leq n-1\ \ \ {\rm (all\ up\ to\ a\ twist)}$$
(we refer to \cite[\S 17.2]{FH} for more details).
\end{ex}

We define the following algebraic representation of $H$ over $E$:
$$\LL:=\bigotimes_{\alpha\in S}L(\lambda_{\alpha}).$$
It is a reducible algebraic representation in general. The representation $\LL$ depends on the choice of weights $\lambda_{\alpha}$ as in Proposition \ref{sc}(iii), but any change in this choice only modifies $\LL$ by a twist, which won't affect the results of this paper. We just keep in mind in the sequel that $\LL$ is canonical only up to twist.

\begin{prop}\label{dual}
There is an algebraic character $\lambda_0\in \Hom_{\rm alg}(H,{\mathbb G}_{\rm m}) \cong X^0(T)$ such that $(\LL)^{\vee}=\LL\otimes L(\lambda_0)$.
\end{prop}
\begin{proof}
Let $w_0$ be the longest element of $W$. Then $-w_0$ permutes the elements of $S$, and thus we have from Definition \ref{fund}
$$-w_0(\lambda_{\alpha})-\lambda_{-w_0(\alpha)} \in X^0(T){\rm\ for\ any\ }\alpha\in S.$$
Since $L(\lambda_{\alpha})^{\vee}=L(-w_0\lambda_{\alpha})$, we get
$$L(\lambda_{\alpha})^{\vee}=L(\lambda_{-w_0(\alpha)})\otimes L\big(-w_0(\lambda_{\alpha})-\lambda_{-w_0(\alpha)}\big).$$
Setting $\lambda_0:=\sum_{\alpha\in S}(-w_0(\lambda_{\alpha})-\lambda_{-w_0(\alpha)}) \in X^0(T)$, we deduce $(\LL)^{\vee}=\LL\otimes L(\lambda_0)$.
\end{proof}

\begin{rem}
Any change in the choice of the $\lambda_{\alpha}$ that doesn't affect $\sum_{\alpha\in S}\lambda_{\alpha}$ doesn't change $\LL$ either (as is immediately checked).
\end{rem}

Note that, if $H=H_1\times H_2$, then one has $\LL=\LL_1\otimes_E\LL_2$ (where we index by $i$ everything related to $H_i$, $i=1,2$).

\subsection{Multiplicity one weights of $\LL$}\label{one}

We determine the weights of $\LL\vert_T$ which occur with multiplicity $1$. We keep the notation of \S\ref{fundamental}.

We start with some lemmas.

\begin{lem}\label{petit}
Let $\alpha\in R^+$ and $\beta\in S$. If $s_{\alpha}(\lambda_{\beta'})=\lambda_{\beta'}$ for all $\beta'\in S$ except possibly $\beta$, then $s_{\alpha}(\lambda_{\beta})=\lambda_{\beta}- \alpha$.
\end{lem}
\begin{proof}
We have
$$\langle \lambda_{\beta}-s_{\alpha}(\lambda_{\beta}),\alpha^{\vee}\rangle=2\langle\lambda_{\beta},\alpha^{\vee}\rangle,$$
where $\langle\lambda_{\beta},\alpha^{\vee}\rangle$ is the coordinate of $\beta^{\vee}$ in $\alpha^{\vee}$ (using the properties of the fundamental weight $\lambda_{\beta}$, see Proposition \ref{sc}). Since $s_{\alpha}(\lambda_{\beta'})=\lambda_{\beta'}$ for all $\beta'\in S\backslash\{\beta\}$, this coordinate is $0$ for $\beta'\ne\beta$ and we thus get $\alpha^{\vee}\in \Z\beta^{\vee}$. Since $\Z\beta^{\vee}\cap R^{\vee}=\pm\beta^{\vee}$ and $\alpha$ is a positive root, we have $\alpha=\beta$ which finishes the proof.
\end{proof}

For $\lambda,\mu\in X(T)$, we write $\mu\leq \lambda$ (resp.\ $\mu\geq\lambda$) if $\lambda-\mu\in \oplus_{\alpha\in S}\Z_{\geq 0}\alpha$ (resp.\ $\lambda-\mu\in \oplus_{\alpha\in S}\Z_{\leq 0}\alpha$). We write $\mu< \lambda$ (resp.\ $\mu>\lambda$) if $\mu\leq \lambda$ and $\lambda\ne \mu$ (resp.\ $\mu\geq\lambda$ and $\lambda\ne\mu$). Recall that, if $\lambda$ is dominant, then $w(\lambda)\leq \lambda$ for all $w\in W$.

\begin{lem}\label{plus}
Let $\lambda\in X(T)$ be a dominant weight and $\mu$ be a weight that appears in $L(\lambda)|_T$. Then there exist an integer $s\geq 1$, a sequence of weights $\mu_1,\dots,\mu_s$ in $L(\lambda)|_T$ and a sequence of positive roots $\alpha_1,\dots,\alpha_{s-1}$ satisfying the following conditions:
\begin{enumerate}
\item $\mu_1=\lambda$ and $\mu_s=\mu$;
\item $\mu_{i+1}\in \mu_i+\Z \alpha_i$ and $s_{\alpha_i}(\mu_i)\leq \mu_{i+1}< \mu_i$ for all $i\in \{1,\dots,s-1\}$.
\end{enumerate}
\end{lem}
\begin{proof}
Since (ii) is empty when $s=1$, we can assume $\mu<\lambda$. We use the following result: there exists $\alpha\in R^+$ such that $\mu+\alpha\leq \lambda$ and $\mu+\alpha$ is still a weight of $L(\lambda)|_T$. Indeed, if $\mu$ is not dominant, then we can take any $\alpha\in R^+$ such that $\langle \mu,\alpha^{\vee}\rangle<0$ (considering the restriction of $\LL$ to the subgroup ${\rm SL}_2$ corresponding to the root $\alpha$, we see by \cite[Prop.\ 21.3]{Hu} that the weight $\mu+\alpha$ appears in $L(\lambda)|_T$ being between $\mu$ and $s_{\alpha}(\mu)=\mu-\langle \mu,\alpha^{\vee}\rangle\alpha$). If $\mu$ is dominant, then by a result of Stembridge (see \cite[Lem.\ 2.3]{Ra}) there exists $\alpha\in R^+$ such that $\mu+\alpha\leq \lambda$ and $\mu+\alpha$ is still a dominant weight, hence $\mu+\alpha$ still appears in $L(\lambda)|_T$ (use that $w(\mu+\alpha)\leq \mu+\alpha\ \forall\ w\in W$ together with \cite[Prop.\ 21.3]{Hu}). Choosing such an $\alpha$, the set of weights $\mu+i\alpha$ for $i\in\Z$ that appear in $L(\lambda)|_T$ is of the form $\{\mu-i_1\alpha,\mu-(i_1-1)\alpha,\dots,\mu+i_2\alpha\}$ for some $i_1\in \Z_{\geq 0}$ and some $i_2\in \Z_{\geq 1}$, and we have $\mu-i_1\alpha=s_{\alpha}(\mu+i_2\alpha)\leq \mu<\mu+i_2\alpha$ (see again \cite[\S 21.3]{Hu}). If $\mu'_2:=\mu+i_2\alpha<\lambda$, we start again with another positive root $\beta$ such that $\mu'_2+\beta\leq\lambda$ (and $\mu'_2+\beta$ appears in $L(\lambda)|_T$) and get a weight $\mu'_3$ such that $s_{\alpha}(\mu'_3)\leq \mu'_2<\mu'_3$. Since all weights in $L(\lambda)|_T$ are bounded by $\lambda$, we necessarily reach $\lambda$ after a finite number of iterations, say $s-1$. We finally set $\mu_s:=\mu$ and $\mu_i:=\mu'_{s+1-i}$ for $i\in \{1,\dots,s-1\}$.
\end{proof}

\begin{lem}\label{cache}
Let $n\in \Z_{\geq 1}$ and $\lambda_1,\dots,\lambda_n\in X(T)$ be dominant. The weights that appear in $\otimes_{i=1}^nL(\lambda_i)|_T$ are the weights that appear in $L(\sum_{i=1}^n\lambda_i)|_T$ \(up to multiplicity\).
\end{lem}
\begin{proof}
Since $L(\sum_{i=1}^n\lambda_i)$ is a direct summand of $\otimes_{i=1}^nL(\lambda_i)$, the weights that appear in $L(\sum_{i=1}^n\lambda_i)|_T$ clearly all appear in $\otimes_{i=1}^nL(\lambda_i)|_T$. Let us prove the converse. Let $\mu_1+\cdots +\mu_n$ be a weight of $\otimes_{i=1}^nL(\lambda_i)|_T$, where $\mu_i$ is a weight of $L(\lambda_i)|_T$. We use the following well-known result: if $\lambda\in X(T)$ is a dominant weight, then the weights that appear in $L(\lambda)|_T$ (up to multiplicity) are exactly all the weights $\mu\in X(T)$ satisfying the following condition:
\begin{equation}\label{convexe}
\lambda-w(\mu)\in \oplus_{\alpha\in S}\Z_{\geq 0}\alpha\ \ \forall w\in W.
\end{equation}
For $w\in W$, we have $\lambda_i-w(\mu_i)\in \oplus_{\alpha\in S}\Z_{\geq 0}\alpha$ by applying (\ref{convexe}) to $L(\lambda_i)$. By summing, we get
$$\sum_{i=1}^n\lambda_i-w\Big(\sum_{i=1}^n\mu_i\Big)\in \oplus_{\alpha\in S}\Z_{\geq 0}\alpha \ \forall w\in W,$$
which implies that $\sum_{i=1}^n\mu_i$ is a weight of $L(\sum_{i=1}^n\lambda_i)|_T$.
\end{proof}

We now state the main theorem of this section.

\begin{thm}\label{mult}
The weights of $\LL\vert_T$ that occur with multiplicity $1$ are exactly the weights $\{w\big(\sum_{\alpha\in S}\lambda_{\alpha}\big) : w\in W\}$.
\end{thm}
\begin{proof}
Let $\lambda:=\sum_{\beta\in S}\lambda_{\beta}$. Since $\lambda$ occurs with multiplicity $1$, the same holds for all the weights in its $W$-orbit. Let $\mu\in X(T)$ be a weight of $\LL\vert_T$, we have to prove that either $\mu$ occurs with multiplicity $>1$ or $\mu$ is in the $W$-orbit of $\lambda$. By Lemma \ref{cache}, $\mu$ occurs in $L(\lambda)|_T$. We will make an induction on the smallest integer $s(\mu)\geq 1$ such that there are weights $\mu_1,\dots,\mu_{s(\mu)}$ in $L(\lambda)|_T$ and positive roots $\alpha_1,\dots,\alpha_{s(\mu)-1}$ as in Lemma \ref{plus}. If $s(\mu)=1$, then $\mu=\mu_1=\lambda$ and the statement is obvious. Let us assume $s(\mu)>1$ (i.e.\ $\mu<\lambda$) and that the statement is true for all weights $\mu'$ of $\LL\vert_T$ with $1\leq s(\mu')<s(\mu)$. By the induction hypothesis applied to $\mu_{s(\mu)-1}$, we have that either $\mu_{s(\mu)-1}$ occurs with multiplicity $>1$ or $\mu_{s(\mu)-1}$ is in the $W$-orbit of $\lambda$. In the first case, we immediately deduce that $s_{\alpha_{s(\mu)-1}}(\mu_{s(\mu)-1})$ occurs with multiplicity $>1$, as well as all the weights of the form $\mu_{s(\mu)-1}+\Z \alpha_{s(\mu)-1}$ between $s_{\alpha_{s(\mu)-1}}(\mu_{s(\mu)-1})$ and $\mu_{s(\mu)-1}$ (use \cite[Prop.\ 21.3]{Hu} as in the proof of Lemma \ref{plus}). So in particular $\mu$ occurs with multiplicity $>1$ in $\LL\vert_T$. In the second case, let $w\in W$ be such that $\mu_{s(\mu)-1}=w(\lambda)$. Applying $w^{-1}$ (which doesn't change the multiplicities), we can assume $\mu_{s(\mu)-1}=\lambda$ and $s_{\alpha}(\lambda)\leq \mu< \lambda$ with $\mu\in \lambda+\Z\alpha$, where $\alpha:=w^{-1}(\alpha_{s(\mu)-1})$ (which is still a positive root). For $\beta\in S$, let $n_{\beta}:=\langle\lambda_{\beta},\alpha^{\vee}\rangle\in \Z_{\geq 0}$. All the weights $\lambda_{\beta}-i\alpha$ for $0\leq i\leq n_{\beta}$ appear in $L(\lambda_{\beta})|_T$. Let $i\in \{1,\dots,\sum_{\beta\in S}n_{\beta}\}$ such that $\mu=\lambda-i\alpha$. If $i=\sum_{\beta\in S}n_{\beta}$, then $\mu=s_{\alpha}(\lambda)$ is in the $W$-orbit of $\lambda$ and we are done. Let us assume $1\leq i\leq (\sum_{\beta\in S}n_{\beta})-1$. We can write $i=\sum_{\beta\in S}i_{\beta}$ for some $i_{\beta}\in\{0,\dots,n_{\beta}\}$. Since $\mu>s_{\alpha}(\lambda)$, there exists $\beta_1\in S$ such that $n_{\beta_1}\ne 0$ and $0\leq i_{\beta_1}\leq n_{\beta_1}-1$. Since $\mu<\lambda$, there exists $\beta_2\in S$ such that $n_{\beta_2}\ne 0$ and $1\leq i_{\beta_2}\leq n_{\beta_2}$. If we can find such $\beta_1,\beta_2$ which are distinct, then $\mu$ appears at least twice in $L(\lambda)|_T$ since we can write
\begin{eqnarray*}
\mu&=&\sum_{\beta\in S}(\lambda_{\beta}-i_{\beta}\alpha)\\
&=&\!\!\!\bigg(\sum_{\beta\in S\backslash\{\beta_1,\beta_2\}}\!\!\!\!\!(\lambda_{\beta}-i_{\beta}\alpha)\bigg)+\big(\lambda_{\beta_1}-(i_{\beta_1}\!+\!1)\alpha\big)+\big(\lambda_{\beta_2}-(i_{\beta_2}\!-\!1)\alpha\big).
\end{eqnarray*}
If we can't, then we have $1\leq i_{\beta_1}\leq n_{\beta_1}-1$ and $i_{\beta}=0$ for all $\beta\ne \beta_1$. By Lemma \ref{petit}, since $n_{\beta_1}\geq 2$ there exists $\beta_2\in S$, $\beta_2\ne\beta_1$ such that $n_{\beta_2}\geq 1$. Then $\mu$ again appears at least twice in $L(\lambda)|_T$ since we can write
\begin{eqnarray*}
\mu&=&\sum_{\beta\in S\backslash\{\beta_1\}}\!\!\!\!\!\lambda_{\beta}+(\lambda_{\beta_1}-i_{\beta_1}\alpha)\\
&=&\!\!\!\sum_{\beta\in S\backslash\{\beta_1,\beta_2\}}\!\!\!\!\!\!\!\lambda_{\beta}+\big(\lambda_{\beta_1}-(i_{\beta_1}-1)\alpha\big)+\big(\lambda_{\beta_2}-\alpha\big).
\end{eqnarray*}
\end{proof}

\begin{rem}
By using Weyl's character formula, one can actually prove the stronger result that the weights of $L(\lambda)\vert_T$ ($\lambda=\sum_{\beta\in S}\lambda_{\beta}$) that occur with multiplicity $1$ are only the weights in the $W$-orbit of $\lambda$ (it is indeed stronger because of Lemma \ref{cache}). Since any weight of $L(\lambda)|_T$ is in the $W$-orbit of some dominant weight and since the multiplicity is constant on a $W$-orbit, we are reduced to proving that if a weight $\mu$ of $L(\lambda)|_T$ is dominant and distinct from $\lambda$, then it occurs with multiplicity $>1$. Arguing as in the proof of lemma \ref{plus} and using again the fact that, if $\mu_i$ occurs with multiplicity $>1$ in $L(\lambda)|_T$, then the same holds for all the weights between $s_{\alpha_i}(\mu_i)$ and $\mu_i$, we can easily reduce to the case $\mu$ is as big as possible for the order relation $\leq $. The same use of Stembridge's lemma as in {\it loc.\ cit.} then yields $\mu=\lambda-\beta$ for some $\beta\in R^+$. If $\beta\in S$, we easily check that $\lambda-\beta=s_{\beta}(\lambda)$ and thus $\lambda-\beta$ is not dominant, hence $\beta\in R^+\backslash S$. We now use Weyl's character formula for $L(\lambda)$ which states that the multiplicity of a weight $\nu$ in $L(\lambda)|_T$ is the coefficient of $e(\nu-(\lambda-\rho))$ in the ratio (see e.g.\ \cite[\S II.5]{Ja})
$$\frac{\prod_{\alpha\in R^+}(e(\alpha)-e(-\alpha))}{\prod_{\alpha\in R^+}(e(\frac{\alpha}{2})-e(-\frac{\alpha}{2}))}=\prod_{\alpha\in R^+}\Big(e\big(\frac{\alpha}{2}\big)+e\big(-\frac{\alpha}{2}\big)\Big)=e(\rho)\prod_{\alpha\in R^+}(1+e(-\alpha)),$$
where $\rho:=\frac{1}{2}\sum_{\alpha\in R^+}\alpha$ (we have $\lambda-\rho\in X^0(T)^W\otimes_{\Z}\Q$). Since $\beta$ is not simple, there exist $\gamma\in R^+$ and $\delta\in S$ such that $\beta=\gamma+\delta$ and we see that the coefficient of $e(\lambda-\beta-(\lambda-\rho))=e(\rho-\beta)$ is already $2$ in the factor $e(\rho)(1+e(-\beta))(1+e(-\gamma))(1+e(-\delta))$.
\end{rem}

\begin{definit}
An {\it ordinary} weight of $\LL$ is a weight $w\big(\sum_{\alpha\in S}\lambda_{\alpha}\big)$ as in Theorem \ref{mult}.
\end{definit}

\begin{rem}
In the sequel, we will actually only need the easy part of Theorem \ref{mult}, that is, any ordinary weight of $\LL$ occurs with multiplicity $1$.
\end{rem}

\subsection{On the restriction of $\LL$ to various subgroups I}\label{C}

For certain subgroups $B_C$ of the Borel subgroup $B$ we define $B_C$-representations $(\LL\vert_{B_C})^{\ord}$ and $L_{C,w_C}$. We keep the notation of \S\ref{fundamental} and \S\ref{one}.

Let $B=TU\subseteq H$ be the Borel subgroup (over $E$) containing $T$ which corresponds to our choice of positive roots $R^+$, where $U\subseteq B$ is the unipotent radical of $B$. If $\alpha\in R$, recall that one has an associated root subgroup $U_{\alpha}\subseteq H$ such that $\alpha$ is the only root of $U_{\alpha}$ (see e.g.\ \cite[\S II.1.2]{Ja}). We have $\alpha\in R^+$ if and only if $U_{\alpha}\subseteq B$. A subset $C\subseteq R$ is said to be {\it closed} if the following condition is satisfied: if $\alpha\in C$, $\beta\in C$ and $\alpha + \beta\in R$ then $\alpha+\beta\in C$. If $C\subseteq R^+$ is a closed subset, we let $U_C\subseteq U$ be the Zariski closed subgroup of $B$ generated by the root subgroups $U_{\alpha}$ for $\alpha\in C$. For instance one has $U_{\varnothing}=\{1\}$ and $U_{R^+}=U$. The roots of the algebraic group $U_C$ are exactly the roots in $C$ (\cite[\S II.1.7]{Ja}). We let $B_C:=TU_C\subseteq B$, it is a (Zariski) closed subgroup of $B$.

\begin{lem}\label{borel}
Let $B'\subseteq B$ be a Zariski closed algebraic subgroup containing $T$. Then there exists a closed subset $C\subseteq R^+$ such that $B'=B_C$.
\end{lem}
\begin{proof}
Since $T\subseteq B'$, we have $B'=TU'$, where $U':=B'\cap U$ is stable under conjugation by $T$ (or $T$-stable in the sense of \cite{Bo}). By \cite[Prop.\ 14.4(2)]{Bo}, we have ${\rm Lie}(U')=\oplus_{\alpha\in C}{\rm Lie}(U_{\alpha})$ for some subset $C\subseteq R^+$. Now let $\alpha,\beta\in C$ such that $\alpha+\beta\in R$. Since $U'$ is closed, it contains the closure $U''$ of the commutator group $[U_{\alpha},U_{\beta}]$. By \cite[Prop.\ 3.17]{Bo}, ${\rm Lie}(U'')$ contains $[{\rm Lie}(U_{\alpha}),{\rm Lie}(U_{\beta})]$. Since $[{\rm Lie}(U_{\alpha}),{\rm Lie}(U_{\beta})]={\rm Lie}(U_{\alpha+\beta})$ (\cite[Rk.\ 14.5(2)]{Bo}, we use here that $E$ has characteristic $0$), we finally get ${\rm Lie}(U_{\alpha+\beta})\subseteq {\rm Lie}(U')$ and hence $\alpha+\beta\in C$. This shows that $C$ is closed.
\end{proof}

\begin{lem}\label{commute}
Let $C\subseteq R^+$ be a closed subset and $I\subseteq C$ be a subset satisfying the following conditions:
\begin{enumerate}
\item a root in $I$ is never the sum of more than one root in $C$;
\item for any distinct $\alpha,\beta\in I$, one has $(\Z_{\geq 0}\alpha\oplus \Z_{\geq 0}\beta)\cap R = \{\alpha,\beta\}$.
\end{enumerate}
Then $I$ and $C\backslash I$ are closed subsets of $R^+$, $U_{C\backslash I}$ is a normal subgroup of $U_C$, $U_I$ is a commutative subgroup of $U_C$ \(isomorphic to $\prod_{\alpha\in I}U_{\alpha}$\) and one has a semi-direct product $U_C=U_{C\backslash I}\rtimes U_I$.
\end{lem}
\begin{proof}
Condition (ii) obviously implies that $I$ is closed whereas condition (i) together with $C$ being closed imply $C\backslash I$ is also closed. Thus $U_I$ and $U_{C\backslash I}$ are well-defined subgroups of $U_C$ and one has an isomorphism of varieties $U_IU_{C\backslash I}\buildrel\sim\over\rightarrow U_C$. Let $\alpha\in I$, $u\in U_{\alpha}$ and $\beta\in C\backslash I$, then the commutation formula \cite[\S II.1.2(5)]{Ja} together with condition (i) imply $uU_{\beta}u^{-1}\subseteq U_{C\backslash I}$. This implies that $U_{C\backslash I}$ is normal in $U_C$ and hence that $U_C=U_{C\backslash I}\rtimes U_I$. Finally, if $\alpha,\beta$ are distinct roots in $I$, by (ii) there exist no positive integers $i,j$ such that $i\alpha+j\beta\in R$ and by \cite[\S II.1.2(5)]{Ja} again we see that $U_{\alpha}$ and $U_{\beta}$ must commute with each other in $U_C$. This finishes the proof.
\end{proof}

If \ $L_1,L_2\subseteq \LL\vert_{B_C}$ \ are \ two \ $B_C$-subrepresentations \ such \ that \ all \ their weights are ordinary, it is clear that the same holds for $L_1+L_2 \subseteq \LL\vert_{B_C}$.

\begin{definit}\label{Cord}
We let
$$(\LL\vert_{B_C})^{\ord}\subseteq \LL\vert_{B_C}$$
be the maximal $B_C$-subrepresentation of $\LL\vert_{B_C}$ such that all its weights are ordinary. We say $(\LL\vert_{B_C})^{\ord}$ is the \emph{ordinary part} of the algebraic $B_C$-represen\-tation $\LL\vert_{B_C}$.
\end{definit}

By (the easy direction in) Theorem \ref{plus}, the representation $(\LL\vert_{B_C})^{\ord}$ is multiplicity free.

\begin{ex}
For $C=\varnothing$, one obviously has $(\LL\vert_{T})^{\ord}=\oplus_{w\in W}w\big(\sum_{\alpha\in S}\lambda_{\alpha}\big)$.
\end{ex}

\begin{rem}\label{rema}
(i) Since the ordinary weights occur with multiplicity $1$ in $\LL$ and also occur in its direct summand $L(\sum_{\alpha\in S}\lambda_{\alpha})$ (Lemma \ref{cache}), we see that $(\LL\vert_{B_C})^{\ord}\subseteq L(\sum_{\alpha\in S}\lambda_{\alpha})\vert_{B_C}$ for all $C$.\\
(ii) If $H=H_1\times H_2$ and $C=C_1\amalg C_2\subseteq R^+=R_1^+\amalg R_2^+$ then one easily checks that $(\LL\vert_{B_C})^{\ord}=(\LL_1\vert_{B_{C_1}})^{\ord}\otimes_E(\LL_2\vert_{B_{C_2}})^{\ord}$ (where we index by $i$ everything related to $H_i$, $i=1,2$).\\
(iii) One could also define the maximal quotient $(\LL\vert_{B_C})_{\ord}$ of $\LL\vert_{B_C}$ such that all its weights are ordinary. From Proposition \ref{dual}, it is easy to deduce an isomorphism $(\LL\vert_{B_C})_{\ord}\cong ((\LL\vert_{B_C})^{\ord})^{\vee}\otimes L(-\lambda_0)$ for $\lambda_0$ as in that proposition.
\end{rem}

We now define certain $B_C$-representations $L_{C,w_C}$. We will show in the next section that they appear in $(\LL\vert_{B_C})^{\ord}$.

We first define the following subset of $W$:
\begin{equation}\label{wc}
W_C:=\{w\in W : w^{-1}(C)\subseteq R^+\}.
\end{equation}
For instance one has $W_{\varnothing}=W$ and $W_{R^+}=\{1\}$. One immediately checks that $w^{-1}(C)$ is again a closed subset of $R^+$ for $w\in W_C$. Let $N_H(T)$ be the normalizer of $T$ in $H$ and recall that the algebraic group $N_H(T)/T$ is a finite group isomorphic to $W$. For $w\in W$, we let $\dot{w}$ be a representative of $w$ in $N_H(T)$ (the choice of which essentially won't matter). The following (easy) lemma gives an alternative description of $W_C$.

\begin{lem}\label{autre}
Let $C\subseteq R^+$ be a closed subset. Then one has
$$W_C=\{w\in W : \dot{w}^{-1}B_C\dot{w}\subseteq B\}.$$
\end{lem}
\begin{proof}
For $w\in W$ and $\alpha\in R$ one has $\dot{w}^{-1}U_{\alpha}\dot{w}=U_{w^{-1}(\alpha)}$ (\cite[\S II.1.4(5)]{Ja}). Assume $w\in W_C$, then we get $\dot{w}^{-1}B_C\dot{w}=B_{w^{-1}(C)}\subseteq B$. Assume $\dot{w}^{-1}B_C\dot{w}\subseteq B$, then in particular $\dot{w}^{-1}U_{\alpha}\dot{w}=U_{w^{-1}(\alpha)}\subseteq B$ for any $\alpha\in C$ which implies $w^{-1}(\alpha)\in R^+$, hence $w\in W_C$.
\end{proof}

Recall that two roots $\alpha,\beta\in R$ are {\it orthogonal} if $\langle \alpha,\beta^\vee\rangle=0$, or equivalently $\langle \beta,\alpha^\vee\rangle=0$.

\begin{lem}\label{orth}
Let $C\subseteq R^+$ be a closed subset, $w_C\in W_C$ and $I\subseteq w_C(S)\cap C$ be a subset of pairwise orthogonal roots. Then $I$ satisfies conditions \(i\) and \(ii\) of Lemma \ref{commute}.
\end{lem}
\begin{proof}
Let $\alpha,\beta$ be two distinct orthogonal roots in $w_C(S)\cap C$ and assume $i\alpha + j\beta \in R$ for some nonzero integers $i,j$. Applying $w_C^{-1}$, we get $iw_C^{-1}(\alpha)+jw_C^{-1}(\beta)\in R$ with $w_C^{-1}(\alpha)$ and $w_C^{-1}(\beta)$ being orthogonal simple roots. In particular $i$ and $j$ must have the same sign. But we also have $s_{w_C^{-1}(\alpha)}(iw_C^{-1}(\alpha)+jw_C^{-1}(\beta))=-iw_C^{-1}(\alpha)+jw_C^{-1}(\beta)\in R$ which is impossible since $-i$ and $j$ now have different signs. Thus condition (ii) of Lemma \ref{commute} holds (even in the stronger form $(\Z\alpha\oplus \Z\beta)\cap R = \{\pm\alpha,\pm\beta\}$). Now assume $\alpha\in w_C(S)\cap C$ is the sum of several roots in $C$, then $w_C^{-1}(\alpha)\in S$ is the sum of several roots in $w_C^{-1}(C)\subseteq R^+$ which is again impossible since $w_C^{-1}(\alpha)$ is simple. This gives condition (i) of Lemma \ref{commute}.
\end{proof}

For $C$ and $I$ as in Lemma \ref{orth}, we thus have $U_C=U_{C\backslash I}\rtimes U_I$ with $U_I$ commutative (Lemma \ref{commute}).

Now let $C\subseteq R^+$ be a closed subset, $w_C\in W_C$, $I\subseteq w_C(S)\cap C$ be a subset of pairwise orthogonal roots, and set $\lambda:=\sum_{\alpha\in S}\lambda_{\alpha}$. Let $H_I$ be the Levi subgroup of $H$ that contains $T$ and whose roots are $\pm I$. Then $B_I=TU_I$ is the Borel subgroup $H_I \cap B$ of $H_I$. We have
\begin{equation}\label{eq:minu}
\langle w_C(\lambda), \alpha^\vee\rangle = 1 \quad \forall \alpha \in I.
\end{equation}
In particular, $w_C(\lambda)$ is a dominant weight for $H_I$. Let $L_I$ be the $B_I$-representation obtained by restriction from the irreducible $H_I$-representation over $E$ of highest weight $w_C(\lambda)$. We view $L_I$ as a representation of $B_C$ via the quotient map $B_C\twoheadrightarrow B_I=B_C/U_{C\backslash I}$ (see Lemma \ref{commute}). If $I'\subseteq I$, it is clear that there is a $B_C$-equivariant injection $L_{I'}\hookrightarrow L_I$ which
is unique up to multiplication by a nonzero scalar (its image is the $B_{I'}$-subrepresentation of $L_I$ generated by $w_C(\lambda)$). We fix a compatible system of such injections, that is, such that for any inclusions $I''\subseteq I'\subseteq I$ the corresponding diagram of injections is commutative (it is always possible to do so, note that there is only a finite number of $I$). We then define the inductive limit
\begin{equation}\label{ii}
L_{C,w_C}:= \ilim I{L_I},
\end{equation}
where $I$ runs among the subsets of $w_C(S)\cap C$ of pairwise orthogonal roots. More explicitly, $L_{C,w_C}$ is the quotient of $\oplus_IL_I$ by the subrepresentation generated by elements $x\oplus -x\in L_{I'}\oplus L_{I}$ for all $x\in L_{I'}$ and all subsets $I'\subseteq I\subseteq w_C(S)\cap C$ of
pairwise orthogonal roots. Up to isomorphism the representation $L_{C,w_C}$ does not depend on the above choice of compatible system of injections and note that the canonical maps $L_I\rightarrow L_{C,w_C}$ are all injections.

One can give a more explicit description of $L_I$. By the same proof as for Lemma~\ref{strgalpha} below (applied to $H$ and $I$, use that $w_C^{-1}(I)\subseteq S$) we have $H_I \cong T'_I \times \GL_2^I$ for some subtorus $T'_I \subseteq T$. Correspondingly, $T \cong T'_I \times \prod_{\alpha \in I} T_\alpha$ and $B_I\cong T'_I \times \prod_{\alpha\in I} B_\alpha$. Then
\begin{equation}\label{li}
L_I \cong w_C(\lambda)|_{T'_I}\otimes\bigg(\bigotimes_{\alpha\in I}L_{\alpha}\bigg),
\end{equation}
where $L_\alpha$ is the restriction to $B_\alpha$ of the irreducible $\GL_2$-representation over $E$ of highest weight $w_C(\lambda)|_{T_\alpha}$. Equation (\ref{eq:minu}) shows that $L_\alpha$ is the unique non-split extension of $s_\alpha w_C(\lambda)|_{T_\alpha}$ by $w_C(\lambda)|_{T_\alpha}$.

\begin{ex}
For $C=\varnothing$ and $w_{\varnothing}\in W_{\varnothing}=W$, one obviously has $L_{\varnothing,w_{\varnothing}}=w_{\varnothing}\big(\sum_{\alpha\in S}\lambda_{\alpha}\big)$.
\end{ex}

The following lemma follows directly from the construction of $L_{C,w_C}$ since the socle filtration is compatible with subobjects.

\begin{lem}\label{cw}
Let $C\subseteq R^+$ be a closed subset and $w_C\in W_C$. The $B_C$-representation $L_{C,w_C}$ has socle filtration $0={\rm Fil}_{-1}L_{C,w_C}\subsetneq {\rm Fil}_0L_{C,w_C}\subseteq \cdots $ such that for $j\in \Z_{\geq 0}$,
\begin{eqnarray*}
{\rm Fil}_{j}L_{C,w_C}/{\rm Fil}_{j-1}L_{C,w_C}&\cong &\bigoplus_{\substack{I\subseteq w_C(S)\cap C\\ \vert I\vert=j}}\bigg(\Big(\prod_{\alpha\in I}s_{\alpha}\Big)w_C\bigg)\Big(\sum_{\alpha\in S}\lambda_{\alpha}\Big)\\
&= &\bigoplus_{\substack{I\subseteq w_C(S)\cap C\\ \vert I\vert=j}}w_C\Big(\sum_{\alpha\in S}\lambda_{\alpha}\Big)-\sum_{\alpha\in I}\alpha
\end{eqnarray*}
for $I$ running among the subsets of $w_C(S)\cap C$ of pairwise orthogonal roots.
\end{lem}

\subsection{On the restriction of $\LL$ to various subgroups II}

We completely describe $(\LL\vert_{B_C})^{\ord}$ for all closed subsets $C\subseteq R^+$ in terms of the representations $L_{C,w_C}$ of (\ref{ii}). We keep the notation of \S\ref{fundamental}, \S\ref{one} and \S\ref{C}.

This section entirely consists of the proof of the following theorem.

\begin{thm}\label{rest}
Let $C\subseteq R^+$ be a closed subset, then
$$(\LL\vert_{B_C})^{\ord}\cong \bigoplus_{w_C\in W_C}L_{C,w_C}.$$
\end{thm}
\begin{proof}
As the weights $w(\lambda)$ occur with multiplicity $1$ in $\LL$, we often identify $w(\lambda) \in X(T)$ with the corresponding $1$-dimensional subspace of $\LL$.

\noindent
Step 1: We prove $\soc_{B_C}(\LL\vert_{B_C})^{\ord}=\oplus_{w_C\in W_C}w_C(\lambda)$.\\
We first prove that the stabilizer of the highest weight space $\lambda$ in $H$ is $B$. It is a closed subgroup of $H$ that contains $B$, hence it is a parabolic subgroup (\cite[Cor.\ 11.2]{Bo}). But it cannot contain any of the root subgroups $U_{-\alpha}$ for $\alpha\in R^+$. Indeed, otherwise it would also contain $\dot{s}_{\alpha}$ but we never have $s_{\alpha}(\lambda)=\lambda$ as $\langle \lambda,\alpha^{\vee}\rangle=\sum_{\beta\in S}\langle \lambda_{\beta},\alpha^{\vee}\rangle$ is always positive (in particular never $0$). Thus this stabilizer must be $B$ itself. Then $w(\lambda)$ is a weight of $\soc_{B_C}(\LL\vert_{B_C})$ if and only if $B_C$ fixes the subspace $w(\lambda)$ if and only if $\dot{w}^{-1}B_C\dot{w}$ fixes the subspace $\lambda$ if and only if $\dot{w}^{-1}B_C\dot{w}\subseteq B$ if and only if $w\in W_C$ by Lemma \ref{autre}. This finishes the proof of Step 1.

Step 2: We prove that for $w\in W$ such that $w(\lambda)$ is a weight of $(\LL\vert_{B_C})^{\ord}$ the following properties hold:
\begin{enumerate}
\item the elements of $C \cap -w(R^+)$ are contained in $-w(S)$ and are pairwise orthogonal;
\item  $\big((\prod_{\alpha \in I} s_{\alpha}) w\big)(\lambda)=w(\lambda)+\sum_{\alpha \in I}\alpha$ for all subsets $I \subseteq C \cap -w(R^+)$;
\item  $\big((\prod_{\alpha \in C \cap -w(R^+)} s_{\alpha}) w\big)(\lambda)$ is a weight of $\soc_{B_C}(\LL\vert_{B_C})^{\ord}$.
\end{enumerate}
Note that the order of the $s_\alpha$ in (ii) and (iii) is irrelevant by (i). We induct on $|C \cap -w(R^+)|$. If $C \cap -w(R^+) =\emptyset$, then $w \in W_C$, (i) and (ii) are trivial and (iii) is proven in Step 1. Otherwise, pick $\beta\in C \cap -w(R^+)$. We have $s_\beta w(\lambda) = w(\lambda) + n_\beta \beta$, where $n_\beta = - \langle w(\lambda),\beta^\vee\rangle = \langle \lambda, -w^{-1}(\beta)^\vee\rangle > 0$, as $\beta \in -w(R^+)$. Thus all weights $w(\lambda) + n \beta$ for $0 \le n \le n_\beta$ have to occur in the $U_{\beta}$-subrepresentation generated by $w(\lambda)$ as immediately follows from the properties of algebraic representations of $\SL_2$ in characteristic $0$, hence in $(\LL\vert_{B_C})^{\ord}$. However, none of the intermediate weights with $0 < n < n_\beta$ can be ordinary. Indeed, otherwise we could use the Weyl group action to map such an intermediate weight to $\lambda$. But since any other weight of $\LL$ is smaller than $\lambda$ in the dominance order it cannot lie on the line segment spanned by two other weights. Therefore $s_{\beta}w(\lambda) = w(\lambda) + \beta$ and, since $\langle \lambda, \gamma^\vee\rangle > 1$ for all $\gamma\in R^+\backslash S$, we see that $-w^{-1}(\beta) \in S$. In other words, (i$'$) $\beta \in -w(S)$. As $s_{w^{-1}(\beta)}$ is a simple reflection, it maps precisely one positive root to a negative root, namely $-w^{-1}(\beta)$. Therefore $C \cap -w(R^+) = (C \cap -(s_\beta w)(R^+)) \amalg \{\beta\}$ is a disjoint union. The induction hypothesis for $s_\beta w$ shows that:
\begin{enumerate}
\item[(i$''$)]$C \cap -(s_\beta w)(R^+)$ consists of pairwise orthogonal roots in $-w(S)$;
\item[(ii$'$)]$\big((\prod_{\alpha \in I} s_{\alpha}) s_\beta w\big)(\lambda)=(s_\beta w)(\lambda)+\sum_{\alpha \in I}\alpha = w(\lambda) + \sum_{\alpha \in I \cup \{\beta\}}\alpha$ for all subsets $I \subseteq C \cap -(s_\beta w)(R^+)$;
\item[(iii)] $\big((\prod_{\alpha \in C \cap -(s_\beta w)(R^+)} s_{\alpha}) s_\beta w\big)(\lambda)$ is a weight of $\soc_{B_C}(\LL\vert_{B_C})^{\ord}$.
\end{enumerate}
It remains to show that $\beta$ is orthogonal to any $\alpha \in C \cap -(s_\beta w)(R^+)$ as then (i) follows from
(i$'$) and (i$''$) and (ii) follows from (ii$'$) by considering all choices of $\beta \in C \cap -w(R^+)$. From (ii$'$) we see that $(s_\alpha s_\beta w)(\lambda) = w(\lambda) + \alpha + \beta$. But, by applying the above argument with $\alpha$ instead of $\beta$, it also equals $s_\alpha(w(\lambda)+\beta)=w(\lambda)+\alpha+ s_\alpha(\beta)$, so $s_\alpha(\beta)=\beta$, i.e.\ $\alpha$ and $\beta$ are orthogonal. This finishes the proof of Step 2.

\noindent
Step 3: We prove that for $w\in W$ such that $w(\lambda)$ is a weight of $(\LL\vert_{B_C})^{\ord}$ the $B_C$-subrepre\-sentation $\langle B_C \cdot w(\lambda)\rangle$ of $(\LL\vert_{B_C})^{\ord}$ (or equivalently of $\LL\vert_{B_C}$) generated by $w(\lambda)$ is isomorphic to one of the representations $L_I$ in (\ref{li}).\\
Let $I:=C \cap -w(R^+)$, $w_C:=(\prod_{\alpha\in I}s_{\alpha})w$ and note that $I\subseteq w_C(S)\cap C$ by (i) of Step 2 (use that $I=-(\prod_{\alpha\in I}s_{\alpha})(I)$) and that $w_C(\lambda)$ is a weight of $\soc_{B_C}(\LL\vert_{B_C})^{\ord}$ by (iii) of Step 2. By Step 1 it follows that $w_C \in W_C$. Note that $U_{C\backslash I}$ acts trivially on the subspace $w(\lambda)$, as $U$ acts trivially on the highest weight space $\lambda$. By Lemma~\ref{commute}, it follows that $\langle B_C \cdot w(\lambda)\rangle = \langle B_I \cdot w(\lambda)\rangle$. Also note that $U_I^-$ fixes $w(\lambda)$ (where $U_I^-$ is the unipotent subgroup generated by the $U_{-\alpha}$, $\alpha\in I$), so $\langle B_I \cdot w(\lambda)\rangle$ is the irreducible $H_I$-representation of lowest weight $w(\lambda)$. Its highest weight is $((\prod_{\alpha\in I}s_{\alpha})w) (\lambda) = w_C (\lambda)$. By definition of $L_I$ we therefore have $\langle B_I \cdot w(\lambda)\rangle \cong L_I$ as $B_C$-representation.

\noindent
Step 4: We prove the theorem.\\
For any $w_C\in W_C$ and any subset $I\subseteq w_C(S)\cap C$ of pairwise orthogonal roots, it is easy to check that $w:=(\prod_{\alpha\in I}s_{\alpha})w_C\in W$ is such that the $B_C$-subrepresentation of $\LL$ generated by $w(\lambda)$ has socle $w_C(\lambda)$ and constituents $\big((\prod_{\alpha\in I'}s_{\alpha})w_C\big)(\lambda)$ for $I'\subseteq I$ (use that $\langle\lambda, \beta^\vee\rangle = 1$ for all $\beta \in S$, we leave here the details to the reader), in particular sits in $(\LL\vert_{B_C})^{\ord}$. Moreover we clearly have $I=C \cap -w(R^+)$ and this $B_C$-subrepresentation is thus $L_I$. Conversely, we have seen that for any $w\in W$ such that $w(\lambda)$ is a weight of $(\LL\vert_{B_C})^{\ord}$ there is $w_C\in W_C$ and a subset $I\subseteq w_C(S)\cap C$ such that the $B_C$-subrepresentation of $(\LL\vert_{B_C})^{\ord}$ generated by $w(\lambda)$ is $L_I$. Let $W(w_C)\subseteq W$ be the subset of $w$ such that the $B_C$-subrepresentation generated by $w(\lambda)$ sits in $(\LL\vert_{B_C})^{\ord}$ and has socle $w_C(\lambda)$, and let $\LL_{C,w_C}\subseteq (\LL\vert_{B_C})^{\ord}$ be the sum of all these $B_C$-subrepresentations for all $w\in W(w_C)$. By what is just above we have
\begin{eqnarray}\label{iio}
\LL_{C,w_C}=\ilim {I}L_I,
\end{eqnarray}
where $I$ runs among the subsets of $w_C(S)\cap C$ of pairwise orthogonal roots and where the injections are the canonical inclusions inside $(\LL\vert_{B_C})^{\ord}$. Since, for $I'\subseteq I$, there is only one injection $L_{I'}\hookrightarrow L_I$ up to multiplication by a nonzero scalar (as both representations have the same socle), we deduce from Step 3 that the inductive limit $\ilim {}L_I$ of (\ref{iio}) is isomorphic to the inductive limit $\ilim {}L_I$ of (\ref{ii}) and that we have a $B_C$-isomorphism $\LL_{C,w_C}\cong L_{C,w_C}$ for all $w_C\in W_C$. Now, the canonical map given by the direct sum of the inclusions
$$\bigoplus_{w_C\in W_C} \LL_{C,w_C} \longrightarrow (\LL\vert_{B_C})^{\ord}$$
is $B_C$-equivariant, surjective (as any $w(\lambda)$ being a weight of $(\LL\vert_{B_C})^{\ord}$ sits in one $\LL_{C,w_C}$) and injective (as it is an isomorphism on the socles). This finishes the proof of the theorem.
\end{proof}

\begin{rem}\label{rema2}
(i) Note that we have used in Step $2$ of the above proof that $E$ has characteristic $0$.\\
(ii) If $H=H_1\times H_2$, $C=C_1\amalg C_2\subseteq R^+=R_1^+\amalg R_2^+$ and $w_C=(w_{C_1},w_{C_2})\in W_C=W_{C_1}\times W_{C_2}\subseteq W=W_1\times W_2$, one has $L_{C,w_C}=L_{C_1,w_{C_1}}\otimes_E L_{C_2,w_{C_2}}$ (where we index by $i$ everything related to $H_i$, $i=1,2$), see Remark \ref{rema}(ii).
\end{rem}

\subsection{Variant mod $p$}\label{variant1}

We give a variant of the previous results when the ground field is $k_E$ and not $E$.

We consider $H/\oE$ a split connected reductive algebraic group, $T\subseteq H$ a split maximal torus over $\oE$ and $B$ a Borel subgroup over $\oE$ containing $T$. We define $\big(X(T),R,X^{\vee}(T),R^{\vee}\big)$, $S$, $R^+$ and $W$ as before. We assume that the derived subgroup $H^{\rm der}$ is simply connected and denote by $\lambda_{\alpha}\in X(T)$, $\alpha\in S$ the fundamental weights.

For $\lambda\in X(T)$ a dominant weight, we consider the following algebraic representation of $H$ over $\oE$:
$$L(\lambda)_{/\oE}:=\big({\rm ind}_{B^-}^{H}\lambda\big)_{/\oE},$$
where $B^-$ is the Borel opposite to $B$ and ind means the algebraic induction functor of \cite[\S I.3.3]{Ja} and we set
$$\Lbar(\lambda):=L(\lambda)_{/\oE}\otimes_{\oE}k_E=\big({\rm ind}_{B^-}^{H}\lambda\big)_{/k_E},$$
where the last equality follows from \cite[II.8.8(1)]{Ja}. We then define as in \S\ref{fundamental},
\begin{equation}\label{lbar}
\LLbar:=\bigotimes_{\alpha\in S}\Lbar(\lambda_{\alpha})=\Big(\bigotimes_{\alpha\in S}L(\lambda_{\alpha})_{/\oE}\Big)\otimes_{\oE}k_E.
\end{equation}
It follows from Theorem \ref{mult} and the second equality in (\ref{lbar}) that the weights $w(\sum_{\alpha}\lambda_{\alpha})$ for $w\in W$ are the only weights that occur exactly once in $\LLbar\vert_T$ and we call them {\it ordinary weights} of $\LLbar$. If $C\subseteq R^+$ is a closed subset and if $B_C\subseteq B$ is the associated closed subgroup of $B$ (\S\ref{C}), as in Definition \ref{Cord} we define $(\LLbar\vert_{B_C})^{\ord}\subseteq \LLbar\vert_{B_C}$ to be the maximal $B_C$-subrepresentation of $\LLbar\vert_{B_C}$ such that all its weights are ordinary.

For $w_C\in W_C$ (see (\ref{wc})) and $I\subseteq w_C(S)\cap C$ a subset of pairwise orthogonal roots, we define an algebraic $B_I$-representation $\Lbar_I$ over $k_E$ as in (\ref{li}) (using (\ref{eq:minu})) which is still the restriction to $B_I$ of the irreducible $H_I$-representation over $k_E$ of highest weight $w_C(\lambda)$. If $I'\subseteq I$ then $\Lbar_{I'}$ embeds into $\Lbar_I$ (uniquely up to nonzero scalar) and we set $\Lbar_{C,w_C}:=\ilim I{\Lbar_I}$ as in (\ref{ii}). Lemma \ref{cw} still holds for $\Lbar_{C,w_C}$.

\begin{definit}[{\cite[\S4.3]{spr-st}}]\label{goodprime}
  We say that $p$ is a \emph{good prime for $H$} if $\langle \lambda_\alpha, \beta^\vee\rangle < p$ for all $\alpha \in S$ and $\beta \in R^+$.
\end{definit}

Explicitly, $p$ fails to be a good prime only in the following cases: $p = 2$ and the root system of $H$ has an irreducible component {\it not} of type A$_r$; $p = 3$ and the root system of $H$ has an irreducible component of type E$_r$, F$_4$, G$_2$; $p = 5$ and the root system of $H$ has an irreducible component of type E$_8$. \emph{In particular, note that $p$ is a good prime for $H$ if $p > 5$ or if $H = \GL_n$.
}
\begin{thm}\label{restbar}
Let $C\subseteq R^+$ be a closed subset.\\
\textup{(i)} We have an embedding of $B_C$-representations over $k_E$,
$$\bigoplus_{w_C\in W_C}\Lbar_{C,w_C}\hookrightarrow (\LLbar\vert_{B_C})^{\ord},$$
which is an isomorphism on the $B_C$-socles.\\
\textup{(ii)} Assume that $p$ is a good prime for $H$. Then
$$\bigoplus_{w_C\in W_C}\Lbar_{C,w_C}\buildrel\sim\over\longrightarrow (\LLbar\vert_{B_C})^{\ord}.$$
\end{thm}
\begin{proof}
(i) Step 1 \ in \ the \ proof \ of \ Theorem \ref{rest} \ works \ over \ $k_E$ \ and \ yields $\soc_{B_C}(\LLbar\vert_{B_C})^{\ord}=\oplus_{w_C\in W_C}w_C(\lambda)$ (here $\lambda=\sum_{\alpha\in S}\lambda_{\alpha}$). Let $w_C\in W_C$ and $I\subseteq w_C(S)\cap C$ a subset of pairwise orthogonal roots, then the same proof as in Step 3 of Theorem \ref{rest} shows that the $B_I$-subrepresentation of $\LLbar\vert_{B_I}$ generated by $((\prod_{\alpha\in I}s_{\alpha})w_C)(\lambda)$ is $\Lbar_I$. (To see that $\langle B_I \cdot ((\prod_{\alpha\in I}s_{\alpha})w_C)(\lambda)\rangle$ is irreducible as $H_I$-representation, note that by \cite[Lem.\ II.2.13]{Ja} it is a quotient of the Weyl module of $H_I$ of highest weight $w_C(\lambda)$ and that $w_C(\lambda)$ is a minuscule weight for $H_I$.)
So we have $\Lbar_I\hookrightarrow \LLbar\vert_{B_I}$ for all such $I$ and thus $\Lbar_{C,w_C}\hookrightarrow \LLbar\vert_{B_C}$. We deduce a map $\bigoplus_{w_C\in W_C}\Lbar_{C,w_C}\rightarrow (\LLbar\vert_{B_C})^{\ord}$ which is injective as it is an isomorphism on the socles.\\
(ii) The crucial point in the proof of Theorem \ref{rest} where we use characteristic $0$ is in Step $2$, see Remark
\ref{rema2}(i). (We freely use the notation of this proof now.) First, using the argument in Step 2 that an ordinary weight
of $\LLbar$ cannot lie on the line segment spanned by two other distinct weights, it suffices to show that $w(\lambda)+\beta$
is a weight of the $U_{\beta}$-subrepresentation generated by $w(\lambda)$. (Note that $s_{\beta}w(\lambda)=w(\lambda)+n_{\beta}\beta$ is the highest weight of this subrepresentation.) Suppose that $A$ is a $k_E$-algebra (it would suffice to take $A = \overline{k_E}$) and that $u \in U_\beta(A)$. For $\alpha \in S$ suppose that $x_\alpha \in \Lbar(\lambda_\alpha) \otimes_{k_E} A$ is of weight $\lambda_\alpha$, and let $x := \otimes_{\alpha \in S} x_\alpha \in \LLbar \otimes_{k_E} A$. By \cite[II.1.19(5)--(6)]{Ja} we know that $u x_\alpha - x_\alpha$ is contained in the sum of weight spaces of weights $w(\lambda_\alpha)+i\beta$ with $i > 0$. Thus
\begin{equation*}
  [ux-x]_{w(\lambda)+\beta} = \sum_{\alpha\in S} \bigg( [ux_\alpha-x_\alpha]_{w(\lambda_\alpha)+\beta} \otimes \bigotimes_{\delta \in S \bs \{\alpha\}} x_\delta \bigg),
\end{equation*}
where $[\cdot]_\mu$ denotes the projection to the $\mu$-weight space. Since the sum
\begin{equation*}
  \sum_{\alpha\in S} \bigg( \Lbar(\lambda_\alpha)_{w(\lambda_\alpha)+\beta} \otimes \bigotimes_{\delta \in S \bs \{\alpha\}} \Lbar(\lambda_\delta)_{w(\lambda_\delta)}\bigg)\subseteq \LLbar
\end{equation*}
is direct, it suffices to show that $w(\lambda_\alpha)+\beta$ is a weight of the $U_\beta$-subrepresentation generated by $w(\lambda_\alpha)$ for some $\alpha \in S$.
Thus by the properties of algebraic representations of ${\rm SL}_2$ in characteristic $p$ it is enough to show that $-p < \langle w(\lambda_\alpha), \beta^\vee\rangle < 0$
for some $\alpha \in S$. The lower bounds holds for all $\alpha$ since $p$ is a good prime, and the upper bounds holds for at least one $\alpha$ since $w^{-1}(\beta) \in -R^+$.
\end{proof}

\begin{rem}\label{h}
Recall that the height of a positive root $\alpha=\sum_{\beta\in S}n_{\beta}\beta$ is by definition the positive integer $h(\alpha):=\sum_{\beta\in S}n_{\beta}$. We let $h:=1+{\rm max}\{h(\alpha) : \alpha\in R^+\}\in \Z_{>0}$ (when the root system associated to $R$ is irreducible, $h$ is called its Coxeter number). Note that, for $\alpha \in S$, the representation $\Lbar(\lambda_\alpha)$  is no longer irreducible in general, but it is irreducible e.g.\ when $p\geq 2h-2$ (by \cite[II.5.6]{Ja}) or when $H=\GL_n$ (as $\lambda_\alpha$ is minuscule then). But the above argument does not depend on the choice of $H$-stable $\oE$-lattice $L(\lambda_\alpha)_{/\oE}$ in $L(\lambda_\alpha)$. It only uses that $\Lbar(\lambda_\alpha)$ is an $H$-representation over $k_E$ that has $\lambda_\alpha$ as its unique highest weight and such that the $\lambda_\alpha$-weight space has dimension one.
\end{rem}

\section{The $G(\Qp)$-representation $\Pi(\rho)^{\ord}$}\label{ord}

We construct the representation $\Pi(\rho)^{\ord}$ of $G(\Qp)$ over $E$ associated to a sufficiently generic ordinary representation $\rho$ of $\gp$ over $E$.

\subsection{Some preliminaries}\label{prel}

We first give a few representation-theoretic preliminaries.

We fix $G/\Qp$ a connected split reductive algebraic group, $T\subseteq G$ a split maximal torus over $\Qp$ and we let $\big(X(T),R,X^{\vee}(T),R^{\vee}\big)$ be the root datum of $G$. We fix a choice $S\subseteq R$ of simple roots, we let $R^+\subseteq R$ be the positive roots and $B\subseteq G$ (resp.\ $B^-\subseteq G$) the Borel subgroup corresponding to $R^+$ (resp.\ $-R^+$). The triple $(G,B,T)$ is determined up to inner automorphism by the {\it based} root datum $\big(X(T),S,X^{\vee}(T),S^{\vee}\big)$. The dual based root datum $\big(X^{\vee}(T),S^{\vee},X(T),S\big)$ determines a dual triple $(\widehat{G},\widehat{B},\widehat{T})$, where $\widehat{G}$ is the dual group of $G$ (over $E$) and $X(\widehat{T})\simeq X^{\vee}(T)$. We let $W$ be the Weyl group of $(G,T)$ or equivalently of $(\widehat{G},\widehat{T})$. We endow the groups $G(\Qp)$, $B(\Qp)$, $B^-(\Qp)$, $T(\Qp)$, $\widehat{G}(E)$, $\widehat{B}(E)$ and $\widehat{T}(E)$ with their natural structure of $p$-adic analytic groups (in particular they are all topological groups).

Let $\chi:T(\Qp)\rightarrow \oE^{\times}$ be a continuous character that takes values in $\oE^{\times}\subseteq E^{\times}$ (we say $\chi$ is unitary). By inflation $B^-(\Qp)\twoheadrightarrow T(\Qp)\buildrel {{\chi}}\over \longrightarrow \oE^{\times}$, we consider ${\chi}$ as a continuous character of $B^-(\Qp)$. Following \cite{Sc1}, we define the continuous parabolic induction,
\begin{multline}\label{para}
\big(\Ind_{B^-(\Qp)}^{G(\Qp)}{\chi}\big)^{{\mathcal C}^0}:=\{f:G(\Qp)\rightarrow E,\ f{\rm \ is\ continuous\ and\ }\\
f(bg)=\chi(b)f(g)\ \forall b\in B^-(\Qp),\ \forall g\in G(\Qp)\}
\end{multline}
that we endow with an $E$-linear left action of $G(\Qp)$ by $(gf)(g'):=f(g'g)$ ($g,g'\in G(\Qp)$). This is a $p$-adic Banach space with a unit ball given by
\begin{equation}\label{ball}
\{f:G(\Qp)\rightarrow \oE,\ f{\rm \ is\ continuous\ and\ }f(bg)=\chi(b)f(g)\ \forall b,g\}.
\end{equation}
By \cite[Prop.\ 2.4]{Sc1}, the above $G(\Qp)$-action makes $\big(\Ind_{B^-(\Qp)}^{G(\Qp)}{\chi}\big)^{{\mathcal C}^0}$ an {\it admissible unitary continuous} representation of $G(\Qp)$ (over $E$). Recall that, by continuous, we mean that the ``evaluation'' map
$$G(\Qp)\times \big(\Ind_{B^-(\Qp)}^{G(\Qp)}{\chi}\big)^{{\mathcal C}^0}\longrightarrow \big(\Ind_{B^-(\Qp)}^{G(\Qp)}{\chi}\big)^{{\mathcal C}^0}$$
is continuous. By admissible, we mean that the continuous dual
$$\Big(\big(\Ind_{B^-(\Qp)}^{G(\Qp)}{\chi}\big)^{{\mathcal C}^0}\Big)^*:=\Hom_{\rm cont}\Big(\big(\Ind_{B^-(\Qp)}^{G(\Qp)}{\chi}\big)^{{\mathcal C}^0},E\Big)$$
is of finite type over the Iwasawa algebra $E\otimes_{\oE}\oE[[G(\Zp)]]$ of $G(\Zp)$ (here $G(\Zp)$ is the $\Zp$-points of an integral model of $G$ over $\Z$). Indeed, Iwasawa decomposition shows that restriction of functions in (\ref{para}) from $G(\Qp)$ to $G(\Zp)$ yields an embedding of $\big(\Ind_{B^-(\Qp)}^{G(\Qp)}{\chi}\big)^{{\mathcal C}^0}$ into the Banach space of continuous functions from $G(\Zp)$ to $E$. As the dual of this latter space is the Iwasawa algebra of $G(\Zp)$, this implies that $\big(\big(\Ind_{B^-(\Qp)}^{G(\Qp)}{\chi}\big)^{{\mathcal C}^0}\big)^*$ is a quotient of $E\otimes_{\oE}\oE[[G(\Zp)]]$. Finally, by unitary, we mean that $\big(\Ind_{B^-(\Qp)}^{G(\Qp)}{\chi}\big)^{{\mathcal C}^0}$ contains an open ball stable by $G(\Qp)$, for instance the unit ball (\ref{ball}). Recall that admissible unitary continuous representations of $G(\Qp)$ on $p$-adic Banach spaces over $E$ form an abelian category (\cite{Sc1}, in fact unitarity is unnecessary).

\begin{thm}\label{classic}
\textup{(i)} The $G(\Qp)$-representation $\big(\Ind_{B^-(\Qp)}^{G(\Qp)}{\chi}\big)^{{\mathcal C}^0}$ is topologically of finite length.\\
\textup{(ii)} Assume that, for all $\alpha\in S$, the reduction $\overline{\chi\circ \alpha^{\vee}}:\Qp^{\times}\rightarrow k_E^{\times}$ is not the trivial character, then $\big(\Ind_{B^-(\Qp)}^{G(\Qp)}{\chi}\big)^{{\mathcal C}^0}$ is topologically irreducible.\\
\textup{(iii)} One has $\big(\Ind_{B^-(\Qp)}^{G(\Qp)}{\chi}\big)^{{\mathcal C}^0}\cong\big(\Ind_{B^-(\Qp)}^{G(\Qp)}{\chi'}\big)^{{\mathcal C}^0}$ if and only if $\chi=\chi'$.
\end{thm}
\begin{proof}
(i) It is enough to prove that the reduction mod $\pE$ of the unit ball (\ref{ball}) is of finite length as a smooth representation of $G(\Qp)$ over $k_E$. When $G={\GL}_n$, this is due to one of us (\cite[Cor.\ 1.2(i)]{He2}). The general split case is due to Abe (\cite[Cor.\ 5.13]{Ab}).\\
(ii) It is enough to prove that the reduction mod $\pE$ of the unit ball (\ref{ball}) is irreducible as a smooth representation of $G(\Qp)$ over $k_E$. When $G={\GL}_n$, this is due to Ollivier (\cite[Thm.\ 4]{Ol}). The general split case is again due to Abe (\cite[Thm.\ 1.3]{Ab}).\\
(iii) It is a direct consequence of \cite[Cor.\ 4.3.5]{Em2}.
\end{proof}

Let us mention the following ``folklore'' conjecture which is a (straightforward) strengthening of a special case of \cite[Conj.\ 2.5]{Sc1}.

\begin{conj}\label{folk}
Let $\chi:T(\Qp)\rightarrow \oE^{\times}\subseteq E^{\times}$ be a unitary continuous character. Then the $G(\Qp)$-representation $\big(\Ind_{B^-(\Qp)}^{G(\Qp)}{\chi}\big)^{{\mathcal C}^0}$ is topologically irreducible if and only if $\chi\circ \alpha^{\vee}\ne 1$ for every $\alpha\in S$.
\end{conj}

We now make the following two assumptions on $G$: its centre is {\it connected} and its derived subgroup is {\it simply connected} (or equivalently by Proposition \ref{sc} the centre of its dual $\widehat G$ is connected). As we have seen, this is equivalent to the fact that both $G$ and $\widehat G$ admit fundamental weights (Proposition \ref{sc}). Following \cite{BG}, we define a twisting element (for $G$) as an element $\theta\in X(T)$ such that for any $\alpha\in S$ one has $\langle \theta,\alpha^{\vee}\rangle=1$. It is obviously unique modulo $X^0(T)$. If $(\lambda_{\alpha})_{\alpha\in S}$ denote the fundamental weights of $G$, it is straightforward to check that $\theta:=\sum_{\alpha\in S}\lambda_{\alpha}$ is such a twisting element.

\begin{rem}
A split connected reductive group with connected centre can have a twisting element without having fundamental weights. For instance consider the group $({\GL}_2\times {\GL}_2)/{\mathbb G}_{\rm m}$, where ${\mathbb G}_{\rm m}$ embeds diagonally into the centre of ${\GL}_2\times {\GL}_2$.
\end{rem}

For $\alpha\in R$, we denote by $U_{\alpha}\subseteq G$ the root subgroup as in \S\ref{C}. If $C\subseteq R$ is a closed subset, we denote by $G_C$ the Zariski closed algebraic subgroup of $G$ generated by $T$, $U_{\alpha}$ and $U_{-\alpha}$ for $\alpha\in C$. If $C=\{\alpha\}$, we write $G_{\alpha}$ instead of $G_{\{\alpha\}}$. If $J\subseteq S$ is a subset of pairwise orthogonal roots, then the same proof as in Lemma \ref{orth} shows that $J$ is closed and thus $G_J$ is defined (and its positive roots are exactly $J$). Moreover, in this case $G_J$ is a Levi subgroup.

\begin{lem}\label{strgalpha}
Let $J\subseteq S$ be a subset of pairwise orthogonal roots. Then there is a subtorus $T'_J\subseteq T$ which is central in $G_J$ such that $G_J\cong T'_J\times {\GL}_2^{J}$.
\end{lem}
\begin{proof}
Denote by $\Z J$ (resp.\ $\Z R$) the sublattice of $X(T)$ generated by the roots in $J$ (resp.\ in $R$). Since $J\subseteq S$ we have that $\Z J$ is a direct summand of $\Z R$. Since $\Z R$ is a direct summand of $X(T)$ (as $G$ has a connected centre), the same holds for $\Z J$ and hence $G_J$ also has a connected centre $Z(G_J)$. Replacing $G$ by its dual $\widehat G$ gives that the derived subgroup $G_J^{\rm der}$ of $G_J$ is simply connected (Proposition \ref{sc}). From the assumption on $J$, we easily get $G_J^{\rm der}\cong {\rm SL}_2^{J}$ and thus its centre $Z(G_J^{\rm der})$ is $\mu_2^{J}$. Consider the natural exact sequence
$$1\longrightarrow Z(G_J^{\rm der})\longrightarrow Z(G_J) \times G_J^{\rm der}\longrightarrow G_J\longrightarrow 1,$$
where $Z(G_J^{\rm der})\cong (\bigcap_J \ker(\alpha)) \cap G_J^{\rm der}\cong Z(G_J) \cap G_J^{\rm der}$ embeds diagonally into $Z(G_J) \times G_J^{\rm der}$. As $Z(G_J)$ is connected (a torus), by the elementary divisor theorem, we can find an isomorphism $Z(G_J) \cong T'_J\times {\mathbb G}_{\rm m}^{J}$ for some torus $T'_J\subseteq T\subseteq G_J$ such that the natural map $Z(G_J^{\rm der}) \hookrightarrow Z(G_J)$ is identified with the natural embedding $\mu_2^{J} \hookrightarrow {\mathbb G}_{\rm m}^{J}\hookrightarrow T'_J\times {\mathbb G}_{\rm m}^{J}$. Therefore we have
\begin{equation*}
G_J \cong T'_J \times \left( \frac{{\mathbb G}_{\rm m} \times {\rm SL}_2}{\mu_2}\right)^{J} \cong T'_J\times \GL_2^{J}.\qedhere
\end{equation*}
\end{proof}

Recall that, if $A$ is any commutative $\Qp$-algebra, one has
\begin{eqnarray*}
T(A)&=&\Hom_{\Spec(\Qp)}\big(\Spec(A),\Spec(\Qp[X(T)])\big)\\
&=&\Hom_{\Z}\big(X(T),A^{\times}\big)\\
&=&\Hom_{\Z}\big(X(T),\Z\big)\otimes_{\Z}A^{\times}\\
&=&X(\widehat{T})\otimes_{\Z}A^{\times},
\end{eqnarray*}
where $A^{\times}$ is the multiplicative group of units in $A$. If $\widehat{\chi}$ is any continuous character
$$\widehat{\chi}: \gp\twoheadrightarrow \gp^{\ab}\longrightarrow \widehat{T}(E),$$
we associate to $\widehat{\chi}$ a continuous character ${\chi}:T(\Qp)\rightarrow E^{\times}$ by taking the composite of the maps
\begin{equation}\label{cft}
T(\Qp)\cong X(\widehat{T})\otimes_{\Z}\Qp^{\times} \hookrightarrow X(\widehat{T})\otimes_{\Z} \gp^{\ab}\rightarrow X(\widehat{T})\otimes_{\Z} \widehat{T}(E)\rightarrow E^{\times},
\end{equation}
where the first injection is induced by local class field theory. One can check that the character $\chi$ is uniquely determined by the relation $\chi\circ\lambda=\lambda\circ\widehat\chi$ (we will always suppress the Artin map from such formulas) for all $\lambda\in X(\widehat{T})=X({T})^\vee$ and that, if $w\in W$, the character $w(\chi)$ (defined as $w(\chi)(t):=\chi(\dot{w}^{-1}t\dot{w})$ for $\dot w$ as in \S\ref{C}) corresponds to $w(\widehat{\chi})$ (defined as $w(\widehat{\chi})(g):=\dot{w}\widehat{\chi}(g)\dot{w}^{-1}$).

\begin{ex}
If $G={\GL}_n$ and $\widehat{\chi}=\smat{\widehat{\chi}_{1}&&\\&\varddots &\\&&\widehat{\chi}_{n}}$ then we have (via $\Qp^{\times}\hookrightarrow \gp^{\ab}$)
$$\chi\begin{pmatrix}x_1&&\\&\varddots &\\&&x_n\end{pmatrix}=\widehat{\chi}_{1}(x_1)\widehat{\chi}_{2}(x_2)\cdots \widehat{\chi}_{n}(x_n).$$
\end{ex}

\subsection{Good conjugates of generic ordinary $\rho$}\label{Ch}

We associate closed subsets of $R^{+\vee}$ to a sufficiently generic ``ordinary'' representation of $\gp$. We keep the notation of \S\ref{prel}.

We consider a continuous homomorphism $\rho$ from $\gp$ to $\widehat{G}(E)$:
$${\rho}:\gp\longrightarrow \widehat{G}(E).$$
When $\rho$ takes values in our fixed Borel subgroup $\widehat{B}(E)$ of $\widehat{G}(E)$ we say $\rho$ is {\it ordinary} (this terminology is usually rather used in a more specific situation, but we don't want to introduce a new terminology). For any ordinary $\rho$, we let $C_{\rho}\subseteq R^{+\vee}$ be the closed subset of roots such that $\widehat{B}_{C_{\rho}}$ is the smallest closed subgroup of $\widehat{B}$ containing $\widehat{T}$ such that $\rho$ takes values in $\widehat{B}_{C_{\rho}}(E)$, that is such that we have
$$\rho:\gp\longrightarrow \widehat{B}_{C_{\rho}}(E)\subseteq \widehat{B}(E)\subseteq \widehat{G}(E).$$
Note that $C_{\rho}$ exists thanks to Lemma \ref{borel}. Equivalently, ${C_{\rho}}$ is the smallest closed subset of $R^{+\vee}$ such that $\widehat{B}_{C_{\rho}}(E)$ contains all the $\rho(g)$ for $g\in \gp$.

For any closed subset $C\subseteq R^{+\vee}$ we denote by $\widehat U_C$ the unipotent radical of $\widehat{B}_C$ and set $\widehat{U}:=\widehat U_{R^{+\vee}}$. For any $\alpha^{\vee}\in R^{+\vee}$ we denote by $U_{\alpha^\vee}\subseteq \widehat{U}$ the root subgroup associated to $\alpha^{\vee}$. Recall that the product induces an isomorphism of varieties (for any order on the $\alpha^{\vee}$):
\begin{equation}\label{prod}
\varphi:\ \ \prod_{\alpha^{\vee}\in C}\!\!U_{\alpha^{\vee}}\ \buildrel\sim\over\longrightarrow \ \widehat U_C.
\end{equation}
There is a concrete way to get ${C_{\rho}}$ using (\ref{prod}). First let $b=tu\in \widehat{B}(E)$ with $t\in \widehat{T}(E)$, $u\in \widehat{U}(E)$ and write $u=\prod u_{\alpha^{\vee}}$ in (\ref{prod}) (for $C=R^{+\vee}$) with $u_{\alpha^{\vee}}\in U_{\alpha^{\vee}}(E)$. Let
$$C_{b,\varphi}:=\{\alpha^{\vee}\in R^{+\vee} : u_{\alpha^{\vee}}\ne 1\}.$$
Then we see that $b \in \wh B_C$ is equivalent to $C_{b,\varphi} \subseteq C$ for any closed subset $C \subseteq R^{+\vee}$. Therefore
$C_{\rho}$ is the smallest closed subset of $R^{+\vee}$ containing all the subsets $C_{\rho(g),\varphi}$ for all $g\in \gp$.

\begin{lem}\label{combi}
Let $C\subseteq R^{+\vee}$ be a closed subset and let $\alpha_1^\vee,\dots,\alpha_n^\vee$ be distinct roots in $R^{+\vee}\backslash C$. Then there is a permutation $\sigma$ on $\{1,\dots,n\}$ such that for all $i$, $\alpha_{\sigma(i)}^\vee$ is not in the smallest closed subset containing $C$ and the $\alpha_{\sigma(j)}^\vee$ for $1\leq j\leq i-1$.
\end{lem}
\begin{proof}
We first prove that if $\alpha^\vee$ and $\beta^\vee$ are two distinct roots in $R^{+\vee}\backslash C$ such that $\beta^\vee$ belongs to the smallest closed subset containing $C$ and $\alpha^\vee$ then $h(\alpha^\vee)<h(\beta^\vee)$ (where $h(\cdot)$ is the height of a positive root, see Remark \ref{h}). Indeed, by assumption we have $\beta^\vee=n \alpha^\vee+\gamma_1^\vee+\cdots +\gamma_r^\vee$ for some $\gamma_i^\vee\in C$ and some $n\in \Z_{>0}$ hence $h(\beta^\vee)=nh(\alpha^\vee)+h(\gamma_1^\vee)+\cdots +h(\gamma_r^\vee)>h(\alpha^\vee)$. To get the statement, we can thus take any permutation $\sigma$ such that $h(\alpha_{\sigma(1)}^\vee) \geq h(\alpha_{\sigma(2)}^\vee) \geq \cdots \geq h(\alpha_{\sigma(n)}^\vee)$.
\end{proof}

For $\rho$ ordinary, we define $\widehat\chi_{\rho}:\gp\buildrel\rho\over\longrightarrow \widehat{B}(E)\twoheadrightarrow \widehat{T}(E)$.

\begin{lem}\label{conjalp}
Let $\rho:\gp\rightarrow \widehat{B}_{C_{\rho}}(E)\subseteq \widehat{B}(E)$ be a continuous homomorphism and assume that $\alpha^\vee\circ \widehat\chi_{\rho}\ne 1$ for all $\alpha^{\vee}\in R^{+\vee}\backslash C_{\rho}$. If $\alpha^\vee\in R^{+\vee}\backslash C_{\rho}$ and $u_{\alpha^\vee}\in U_{\alpha^\vee}$, $u_{\alpha^\vee}\ne 1$, then the subset $C_{u_{\alpha^\vee}^{\;}\rho u_{\alpha^\vee}^{-1}}$ is equal to the smallest closed subset containing $C_{\rho}$ and $\alpha^\vee$.
\end{lem}
\begin{proof}
Denote by $C_{\rho,\alpha^\vee}$ the smallest closed subset containing $C_{\rho}$ and $\alpha^\vee$ and let $\widetilde C_{\rho,\alpha^\vee}\subseteq C_{\rho,\alpha^\vee}$ be the subset of roots that are not the sum of at least two roots of $C_{\rho,\alpha^\vee}$. Then $C_{\rho,\alpha^\vee}$ is also the smallest closed subset containing $\alpha^\vee$ and $\widetilde C_{\rho,\alpha^\vee}\cap C_{\rho}$. It is thus enough to prove:
\begin{enumerate}
\item $C_{u_{\alpha^\vee}^{\;}\rho u_{\alpha^\vee}^{-1}}\subseteq C_{\rho,\alpha^\vee}$;
\item $\{\alpha^\vee\}\cup \big(\widetilde C_{\rho,\alpha^\vee}\cap C_{\rho}\big)\subseteq C_{u_{\alpha^\vee}^{\;}\rho u_{\alpha^\vee}^{-1}}$.
\end{enumerate}
For $g\in \gp$, write $\rho(g)=t(g)\prod_{\beta^{\vee}\in C_{\rho}}\rho(g)_{\beta^\vee}$ with $t(g)\in \widehat T(E)$ and $\rho(g)_{\beta^\vee}\in U_{\beta^\vee}(E)$ (decomposition (\ref{prod}) for $C=C_{\rho}$), and note that $t(g)=\widehat\chi_{\rho}(g)$. For all $g\in\gp$ we have $u_{\alpha^\vee}t(g)u_{\alpha^\vee}^{-1}\in \widehat T(E)U_{\alpha^\vee}(E)$ and because of the commutation formula \cite[II.1.2(5)]{Ja} and the commutativity of $U_{\alpha^\vee}$, we also have
\begin{equation}\label{chaud}
u_{\alpha^\vee}\Big(\prod_{\beta^{\vee}\in C_{\rho}}\rho(g)_{\beta^\vee}\Big)u_{\alpha^\vee}^{-1}\ \ \in \prod_{\beta^{\vee}\in C_{\rho,\alpha^\vee}\backslash\{\alpha^\vee\}}\!\!\!\!\!\!\!U_{\beta^\vee}(E)\ \ \ \ \forall g\in\gp.
\end{equation}
So $u_{\alpha^\vee}\rho(g)u_{\alpha^\vee}^{-1}\in \widehat T(E) \prod_{\beta^{\vee}\in C_{\rho,\alpha^\vee}}U_{\beta^\vee}(E)$ for all $g\in \gp$ which already implies (i) by the discussion preceding Lemma \ref{combi}. Now if $\beta^\vee\in \widetilde C_{\rho,\alpha^\vee}\cap C_{\rho}$, the commutation formula \cite[II.1.2(5)]{Ja} again shows that $\rho(g)_{\beta^\vee}$ is not touched in the conjugation by elements of $U_{\alpha^\vee}$, that is, $\rho(g)_{\beta^\vee}$ is still the entry of (\ref{chaud}) in the factor $\widehat U_{\beta^\vee}(E)$. Since $\rho(g)_{\beta^\vee}\ne 1$ for some $g\in \gp$ (otherwise $\beta^\vee$ wouldn't be in $C_{\rho}$), this implies $\widetilde C_{\rho,\alpha^\vee}\cap C_{\rho}\subseteq C_{u_{\alpha^\vee}^{\;}\rho u_{\alpha^\vee}^{-1}}$ by the discussion preceding Lemma \ref{combi}. In order to get (ii), it remains to prove that $\alpha^\vee\in C_{u_{\alpha^\vee}^{\;}\rho u_{\alpha^\vee}^{-1}}$. Let us choose an isomorphism $x_{\alpha^\vee}:{\mathbb G}_{\rm a}\buildrel\sim\over\rightarrow U_{\alpha^\vee}$. Then $u_{\alpha^\vee}=x_{\alpha^\vee}(a)$ for some $a\in E^{\times}$ and we compute:
\begin{eqnarray*}
u_{\alpha^\vee}t(g)u_{\alpha^\vee}^{-1}&=&t(g)\big(t(g)^{-1}u_{\alpha^\vee}t(g)\big)u_{\alpha^\vee}^{-1}\\
&=&t(g)\big(t(g)^{-1}x_{\alpha^\vee}(a)t(g)\big)x_{\alpha^\vee}(-a)\\
&=&t(g)x_{\alpha^\vee}\big(\alpha^\vee(t(g)^{-1})a\big)x_{\alpha^\vee}(-a)\\
&=&t(g)x_{\alpha^\vee}\big(a(\alpha^\vee(t(g))^{-1}-1)\big).
\end{eqnarray*}
Since $a\ne 0$ and since there exists $g\in \gp$ such that $\alpha^\vee(t(g))^{-1}=\alpha^\vee(\widehat\chi_{\rho}(g))^{-1}\ne 1$ by assumption, we see that the entry in $U_{\alpha^\vee}(E)$ of $u_{\alpha^\vee}\rho(g) u_{\alpha^\vee}^{-1}$ in the decomposition (\ref{prod}) is nontrivial for some $g\in \gp$, and hence $\alpha^\vee\in C_{u_{\alpha^\vee}^{\;}\rho u_{\alpha^\vee}^{-1}}$.
\end{proof}

\begin{prop}\label{conjbor}
Let $\rho:\gp\rightarrow \widehat{B}(E)$ be a continuous homomorphism and assume that $\alpha^\vee\circ \widehat\chi_{\rho}\ne 1$ for all $\alpha\in R^{+}$. Then there is $b_0\in \widehat{B}(E)$ such that $C_{b_0\rho b_0^{-1}}\subseteq C_{b\rho b^{-1}}$ for all $b\in \widehat{B}(E)$.
\end{prop}
\begin{proof}
Since $\widehat{T}$ normalizes the $\widehat{B}_C$, one can restrict to $b=u\in \widehat U(E)$. Since $C_{u\rho u^{-1}}\subseteq C_{\rho}$ for all $u\in \widehat{U}_{C_{\rho}}(E)$, replacing $\rho$ by a suitable conjugate $u_0\rho u_0^{-1}$ with $u_0\in \widehat{U}_{C_{\rho}}(E)$ we can assume
\begin{equation}\label{minim}
C_{u\rho u^{-1}}= C_{\rho}\ \ \ \ \forall u\in \widehat{U}_{C_{\rho}}(E).
\end{equation}
It is enough to prove that $C_{\rho}\subseteq C_{u\rho u^{-1}}$ for all $u\in \widehat{U}(E)$. By Lemma \ref{combi}, we can enumerate the roots $\alpha_{n}^\vee,\dots,\alpha_{1}^\vee$ in $R^{+\vee}\backslash C_{\rho}$ so that $\alpha_{i}^\vee$ is not in the smallest closed subset containing $C_{\rho}$ and the $\alpha_{j}^\vee$ for $1\leq j\leq i-1$. Since (\ref{prod}) holds for any order, we can write any $u\in \widehat{U}(E)$ as $u=u_nu_{n-1}\cdots u_1u_{C_{\rho}}$, where $u_i\in U_{\alpha_{i}^\vee}(E)$ and $u_{C_{\rho}}\in \widehat U_{C_{\rho}}(E)$. Moreover we have $C_{u_{C_{\rho}}^{\;}\rho u_{C_{\rho}}^{-1}}=C_{\rho}$ by (\ref{minim}). Then it follows from Lemma \ref{conjalp} (and a straightforward induction) that $C_{u\rho u^{-1}}$ is the smallest closed subset containing $C_{\rho}$ and $\{\alpha_i^{\vee} : u_i\ne 1\}$. {\it A fortiori} we thus have $C_{\rho}\subseteq C_{u\rho u^{-1}}$.
\end{proof}

Note that if $\rho$ ordinary is such that $\alpha^\vee\circ \widehat\chi_{\rho}\ne 1$ for all $\alpha\in R^{+}$, then this condition is satisfied by any other conjugate $\rho'$ of $\rho$ taking values in $\widehat{B}(E)$ (as is easily deduced from the fact that $\widehat\chi_{\rho'}$ is conjugate to $\widehat\chi_{\rho}$ in $\widehat G(E)$ by the Jordan decomposition).

\begin{definit}\label{conjgoo}
Let $\rho:\gp\rightarrow \widehat{G}(E)$ be a continuous homomorphism. Assume that $\rho'$ is a conjugate of $\rho$ taking values in $\widehat{B}(E)$ and such that $\alpha^\vee\circ \widehat\chi_{\rho'}\ne 1$ for all $\alpha\in R^{+}$. We say that $\rho'$ is a \emph{good conjugate} (of $\rho$) if $C_{\rho'}\subseteq C_{b\rho' b^{-1}}$ for all $b\in \widehat{B}(E)$.
\end{definit}

By Proposition \ref{conjbor}, good conjugates always exist (under the assumptions of Definition \ref{conjgoo}).

\begin{lem}\label{conjexe}
Let $\rho:\gp\rightarrow \widehat{G}(E)$ be a continuous homomorphism and let $\rho'$ be a good conjugate of $\rho$ as in Definition \ref{conjgoo}. Then any $b\rho' b^{-1}$ for $b\in \widehat{B}_{C_{\rho'}}(E)$ and any $\dot{w}^{-1} \rho' \dot{w}$ for $w\in W_{C_{\rho'}}\subseteq W$ \(see \eqref{wc}\) is a good conjugate of $\rho$. Moreover we have $C_{b\rho' b^{-1}}=C_{\rho'}$ and $C_{\dot{w}^{-1} \rho' \dot{w}}=w^{-1}(C_{\rho'})$.
\end{lem}
\begin{proof}
The statement is obvious for $b\in \widehat{B}_{C_{\rho'}}(E)$. If $w\in W_{C_{\rho'}}$, from the discussion preceding Lemma \ref{combi} and $\dot{w}^{-1}U_{\alpha^\vee}\dot{w}=U_{\dot{w}^{-1}(\alpha^\vee)}$, we already see that $C_{\dot{w}^{-1} \rho' \dot{w}}=w^{-1}(C_{\rho'})$ (we don't need here that $\rho'$ is a good conjugate). From the proof of Proposition \ref{conjbor}, it is enough to have $C_{u\dot{w}^{-1} \rho' \dot{w}u^{-1}}=C_{\dot{w}^{-1} \rho' \dot{w}}$ for all $u\in \widehat{U}_{w^{-1}(C_{\rho'})}(E)$. But $u\dot{w}^{-1} \rho' \dot{w}u^{-1}=\dot{w}^{-1} (\dot{w}u\dot{w}^{-1}) \rho' (\dot{w}u^{-1}\dot{w}^{-1})\dot{w}$ with $\dot{w}u\dot{w}^{-1}\in \widehat{U}_{C_{\rho'}}(E)$ and since $\rho'$ is a good conjugate we have $C_{(\dot{w}u\dot{w}^{-1}) \rho' (\dot{w}u^{-1}\dot{w}^{-1})}=C_{\rho'}$. We deduce
$$C_{u\dot{w}^{-1} \rho' \dot{w}u^{-1}}=w^{-1}\big(C_{(\dot{w}u\dot{w}^{-1}) \rho' (\dot{w}u^{-1}\dot{w}^{-1})}\big)=w^{-1}(C_{\rho'})=C_{\dot{w}^{-1} \rho' \dot{w}},$$
which finishes the proof.
\end{proof}

\begin{prop}\label{conj}
Let $\rho:\gp\rightarrow \widehat{G}(E)$ be a continuous homomorphism and let $\rho',\ \rho''$ be good conjugates of $\rho$ as in Definition \ref{conjgoo}. Then there exist $b\in \widehat{B}_{C_{\rho'}}(E)$ and $w\in W_{C_{\rho'}}$ such that $\rho''=\dot{w}^{-1}(b^{-1}\rho' b) \dot{w}$. In particular, one has $C_{\rho''}=w^{-1}(C_{\rho'})$.
\end{prop}
\begin{proof}
By assumption there is $x\in \widehat{G}(E)$ such that $\rho''(g)=x\rho'(g) x^{-1}$ for all $g\in \gp$. By the Bruhat decomposition $\widehat G(E)=\widehat B(E)W\widehat B(E)$, we can write $x=tu' \dot{w}^{-1} u$ with $t\in \widehat{T}(E)$, $w\in W$ and $u,u'\in \widehat{U}(E)$. As in the proof of Proposition \ref{conjbor}, we use Lemma \ref{combi} to enumerate the roots $\alpha_{n}^\vee,\dots,\alpha_{1}^\vee$ in $R^{+\vee}\backslash C_{\rho'}$ so that $\alpha_{i}^\vee$ is not in the smallest closed subset containing $C_{\rho'}$ and the $\alpha_{j}^\vee$ for $1\leq j\leq i-1$, and we write $u=u_nu_{n-1}\cdots u_1u_{C_{\rho'}}$, where $u_i\in U_{\alpha_{i}^\vee}(E)$ for all $i$ and $u_{C_{\rho'}}\in \widehat U_{C_{\rho'}}(E)$. Replacing $\rho'$ by $u_{C_{\rho'}}\rho' u_{C_{\rho'}}^{-1}$, we can assume $u_{C_{\rho'}}=1$. By Lemma \ref{conjalp} and an obvious induction, $C_{u_n\cdots u_1 \rho' u_1^{-1}\cdots u_n^{-1}}$ is the smallest closed subset containing $C_{\rho'}$ and the $\alpha_{i}^\vee$ such that $u_i\ne 1$. Since $\dot{w}^{-1} (u_n\cdots u_1 \rho' u_1^{-1}\cdots u_n^{-1})\dot{w}$ still takes values in $\widehat{B}(E)$, the discussion preceding Lemma \ref{combi} together with the $\dot{w}^{-1}$-conjugate of (\ref{prod}) show that we must have ${w}^{-1}\big(C_{u_n\cdots u_1 \rho' u_1^{-1}\cdots u_n^{-1}}\big)\subseteq R^{+\vee}$, i.e.\ $w\in W_{C_{\rho'}}$ and $w^{-1}(\alpha_{i}^\vee)\in R^{+\vee}$ for all $i$ such that $u_i\ne 1$. Setting $v_i:=\dot{w}^{-1} u_i\dot{w}$, we have
$$\rho''=t(u'v_n\cdots v_1) (\dot{w}^{-1}\rho' \dot{w})(v_1^{-1}\cdots v_n^{-1}u'^{-1})t^{-1},$$
where $u'v_n\cdots v_1\in \widehat{B}(E)$ and where $\rho''$, $\dot{w}^{-1}\rho' \dot{w}$ are good conjugates of $\rho$ (the latter by Lemma \ref{conjexe}). This already implies $C_{\rho''}=C_{\dot{w}^{-1} \rho' \dot{w}}=w^{-1}(C_{\rho'})$. Writing again $u'v_n\cdots v_1=u'_nu'_{n-1}\cdots u'_1u'_{w^{-1}(C_{\rho'})}$, where $u'_i\in U_{\beta_{i}^\vee}(E)$ and $u'_{w^{-1}(C_{\rho'})}\in \widehat U_{w^{-1}(C_{\rho'})}(E)$ with $\beta_{n}^\vee,\dots,\beta_{1}^\vee\in R^{+\vee}\backslash w^{-1}(C_{\rho'})$ as in Lemma \ref{combi}, we see from Lemma \ref{conjalp} that we must have $u'_i=1$ for all $i$ (otherwise $C_{\rho''}$ would be strictly bigger than $C_{u'_{w^{-1}(C_{\rho'})}\dot{w}^{-1} \rho' \dot{w}u_{w^{-1}(C_{\rho'})}'^{-1}}=w^{-1}(C_{\rho'})$). Setting
$$b:=\dot{w}tu'_{w^{-1}(C_{\rho'})}\dot{w}^{-1}\in \dot{w}\widehat{B}_{w^{-1}(C_{\rho'})}(E)\dot{w}^{-1}=\widehat{B}_{C_{\rho'}}(E)$$
(see the proof of Lemma \ref{autre} for the last equality) gives the statement.
\end{proof}

\subsection{Construction of $\Pi(\rho)^{\ord}$ for $\rho$ generic ordinary}\label{cons}

In this section, we associate to a generic ordinary representation $\rho$ of $\gp$  a represen\-tation $\Pi(\rho)^{\ord}$ of $G(\Qp)$ which is ``modelled'' on the representation $(\LL\vert_{\widehat{B}_{C_{\rho}}})^{\ord}$ of \S\ref{C}. We keep the notation and assumptions of \S\ref{prel} and of \S\ref{Ch} and we denote by $\theta$ a twisting element for $G$.

We fix a good conjugate $\rho$ as in Definition \ref{conjgoo}:
$$\rho:\gp\longrightarrow \widehat{B}_{C_{\rho}}(E)\subseteq \widehat{B}(E)\subseteq \widehat{G}(E),$$
that is, such that the closed subset $C_{\rho}\subseteq R^{+\vee}$ is minimal under conjugation by $\widehat{B}(E)$. We denote by $\chi_{\rho}$ the character (\ref{cft}) on $T(\Qp)$ corresponding to the character $\widehat\chi_{\rho}:\gp\buildrel\rho\over\longrightarrow \widehat{B}_{C_{\rho}}(E)\twoheadrightarrow \widehat{T}(E)$. Note that $\chi_{\rho}$ is unitary since $\gp$ is a compact group and that, from the condition in Definition \ref{conjgoo}, we have $\chi_{\rho}\circ \alpha^{\vee}\ne 1$ for all $\alpha\in R^+$. Recall that $\varepsilon$ is the $p$-adic cyclotomic character that we view as a continuous character $\varepsilon:\Qp^{\times}\rightarrow \oE^{\times}\hookrightarrow E^{\times}$ via class field theory.

\begin{definit}\label{gen}
We say that $\rho$ is \emph{generic} if $\alpha^\vee\circ \widehat\chi_{\rho}\notin \{1,\varepsilon,\varepsilon^{-1}\}$ for all $\alpha\in R^{+}$ (or equivalently all $\alpha\in R$).
\end{definit}

This is equivalent to $\chi_{\rho}\circ \alpha^{\vee}\notin \{1,\varepsilon,\varepsilon^{-1}\}$ for all $\alpha\in R^{+}$. In fact, one could slightly weaken this genericity condition (see e.g.\ Remark \ref{add} below).

\begin{lem}\label{invgen}
Let $\rho'$ be another good conjugate of $\rho$ as in Definition \ref{conjgoo}. Then $\rho$ is generic if and only if $\rho'$ is generic.
\end{lem}
\begin{proof}
By Proposition \ref{conj}, replacing $\rho$ by a conjugate in $\widehat{B}_{C_{\rho}}(E)$ (which doesn't change $\widehat\chi_{\rho}$ nor $C_{\rho}$), there is $w\in W_{C_{\rho}}$ such that $\rho'=\dot w^{-1}\rho\dot w$. Since $\widehat{\chi}_{\rho'}=\dot{w}^{-1}\widehat{\chi}_{\rho}\dot{w}=w^{-1}(\widehat{\chi}_{\rho})$, we have $\alpha^\vee\circ \widehat\chi_{\rho'}=w(\alpha)^{\vee}\circ \widehat\chi_{\rho}$ for $\alpha\in R$ from which the statement is obvious.
\end{proof}

We now assume that our good conjugate $\rho$ is generic.

We let $(\widehat\lambda_{\alpha})_{\alpha\in S}$ be fundamental weights for $\widehat G$ and define $(\LL\vert_{\widehat{B}_{C_{\rho}}})^{\ord}$ as in \S\ref{C}. We also define the following subset of $W$:
\begin{equation}\label{wrho}
W_{\rho}:=\{w\in W : w\Big(\sum_{\alpha\in S}\widehat\lambda_{\alpha}\Big){\rm\ is\ a\ weight\ of\ }(\LL\vert_{\widehat{B}_{C_{\rho}}})^{\ord}\}.
\end{equation}
Of course $W_{C_{\rho}}\subseteq W_{\rho}$ and we can describe $W_{\rho}$ as in Lemma \ref{cw}.

\begin{prop}\label{piI}
Let $w_{C_{\rho}}\in W_{C_{\rho}}$, $I\subseteq w_{C_{\rho}}(S^\vee)\cap C_{\rho}$ a subset of pairwise orthogonal roots and set $J:=w_{C_{\rho}}^{-1}(I)^\vee\subseteq S$. There exists a unique admissible unitary continuous representation $\widetilde\Pi(\rho)_I$ of $G_J(\Qp)$ over $E$ with socle filtration ${\rm Fil}_j\widetilde\Pi(\rho)_I$ such that
$$0={\rm Fil}_{-1}\widetilde\Pi(\rho)_I\subsetneq {\rm Fil}_0\widetilde\Pi(\rho)_I\subsetneq \cdots \subsetneq {\rm Fil}_{\vert I\vert-1}\widetilde\Pi(\rho)_I\subsetneq {\rm Fil}_{\vert I\vert}\widetilde\Pi(\rho)_I=\widetilde\Pi(\rho)_I,$$
where for $j\in \{0,\dots,\vert I\vert\}$
\begin{multline}\label{ouf}
{\rm Fil}_{j}\widetilde\Pi(\rho)_I/{\rm Fil}_{j-1}\widetilde\Pi(\rho)_I\cong \\
\bigoplus_{\substack{I'\subseteq I\\ \vert I'\vert=j}}\Big(\Ind_{B^-(\Qp)\cap G_J(\Qp)}^{G_J(\Qp)}\Big(\big(\prod_{\alpha\in I'^\vee}s_{\alpha}\big)w_{C_{\rho}}\Big)^{-1}(\chi_{\rho})\cdot(\varepsilon^{-1}\circ\theta)\Big)^{{\mathcal C}^0}
\end{multline}
and where we view any character of $T(\Qp)$ as a character of $B^-(\Qp)\cap G_J(\Qp)$ by inflation $B^-(\Qp)\cap G_J(\Qp)\twoheadrightarrow T(\Qp)$.
\end{prop}
\begin{proof}
For $J'\subseteq J$, we set
$$\Pi_{J'}:=\Big(\Ind_{B^-(\Qp)\cap G_J(\Qp)}^{G_J(\Qp)}\Big(\big(\prod_{\alpha\in w_{C_{\rho}}(J')}s_{\alpha}\big)w_{C_{\rho}}\Big)^{-1}(\chi_{\rho})\cdot(\varepsilon^{-1}\circ\theta)\Big)^{{\mathcal C}^0}.$$
\noindent
Step 1: We prove several statements on $\Pi_{J'}$.\\
The set $J$ being a subset of $S$ of pairwise orthogonal roots, by Lemma \ref{strgalpha} we can fix an isomorphism $T'_J\times {\GL}_2^{J}\cong G_J$. Under this isomorphism we have $T'_J \times (\prod_{\beta\in J}T_{\beta})\cong T$, where $T_{\beta}$ is a split maximal torus in the copy of ${\GL}_2$ corresponding to the $\beta$ entry. Letting $\chi'_{\rho,J}:=w_{C_{\rho}}^{-1}(\chi_{\rho})\vert_{T'_J(\Qp)}$ and $\chi_{\rho,J',\beta}:=\big((\prod_{\alpha\in w_{C_{\rho}}(J')}s_{\alpha})w_{C_{\rho}}\big)^{-1}\!(\chi_{\rho})\vert_{T_{\beta}(\Qp)}$ for $J'\subseteq J$ and $\beta\in J$, we can rewrite $\Pi_{J'}$ as
\begin{eqnarray}\label{dirsum}
\nonumber \Pi_{J'}&\!\!\!\!\cong\!\!\!\!&\chi'_{\rho,J}\!\cdot\!(\varepsilon^{-1}\!\circ\theta\vert_{T'_J(\Qp)})\!\otimes_E \!\bigg(\!\Ind_{\smat{*&0\\ *&*}^{J}}^{\Gp^{J}}\!\!\otimes_{\beta\in J}\big(\chi_{\rho,J',\beta}\cdot(\varepsilon^{-1}\!\circ\theta)\vert_{T_{\beta}(\Qp)}\big)\bigg)^{{\mathcal C}^0}\\
&\!\!\!\!\cong\!\!\!\!&\chi'_{\rho,J}\!\cdot\!(\varepsilon^{-1}\!\circ\theta\vert_{T'_J(\Qp)})\!\otimes_E \!\Big(\!\widehat\otimes_{\beta\in J}\Big(\!\Ind_{\smat{*&0\\ *&*}}^{\Gp}\!\chi_{\rho,J',\beta}\cdot(\varepsilon^{-1}\!\circ\theta)\vert_{T_{\beta}(\Qp)}\Big)^{{\mathcal C}^0}\Big),
\end{eqnarray}
where $\widehat\otimes$ is the completed tensor product in the category of $p$-adic Banach spaces. We have
\begin{multline}\label{beta}
\chi_{\rho,J',\beta}=\Big(\big(\prod_{\alpha\in w_{C_{\rho}}(J')}s_{w_{C_{\rho}}^{-1}(\alpha)}\big)w_{C_{\rho}}^{-1}\Big)(\chi_{\rho})\vert_{T_{\beta}(\Qp)}=\\
\left\{ \begin{array}{lcc}
w_{C_{\rho}}^{-1}(\chi_{\rho})\vert_{T_{\beta}(\Qp)}\ &{\rm if}&\ \beta\notin J'\\
(s_{\beta}w_{C_{\rho}}^{-1})(\chi_{\rho})\vert_{T_{\beta}(\Qp)}\ &{\rm if}&\ \beta\in J'
\end{array}\right.
\end{multline}
and also
$$\big(\chi_{\rho,J',\beta}\cdot(\varepsilon^{-1}\circ\theta)\vert_{T_{\beta}(\Qp)}\big)\circ \beta^{\vee}=(\chi_{\rho,J',\beta}\circ \beta^{\vee})\cdot\big((\varepsilon^{-1}\circ\theta)\circ \beta^{\vee}\big)=(\chi_{\rho}\circ w(\beta)^{\vee})\cdot \varepsilon^{-1},$$
where $w:=\big(\prod_{\alpha\in w_{C_{\rho}}(J')}s_{\alpha}\big)w_{C_{\rho}}\in W_{\rho}$ and where we use $\langle \theta,\beta^{\vee}\rangle=1$ (as $\beta\in J\subseteq S$). Since $\rho$ is generic it follows from Definition \ref{gen} that this is never the trivial character (of $\Qp^{\times}$) and from Proposition \ref{fu}, we get that the representation (\ref{dirsum}) is always topologically irreducible. If $J'_1,J'_2$ are two distinct subsets in $J$, we have $\big((\prod_{\alpha\in w_{C_{\rho}}(J'_1)}s_{\alpha})w_{C_{\rho}}\big)^{-1}\!(\chi_{\rho})\ne \big((\prod_{\alpha\in w_{C_{\rho}}(J'_2)}s_{\alpha})w_{C_{\rho}}\big)^{-1}\!(\chi_{\rho})$. Indeed, assume, say, there exists $\beta$ such that $\beta\notin J'_1$, $\beta\in J'_2$. By (\ref{beta}) we have $\chi_{\rho,J_1',\beta}=w_{C_{\rho}}^{-1}(\chi_{\rho})\vert_{T_{\beta}(\Qp)}$ and $\chi_{\rho,J_2',\beta}=s_{\beta}(w_{C_{\rho}}^{-1}(\chi_{\rho})\vert_{T_{\beta}(\Qp)})$ and, since $T_{\beta}$ is a split maximal torus in $\GL_2$, it suffices to have $w_{C_{\rho}}^{-1}(\chi_{\rho})\circ \beta^{\vee}\ne 1$ or equivalently $\chi_{\rho}\circ w_{C_{\rho}}(\beta)^{\vee}\ne 1$. But this follows from Definition \ref{gen}. By Theorem \ref{classic}(iii), we get $\Pi_{J'_1}\ncong \Pi_{J'_2}$. All the $\Pi_{J'}$ have scalar endomorphisms and are residually of finite length, as follows from (\ref{dirsum}) and Proposition \ref{fu}.

\noindent
Step 2: We prove the existence of a representation as in the statement.\\
If $\beta\in J$ and $\chi_{\beta}:T_{\beta}(\Qp)\rightarrow \oE^{\times}\subseteq E^{\times}$ is a unitary continuous character, define
$$\Pi_{\beta}(\chi_{\beta}):=\Big(\Ind_{\smat{*&0\\ *&*}}^{\Gp}\chi_{\beta}\cdot(\varepsilon^{-1}\circ\theta)\vert_{T_{\beta}(\Qp)}\Big)^{{\mathcal C}^0}.$$
Consider the following admissible unitary continuous representation of $T'_J(\Qp)\times \Gp^{J}$:
$$\widetilde\Pi(\rho)_I:= \chi'_{\rho,J}\cdot(\varepsilon^{-1}\circ\theta\vert_{T'_J(\Qp)})\otimes_E \big(\widehat\otimes_{\beta\in J}\mathcal E_{\beta}\big),$$
where \ the \ $\Gp$-representation \ $\mathcal E_{\beta}$ \ is \ the \ unique \ non-split \ extension \ of $\Pi_{\beta}(s_{\beta}(w_{C_{\rho}}^{-1}\!(\chi_{\rho})\vert_{T_{\beta}(\Qp)}))$ by $\Pi_{\beta}(w_{C_{\rho}}^{-1}(\chi_{\rho})\vert_{T_{\beta}(\Qp)})$, see Proposition \ref{ch}. From Proposition \ref{fu} its constituents are
$$\chi'_{\rho,J}\cdot(\varepsilon^{-1}\circ\theta\vert_{T'_J(\Qp)})\otimes_E \big(\widehat\otimes_{\beta\in J}\Pi_{\beta}(\chi_{\beta})\big)$$
for $\chi_{\beta}\in \{w_{C_{\rho}}^{-1}(\chi_{\rho})\vert_{T_{\beta}(\Qp)},s_{\beta}(w_{C_{\rho}}^{-1}\!(\chi_{\rho})\vert_{T_{\beta}(\Qp)})\}$ and we see from (\ref{dirsum}) and (\ref{beta}) that \ they \ are \ exactly \ the \ $\Pi_{J'}$. \ Now \ let \ $\beta\in J$ \ and \ choose \ $\chi_{\beta'}\in \{w_{C_{\rho}}^{-1}(\chi_{\rho})\vert_{T_{\beta'}(\Qp)},s_{\beta'}(w_{C_{\rho}}^{-1}\!(\chi_{\rho})\vert_{T_{\beta'}(\Qp)})\}$ for all $\beta'\in J\backslash \{\beta\}$. The $\Gp^{J}$-repre\-sentation
$$\mathcal E_{\beta}\widehat\otimes_E \big(\widehat\otimes_{\beta'\in J\backslash \{\beta\}}\Pi_{\beta'}(\chi_{\beta'})\big)$$
(a subquotient of $\widetilde\Pi(\rho)_I$) is still a {\it non-split} extension between its $2$ irreducible constituents by Lemma \ref{penible3} applied inductively to $V_1\!:=\!\mathcal E_{\beta}\widehat\otimes_E \big(\widehat\otimes_{\beta'\in J'}\Pi_{\beta'}(\chi_{\beta'})\big)$ and $\Pi_2:=\Pi_{\beta''}(\chi_{\beta''})$ for $J'\subsetneq J\backslash \{\beta\}$ and $\beta''\in J\backslash (\{\beta\}\amalg J')$ (all the assumptions are satisfied by Proposition \ref{fu} and by Step 1). It is then easy to check that the socle filtration of $\widetilde\Pi(\rho)_I$ is exactly as in the statement. More generally, the same argument shows that, for any $J_1\subseteq J_2\subseteq J$, the representation
\begin{multline}\label{rab}
\chi'_{\rho,J}\cdot(\varepsilon^{-1}\circ\theta\vert_{T'_J(\Qp)})\otimes_E \big(\widehat\otimes_{\beta\in J\backslash J_2}\Pi_{\beta}(w_{C_{\rho}}^{-1}(\chi_{\rho})\vert_{T_{\beta}(\Qp)})\big)\widehat\otimes_E \\
\big(\widehat\otimes_{\beta\in J_1}\Pi_{\beta}(s_{\beta}(w_{C_{\rho}}^{-1}(\chi_{\rho})\vert_{T_{\beta}(\Qp)}))\big)\widehat\otimes_E \big(\widehat\otimes_{\beta\in J_2\backslash J_1}\mathcal E_{\beta}\big)
\end{multline}
has socle filtration with graded pieces $\oplus_{\substack{J_1\subseteq J'\subseteq J_2\\ \vert J'\vert=j}}\Pi_{J'}$.

\noindent
Step 3: We finally prove unicity.\\
Let $J_1\subseteq J_2\subseteq J$, we prove the following statement by induction on $\vert J_2\backslash J_1\vert$: there is (up to isomorphism) at most one representation of $G_J(\Qp)$ such that the graded pieces of its socle filtration are $\oplus_{\substack{J_1\subseteq J'\subseteq J_2\\ \vert J'\vert=j}}\Pi_{J'}$. Denote by $\Pi$ such a representation. Note first that by Lemma \ref{penible3} (applied inductively with $G_2 = \GL_2$ together with Proposition \ref{fu}, Proposition \ref{ch}, and Lemma \ref{penible2}), a constituent $\Pi_{J'}$ has a (unique) non-split extension with another (distinct) constituent $\Pi_{J''}$ if and only if either $J''\subsetneq J'$ and $\vert J'\backslash J''\vert=1$ or $J'\subsetneq J''$ and $\vert J''\backslash J'\vert=1$. If $J''=J'\amalg \{\beta_0\}$, this implies (together with the fact that all constituents are distinct by Step 1) that $\Pi$ has a unique subquotient with constituents $\Pi_{J'}$ and $\Pi_{J''}$. Let us prove by induction on $\vert J'\backslash J_1\vert$ that this subquotient is necessarily a non-split extension of $\Pi_{J''}$ by $\Pi_{J'}$. If $\vert J'\backslash J_1\vert=0$, this is obviously the case, otherwise $\Pi_{J''}$ would be in a lower socle layer of $\Pi$. Assume $\vert J'\backslash J_1\vert\geq 1$, and suppose that the above subquotient is {\it not} a non-split extension of $\Pi_{J''}$ by $\Pi_{J'}$. Then the only other possibility (in view of the socle filtration of $\Pi$) is $\Pi_{J'}\oplus \Pi_{J''}$. Let $\beta_1\in J'$ be such that the unique non-split extension of $\Pi_{J''}$ by $\Pi_{J''\backslash\{\beta_1\}}$ occurs as a subquotient of $\Pi$ ($\beta_1$ necessarily exists because otherwise $\Pi_{J''}$ would be in the socle of $\Pi$). By induction (applied to $J'\backslash \{\beta_1\}$), $\Pi$ has a (unique) subquotient which is the unique non-split extension of $\Pi_{J''\backslash\{\beta_1\}}$ by $\Pi_{J'\backslash \{\beta_1\}}$ (namely the representation (\ref{rab}) replacing $J_1$ by $J'\backslash \{\beta_1\}$ and $J_2$ by $J''\backslash\{\beta_1\}$). Hence $\Pi$ has a unique subquotient $\Pi'$ with socle $\Pi_{J'\backslash \{\beta_1\}}$, cosocle $\Pi_{J''}$, and unique intermediate constituent $\Pi_{J''\backslash\{\beta_1\}}$. (Here we used that there is no extension between $\Pi_{J''}$ and $\Pi_{J'}$.) But it follows from Lemma \ref{penible3}(i) by induction that such a representation $\Pi'$ doesn't exist. We now prove the unicity statement at the beginning of Step 3. When $\vert J_2\backslash J_1\vert=0$, the result is trivial and when $\vert J_2\backslash J_1\vert= 1$ it follows from Proposition \ref{ch}(i). Assume it is true when $\vert J_2\backslash J_1\vert=n$ and let us prove it for $\vert J_2\backslash J_1\vert=n+1\geq 2$. Let $\widetilde\Pi(\rho)_{J_1,J_2}$ be such a representation and fix $J'_2$ such that $J_1\subseteq J'_2\subsetneq J_2$ and $J_2\backslash J'_2=\{\beta_0\}$. By what we have seen above, there is a (unique) subrepresentation of $\widetilde\Pi(\rho)_{J_1,J_2}$ with constituents all the $\Pi_{J''}$ for $J_1\subseteq J''\subseteq J'_2$ and a (unique) quotient of $\widetilde\Pi(\rho)_{J_1,J_2}$ with constituents the $\Pi_{J'}$ for $J'_1\subseteq J'\subseteq J_2$ where $J'_1:=J_1\amalg \{\beta_0\}$. Since the socle filtration induces the socle filtration on subrepresentations, we already have by induction that this subrepresentation is necessarily $\widetilde\Pi(\rho)_{J_1,J'_2}$. To prove (by induction) that this quotient is $\widetilde\Pi(\rho)_{J'_1,J_2}$, we need to check that the graded pieces of its socle filtration are $\oplus_{\substack{J'_1\subseteq J'\subseteq J_2\\ \vert J'\vert=j}}\Pi_{J'}$. But this easily follows from the fact that any subquotient (of this quotient) with two constituents $\Pi_{J'}$, $\Pi_{J''}$, where $J' \subsetneq J''$ and $\vert J''\backslash J'\vert=1$, is a non-split extension between them (see above). Now set $\Pi_1:=\Pi_{\beta_0}(w_{C_{\rho}}^{-1}(\chi_{\rho})\vert_{T_{\beta_0}(\Qp)})$, $\Pi'_1:=\Pi_{\beta_0}(s_{\beta_0}(w_{C_{\rho}}^{-1}(\chi_{\rho})\vert_{T_{\beta_0}(\Qp)}))$ and
\begin{multline*}
\Pi_2:= \chi'_{\rho,J}\cdot(\varepsilon^{-1}\circ\theta\vert_{T'_J(\Qp)})\otimes_E\big(\widehat\otimes_{\beta\in J\backslash J_2}\Pi_{\beta}(w_{C_{\rho}}^{-1}(\chi_{\rho})\vert_{T_{\beta}(\Qp)})\big)\widehat\otimes_E \\
\big(\widehat\otimes_{\beta\in J_1}\Pi_{\beta}(s_{\beta}(w_{C_{\rho}}^{-1}(\chi_{\rho})\vert_{T_{\beta}(\Qp)}))\big)\widehat\otimes_E \big(\widehat\otimes_{\beta\in J'_2\backslash J_1}\mathcal E_{\beta}\big).
\end{multline*}
By (\ref{rab}) we must have $\widetilde\Pi(\rho)_{J_1,J'_2}\cong \Pi_1\widehat\otimes_E \Pi_2$ and $\widetilde\Pi(\rho)_{J'_1,J_2}\cong \Pi'_1\widehat\otimes_E \Pi_2$. Moreover, since $\vert J_2\backslash J'_2\vert=1$, any $J$ between $J_1$ and $J_2$ is necessarily either between $J_1$ and $J'_2$ or between $J'_1$ and $J_2$. This implies that $\widetilde\Pi(\rho)_{J_1,J_2}$ is an extension of $\widetilde\Pi(\rho)_{J'_1,J_2}$ by $\widetilde\Pi(\rho)_{J_1,J'_2}$. This extension is non-split otherwise it wouldn't have the right socle. But Lemma \ref{penible3} tells us that such a non-split extension is necessarily unique. (Note that in order to deal with the assumption $\dim_E{\rm Ext}^1_{G_2}(\Pi_2,\Pi_2)<\infty$ in (ii) of this lemma, we actually need to apply Lemma \ref{penible3} inductively replacing $\Pi_2$ first by $\chi'_{\rho,J}\cdot(\varepsilon^{-1}\circ\theta\vert_{T'_J(\Qp)})$, then $\Pi_1$, $\Pi'_1$ by $\Pi_1\otimes \chi'_{\rho,J}\cdot(\varepsilon^{-1}\circ\theta\vert_{T'_J(\Qp)})$, $\Pi'_1\otimes \chi'_{\rho,J}\cdot(\varepsilon^{-1}\circ\theta\vert_{T'_J(\Qp)})$ and $\Pi_2$ by $\Pi_{\beta}(w_{C_{\rho}}^{-1}(\chi_{\rho}))$ for some $\beta\in J\backslash J_2$ etc., we leave the easy details to the reader.) This proves the unicity of $\widetilde\Pi(\rho)_{J_1,J_2}$, and in particular of $\widetilde\Pi(\rho)_I=\widetilde\Pi(\rho)_{\varnothing,J}$.
\end{proof}

\begin{rem}\label{add}
(i) It follows from the proof of Proposition \ref{piI} that one could replace the condition in Definition \ref{gen} by the two conditions: $\chi_{\rho}\circ \alpha^{\vee}\ne 1$ for $\alpha\in R^{+}$ and $\chi_{\rho}\circ \alpha^{\vee}\ne \varepsilon$ for $\alpha\in w(S)$ and $w\in W_{\rho}$ (one can check that Lemma \ref{invgen} still holds with this weaker condition of genericity).\\
(ii) From the proof of Proposition \ref{piI} (in particular Step 3), one also gets that the representation $\widetilde\Pi(\rho)_I$ is rigid, that is, its cosocle filtration equals its socle filtration (up to numbering). In fact, the proof shows that the submodule structure of $\widetilde\Pi(\rho)_I$ is given by a ``hypercube''.
\end{rem}

If $I'\subseteq I$ and $J':=w_{C_{\rho}}^{-1}(I')^\vee\subseteq J$, then $(B^-\cap G_J)G_{J'}$ is a parabolic subgroup of $G_J$ with Levi $G_{J'}$ and we easily get from (\ref{rab}) (with the same notation as in the previous proof, see in particular Step 3 for $\widetilde\Pi(\rho)_{\varnothing,J'}$),
\begin{equation}\label{j'}
\widetilde\Pi(\rho)_{\varnothing,J'}\cong \big(\Ind_{(B^-(\Qp)\cap G_J(\Qp))G_{J'}(\Qp)}^{G_J(\Qp)}\widetilde\Pi(\rho)_{I'}\big)^{{\mathcal C}^0},
\end{equation}
where we consider $\widetilde\Pi(\rho)_{I'}$ by inflation as an (admissible) unitary continuous representation of $(B^-(\Qp)\cap G_J(\Qp))G_{J'}(\Qp)$ and where the parabolic induction is analogous to that in (\ref{ball}) (see \cite[\S4.1]{Em2} for details). For $I$, $J$ as in Proposition \ref{piI}, $B^-G_J\subseteq G$ is the standard parabolic subgroup of $G$ containing $B^-$ with Levi subgroup $G_J$. We set
$$\Pi(\rho)_I:=\big(\Ind_{B^-(\Qp)G_J(\Qp)}^{G(\Qp)}\widetilde\Pi(\rho)_I\big)^{{\mathcal C}^0},$$
where $\widetilde\Pi(\rho)_I$ is seen as a (unitary continuous) representation of $B^-(\Qp)G_J(\Qp)$ by inflation. The representation $\Pi(\rho)_I$ of $G(\Qp)$ is admissible unitary continuous of finite length (by Proposition \ref{piI} and Theorem \ref{classic}). Moreover by \cite[Thm.\ 4.4.6]{Em2} and \cite[Cor.\ 4.3.5]{Em2} we have for $I'\subseteq I$,
\begin{multline*}
\Hom_{G(\Qp)}(\Pi(\rho)_{I'},\Pi(\rho)_I)\cong \\
\Hom_{G_J(\Qp)}\Big(\big(\Ind_{(B^-(\Qp)\cap G_J(\Qp))G_{J'}(\Qp)}^{G_J(\Qp)}\widetilde\Pi(\rho)_{I'}\big)^{{\mathcal C}^0},\widetilde\Pi(\rho)_I\Big)
\end{multline*}
and we deduce from (\ref{j'}) and the proof of Proposition \ref{piI},
$$\Hom_{G(\Qp)}(\Pi(\rho)_{I'},\Pi(\rho)_I)\cong \Hom_{G_J(\Qp)}\big(\widetilde\Pi(\rho)_{\varnothing,J'},\widetilde\Pi(\rho)_I\big)=E$$
(for the last equality, note that any $G_J(\Qp)$-equivariant morphism $\widetilde\Pi(\rho)_{\varnothing,J'}\rightarrow \widetilde\Pi(\rho)_I$ is necessarily scalar since $\widetilde\Pi(\rho)_{\varnothing,J'}$ is contained in $\widetilde\Pi(\rho)_I$ and these two representations have the same irreducible socle, $\Pi_{\varnothing}$, which has scalar endomorphisms and appears only once). We can then proceed as in (\ref{ii}) fixing a compatible system of injections $\Pi(\rho)_{I'}\hookrightarrow \Pi(\rho)_I$ for $I'\subseteq I$ (the choice of which won't matter) and define in the abelian category of admissible unitary continuous representations of $G(\Qp)$ over $E$ the inductive limit
$$\Pi(\rho)_{C_{\rho},w_{C_{\rho}}}:=\ilim I{\Pi(\rho)_I},$$
where $I$ runs among the subsets of $w_{C_{\rho}}(S^\vee)\cap C_{\rho}$ of pairwise orthogonal roots. Finally we set
$$\Pi(\rho)^{\ord}:=\oplus_{w_{C_{\rho}}\in W_{C_{\rho}}}\Pi(\rho)_{C_{\rho},w_{C_{\rho}}}.$$
This is an admissible unitary continuous representation of $G(\Qp)$ of finite length (see \S\ref{ques} for more on its constituents).

\begin{lem}\label{invpi}
Let $\rho'$ be another good conjugate of $\rho$ as in Definition \ref{conjgoo}. Then $\Pi(\rho)^{\ord}\cong \Pi(\rho')^{\ord}$.
\end{lem}
\begin{proof}
By Proposition \ref{conj}, replacing $\rho$ by a conjugate in $\widehat{B}_{C_{\rho}}(E)$ (which doesn't change $\chi_{\rho}$ nor $C_{\rho}$ and hence doesn't change $\Pi(\rho)^{\ord}$), there is $w\in W_{C_{\rho}}$ such that $\rho'=\dot w^{-1}\rho\dot w$. We have ${\chi}_{\rho'}=w^{-1}(\chi_{\rho})$ (see the proof of Lemma \ref{invgen}) and it is easy to check we also have $C_{\rho'}=w^{-1}(C_{\rho})$ and $W_{C_{\rho'}}=w^{-1}W_{C_{\rho}}$. Let $w_{C_{\rho'}}\in W_{C_{\rho'}}$, $I'\subseteq w_{C_{\rho'}}(S^\vee)\cap C_{\rho'}$ a subset of pairwise orthogonal roots and set $w_{C_{\rho}}:=ww_{C_{\rho'}}\in W_{C_{\rho}}$ and $I:=w(I')\subseteq w_{C_{\rho}}(S^\vee)\cap C_{\rho}$ (also a subset of pairwise orthogonal roots). Then one has $\widetilde\Pi(\rho')_{I'}=\widetilde\Pi(\rho)_{w(I')}$. Indeed, they are representations of the same $G_J(\Qp)$ and by Proposition \ref{piI} it suffices to check they have the same constituents in the socle filtration, which is immediate as, for any $I''\subseteq I'$,
\begin{eqnarray*}
\Big(\big(\prod_{\alpha\in I''^\vee}s_{\alpha}\big)w_{C_{\rho'}}\Big)^{-1}(\chi_{\rho'})&=& \Big(\big(\prod_{\alpha\in I''^\vee}s_{w_{C_{\rho'}}^{-1}(\alpha)}\big)w_{C_{\rho'}}^{-1}\Big)(\chi_{\rho'})\\
&=&\Big(\big(\prod_{\alpha\in I''^\vee}s_{w_{C_{\rho}}^{-1}(w(\alpha))}\big)w_{C_{\rho}}^{-1}w\Big)(w^{-1}(\chi_{\rho}))\\
&=&\Big(\big(\prod_{\alpha\in w(I'')^\vee}s_{\alpha}\big)w_{C_{\rho}}\Big)^{-1}(\chi_{\rho}).
\end{eqnarray*}
We deduce $\Pi(\rho')_{I'}=\Pi(\rho)_{w(I')}$, $\Pi(\rho')_{C_{\rho'},w_{C_{\rho'}}}=\Pi(\rho)_{C_{\rho},w_{C_{\rho}}}$ and finally $\Pi(\rho)^{\ord}= \Pi(\rho')^{\ord}$.
\end{proof}

Therefore $\Pi(\rho)^{\ord}$ doesn't depend on the choice of a good conjugate in the sense of Definition \ref{conjgoo}. Using Proposition \ref{piI} and arguing as in the proof of Proposition \ref{trans} below, one can check more generally that $\Pi(\rho)^{\ord}$ only depends on the conjugacy class of $\rho$ in $\widehat G(E)$ and not on the choice of the Borel $\widehat B$ (we leave the details to the reader).

\begin{rem}\label{rema3}
If $G=G_1\times G_2$, $\rho=\rho_1\oplus \rho_2$ with $\rho_i:\gp\longrightarrow \widehat{G}_i(E)$ generic ordinary ($i=1,2$) and $w_{C_{\rho}}=(w_{C_{\rho_1}},w_{C_{\rho_2}})\in W_{C_{\rho}}=W_{C_{\rho_1}}\times W_{C_{\rho_2}}\subseteq W=W_1\times W_2$, one easily checks that $\Pi(\rho)_{C_{\rho},w_{C_{\rho}}}=\Pi(\rho_1)_{C_{\rho_1},w_{C_{\rho_1}}}\widehat\otimes_E\Pi(\rho_2)_{C_{\rho_2},w_{C_{\rho_2}}}$ and that one has $\Pi(\rho)^{\ord}=\Pi(\rho_1)^{\ord}\widehat \otimes_E\Pi(\rho_2)^{\ord}$ (where we index by $i$ everything related to $G_i$), see Remark \ref{rema}(ii) and Remark \ref{rema2}(ii).
\end{rem}

\subsection{Variant mod $p$}\label{variant2}

We construct the representation $\Pi(\rhobar)^{\ord}$ of $G(\Qp)$ over $k_E$ associated to a sufficiently generic ordinary representation $\rhobar$ of $\gp$ over $k_E$.

We first modify the setting of \S\ref{prel} so as to deal with characteristic $p$. In this section, $G/\Zp$ is a connected split reductive algebraic group over $\Zp$, $T\subseteq G$ a split maximal torus over $\Zp$ and we let as usual $\big(X(T),R,X^{\vee}(T),R^{\vee}\big)$ be the root datum of $G$. We fix a choice $S\subseteq R$ of simple roots, we let $R^+\subseteq R$ be the positive roots, $B\subseteq G$ (resp.\ $B^-\subseteq G$) the Borel subgroup over $\Zp$ corresponding to $R^+$ (resp.\ $-R^+$) and $(\widehat{G},\widehat{B},\widehat{T})$ the dual triple of $(G,B,T)$ over $\oE$. We moreover assume that both $G$ and $\widehat G$ have a connected centre and denote by $\theta$ a twisting element for $G$. We also assume that $p$ is large enough so that the following lemma holds (so this is no restriction if $G = \GL_n$).

\begin{lem}\label{borelp}
Suppose that $p > 3$ or that $p$ is a good prime for $G$ \(Definition \ref{goodprime}\). Let $B'\subseteq B$ be a Zariski closed algebraic subgroup containing $T$. Then there exists a closed subset $C\subseteq R^+$ such that $B'=B_C$.
\end{lem}
\begin{proof}
  The proof of Lemma \ref{borel} still holds by \cite[\S2.5]{BT} (even under a slightly weaker condition on $p$).
\end{proof}

Let ${\rhobar}:\gp\longrightarrow \widehat{G}(k_E)$ be a continuous homomorphism. When $\rhobar$ takes values in $\widehat{B}(k_E)\subseteq\widehat{G}(k_E)$ we say $\rhobar$ is {\it ordinary}. For any ordinary $\rhobar$, we let $C_{\rhobar}\subseteq R^{+\vee}$ be the closed subset of roots such that $\widehat{B}_{C_{\rhobar}}$ is the smallest closed subgroup of $\widehat{B}$ containing $\widehat{T}$ such that $\rhobar$ takes values in $\widehat{B}_{C_{\rhobar}}(k_E)$ ($C_{\rhobar}$ exists thanks to Lemma \ref{borelp}). For $\rhobar$ ordinary, we define $\widehat\chi_{\rhobar}:\gp\buildrel\rhobar\over\longrightarrow \widehat{B}(k_E)\twoheadrightarrow \widehat{T}(k_E)$.

If $\alpha^{\vee}\circ \widehat\chi_{\rhobar}\ne 1$ for all $\alpha\in R^{+}$, the same proof as in \S\ref{Ch} shows that, conjugating $\rhobar$ by an element in $\widehat{B}(k_E)$ if necessary, it is always possible to assume that $C_{\rhobar}$ is minimal under conjugation by elements in $\widehat{B}(k_E)$ (indeed, the fact that $E$ has characteristic $0$ is never used in \S\ref{Ch}). As in Definition \ref{conjgoo}, we call ``good conjugates'' ordinary $\rhobar$ such that $C_{\rhobar}$ is minimal, and as in Proposition \ref{conj}, we can show that any two good conjugates $\rhobar$ and $\rhobar'$ are related by $\rhobar'=\dot{w}^{-1}(b^{-1}\rhobar b) \dot{w}$ for some $b\in \widehat{B}_{C_{\rhobar}}(k_E)$ and some $w\in W_{C_{\rhobar}}$.

\begin{definit}\label{genbar}
An ordinary $\rhobar$ is \emph{generic} if $\alpha^{\vee}\circ \widehat\chi_{\rhobar}\notin \{1,\omega,\omega^{-1}\}$ for all $\alpha\in R^{+}$ (or equivalently all $\alpha\in R$).
\end{definit}

This is equivalent to $\chi_{\rhobar}\circ \alpha^{\vee}\notin \{1,\omega,\omega^{-1}\}$ for all $\alpha\in R^{+}$, where $\chi_{\rhobar}$ corresponds to $\widehat\chi_{\rhobar}$ as in (\ref{cft}).

From now on we fix a good conjugate $\rhobar:\gp\rightarrow \widehat{B}_{C_{\rhobar}}(k_E)\subseteq \widehat{B}(k_E)\subseteq \widehat{G}(k_E)$ that is generic. Then Proposition \ref{piI} still holds replacing $\rho$ by $\rhobar$ and $E$ by $k_E$ and allows to define a finite length admissible smooth representation $\widetilde\Pi(\rhobar)_I$ of $G_J(\Qp)$ over $k_E$ (for $I\subseteq w_{C_{\rhobar}}(S^\vee)\cap C_{\rhobar}$ a subset of pairwise orthogonal roots and $J=w_{C_{\rhobar}}^{-1}(I)^\vee$, where $w_{C_{\rhobar}}\in W_{C_{\rhobar}}$) with socle filtration as in Proposition \ref{piI}. Indeed, the proof of this proposition is essentially based on the results of Appendix A and Appendix B, but all these results hold replacing $E$ by $k_E$ and ``unitary continuous'' by ``smooth'' (and are easier to prove over $k_E$ than over $E$!). We then proceed exactly as in \S\ref{cons} and define $\Pi(\rhobar)_I:=\Ind_{B^-(\Qp)G_J(\Qp)}^{G(\Qp)}\widetilde\Pi(\rhobar)_I$ and $\Pi(\rhobar)_{C_{\rhobar},w_{C_{\rhobar}}}:=\ilim I{\Pi(\rhobar)_I}$, where $I$ runs among the subsets of $w_{C_{\rhobar}}(S^\vee)\cap C_{\rhobar}$ of pairwise orthogonal roots. But now, arguing as in the proof of Theorem \ref{classic}(ii) and using \cite[Cor.\ 4.3.5]{Em2} (and the genericity of $\rhobar$), we {\it know} that the socle filtration ${\rm Fil}_j\Pi(\rhobar)_{C_{\rhobar},w_{C_{\rhobar}}}$ of $\Pi(\rhobar)_{C_{\rhobar},w_{C_{\rhobar}}}$ is such that for $j\in \Z_{\geq 0}$,
\begin{multline*}
{\rm Fil}_{j}\Pi(\rhobar)_{C_{\rhobar},w_{C_{\rhobar}}}/{\rm Fil}_{j-1}\Pi(\rhobar)_{C_{\rhobar},w_{C_{\rhobar}}}\cong \\
\bigoplus_{\substack{I\subseteq w_{C_{\rhobar}}(S^\vee)\cap C_{\rhobar}\\ \vert I\vert=j}}\Ind_{B^-(\Qp)}^{G(\Qp)}\Big(\big(\prod_{\alpha\in I^\vee}s_{\alpha}\big)w_{C_{\rhobar}}\Big)^{-1}(\chi_{\rhobar})\cdot(\omega^{-1}\circ\theta)
\end{multline*}
($I$ being a subset of pairwise orthogonal roots). Finally, we set
$$\Pi(\rhobar)^{\ord}:=\oplus_{w_{C_{\rhobar}}\in W_{C_{\rhobar}}}\Pi(\rhobar)_{C_{\rhobar},w_{C_{\rhobar}}}$$
and check as in Lemma \ref{invpi} that $\Pi(\rhobar)^{\ord}$ doesn't depend on which good conjugate we choose. Also, the genericity of $\rhobar$ implies that all the irreducible constituents of $\Pi(\rhobar)^{\ord}$ are distinct.

We now give the link between $\Pi(\rhobar)^{\ord}$ and certain Serre weights of $\rhobar$. Let $X_1(T):=\{\lambda\in X(T) : 0\leq \langle\lambda,\alpha^\vee\rangle \leq p-1\ \forall \alpha\in S\}$. Since $G^{\rm der}$ is simply connected, Serre weights for $G(\Fp)$ are exactly given by the $k_E$-representations
$$F(\lambda)\vert_{G(\Fp)},$$
where $\lambda\in X_1(T)$ and $F(\lambda)$ is the unique (absolutely) irreducible representation of the algebraic group $G\times_{\Zp}k_E$ of highest weight $\lambda$. We have $F(\lambda)\vert_{G(\Fp)}\cong F(\mu)\vert_{G(\Fp)}$ if and only if $\lambda-\mu\in (p-1)X^0(T)$ (for all this see e.g.\ \cite[\S3.1]{He1}). In the sequel, we often write $F(\lambda)$ instead of $F(\lambda)\vert_{G(\Fp)}$ for a Serre weight.

\begin{definit}\label{ordser}
The set of \emph{ordinary Serre weights} of $\rhobar$ is the set of irreducible constituents of $\soc_{G(\Zp)}\big(\Pi(\rhobar)^{\ord}\vert_{G(\Zp)}\big)$ (up to isomorphism).
\end{definit}

Equivalently, it is the union for $w_{C_{\rhobar}}\in W_{C_{\rhobar}}$ (forgetting possible multiplicities) of the irreducible constituents of $\soc_{G(\Zp)}(\Pi(\rhobar)_{C_{\rhobar},w_{C_{\rhobar}}}\vert_{G(\Zp)})$.

We now make more explicit the ordinary Serre weights of $\rhobar$ in a special but important case.

\begin{definit}\label{genbarfor}
We say that $\rhobar$ is \emph{inertially generic} if $(\alpha^{\vee}\circ \widehat\chi_{\rhobar})\vert_{\ip}\notin \{1,\omega,\omega^{-1}\}$ for all $\alpha\in R^{+}$ (or equivalently all $\alpha\in R$).
\end{definit}

Again, this is equivalent to $(\chi_{\rhobar}\circ \alpha^{\vee})\vert_{\Zp^{\times}}\notin \{1,\omega,\omega^{-1}\}$ for all $\alpha\in R^{+}$. Of course, $\rhobar$ inertially generic implies $\rhobar$ generic. Note that the existence of inertially generic $\rhobar$ implies that $p$ is large enough so that Lemma~\ref{borelp} holds: if $G$ is not a torus then one has $p>3$ (otherwise all powers of $\omega$ belong to $\{1,\omega,\omega^{-1}\}$). When $G=\GL_n$, one can check that there are inertially generic $\rhobar$ if and only if $p>2n$.

Since any continuous character $\Zp^{\times}\rightarrow k_E^{\times}$ is an (integral) power of the reduction mod $p$ on $\Zp^{\times}$, one easily checks that, for any continuous character $\chi:T(\Qp)\rightarrow k_E^{\times}$ such that $(\chi\circ \alpha^{\vee})\vert_{\Zp^{\times}}\ne 1$ for all $\alpha\in S$, there exists $\lambda_{\chi}\in X_1(T)$ uniquely determined modulo $(p-1)X^0(T)$ such that $\overline{\lambda_{\chi}\vert_{T(\Zp)}}=\chi\vert_{T(\Zp)}$. One then has $\langle \lambda_{\chi},\alpha^\vee\rangle\in \{1,\dots,p-2\}$ for all $\alpha\in S$. If moreover $(\chi\circ \alpha^{\vee})\vert_{\Zp^{\times}}\notin \{1,\omega,\omega^{-1}\}$ for all $\alpha\in S$, one has $\langle \lambda_{\chi},\alpha^\vee\rangle\in \{2,\dots,p-3\}$.

\begin{prop}\label{ordserex}
Assume $\rhobar$ is ordinary and inertially generic. Then $\lambda_{w_{C_{\rhobar}}^{-1}(\chi_{\rhobar})}-\theta\in X_1(T)$ for all $w_{C_{\rhobar}}\in W_{C_{\rhobar}}$ and the set of ordinary Serre weights of $\rhobar$ is
\begin{align}
\{F(\lambda_{w_{C_{\rhobar}}^{-1}(\chi_{\rhobar})}-\theta): w_{C_{\rhobar}}\in W_{C_{\rhobar}}\}\qquad\qquad\qquad\qquad\qquad\qquad\qquad\label{eq:2} \\
\qquad\qquad\qquad = \{F(\lambda_{\chi_{\o{\rho_1}}}-\theta): \text{$\o{\rho_1}$ is an ordinary conjugate of $\rhobar$}\}.\label{eq:3}
\end{align}
Moreover, the Serre weights in~\eqref{eq:2} are distinct.
\end{prop}
\begin{proof}
We first claim that for all $w\in W$, we have $\lambda_{w^{-1}(\chi_{\rhobar})} - \theta \in X_1(T)$. Note that $(w^{-1}(\chi_{\rhobar})\circ \alpha^\vee)\vert_{\Zp^{\times}} =(\chi_{\rhobar}\circ w(\alpha)^\vee)\vert_{\Zp^{\times}}\notin \{1,\omega,\omega^{-1}\}$ for $\alpha\in S$, as $\rhobar$ is inertially generic. Hence $\langle \lambda_{w^{-1}(\chi_{\rhobar})}-\theta,\alpha^\vee\rangle =\langle \lambda_{w^{-1}(\chi_{\rhobar})},\alpha^\vee\rangle-1\in \{1,\dots,p-4\}$ ($\alpha\in S$), which proves the claim.

Next we show that the Serre weights $F(\lambda_{w^{-1}(\chi_{\rhobar})}-\theta)$ for $w\in W$ are distinct. If $F(\lambda_{w_1^{-1}(\chi_{\rhobar})}-\theta) \cong
  F(\lambda_{w_2^{-1}(\chi_{\rhobar})}-\theta)$ for distinct elements $w_1$, $w_2$ of $W$, then $w_1^{-1}
  \lambda_{\chi_{\rhobar}} \equiv w_2^{-1} \lambda_{\chi_{\rhobar}} \pmod {(p-1)X(T)}$ since these weights induce the same
  character $T(\Zp) \to k_E^\times$.  This implies that $\lambda_{\chi_{\rhobar}} - w_1 w_2^{-1} \lambda_{\chi_{\rhobar}}$ is
  an element of $(p-1)X(T)$, but it is also clearly contained in $\Z R$.  As the centre of $G$ is connected, it follows that
  it is even an element of $(p-1)\Z R$. In other words, $\lambda_{\chi_{\rhobar}}$ is fixed by a non-trivial element of the
  affine Weyl group $(p-1)\Z R \rtimes W$. By \cite[Prop.\ V.3.3.1]{Bou}, $\lambda_{\chi_{\rhobar}}$ is fixed by an affine
  reflection, i.e.\ $\langle \lambda_{\chi_{\rhobar}}, \alpha^\vee\rangle \in (p-1)\Z$ or equivalently $\chi_{\rhobar} \circ
  \alpha^\vee = 1$. But this contradicts the inertial genericity of $\rhobar$.

We now prove that $F(\lambda_{w_{C_{\rhobar}}^{-1}(\chi_{\rhobar})}-\theta)$ is the $G(\Zp)$-socle of $\Pi(\rhobar)_{C_{\rhobar},w_{C_{\rhobar}}}$. First, it follows from \cite[(2.13)]{He2} that $F(\lambda_{\chi})$ is the $G(\Zp)$-socle of $\Ind_{B^-(\Qp)}^{G(\Qp)}\chi$ for any $\chi:T(\Qp)\rightarrow k_E^{\times}$ such that $(\chi\circ \alpha^{\vee})\vert_{\Zp^{\times}}\ne 1$ $\forall\alpha\in S$. In particular $F(\lambda_{w^{-1}(\chi_{\rhobar})}-\theta)\cong \soc_{G(\Zp)}\!\big(\!\Ind_{B^-(\Qp)}^{G(\Qp)}\!w^{-1}(\chi_{\rhobar})\cdot\omega^{-1}\circ\theta\big)$ for any $w\in W$. Since the Serre weights $F(\lambda_{w^{-1}(\chi_{\rhobar})}-\theta)$ are all distinct, an obvious d\'evissage on the constituents of $\Pi(\rhobar)_{C_{\rhobar},w_{C_{\rhobar}}}$ shows that $\dim_{k_E}\Hom_{G(\Zp)}(F(\lambda_{w^{-1}(\chi_{\rhobar})}-\theta),\Pi(\rhobar)_{C_{\rhobar},w_{C_{\rhobar}}})\in \{0,1\}$. Moreover, when the dimension is $1$, it follows from Corollary \ref{prat} below that the Hecke algebra ${\mathcal H}_{G}\big(F(\lambda_{w^{-1}(\chi_{\rhobar})}-\theta)\big)$ of \S\ref{locres} acts on $\Hom_{G(\Zp)}(F(\lambda_{w^{-1}(\chi_{\rhobar})}-\theta),\Pi(\rhobar)_{C_{\rhobar},w_{C_{\rhobar}}})$ via an ordinary character (see Definition \ref{deford}, since the dimension is $1$ there is no place for any other character). From Corollary \ref{prat} again we deduce that restriction to the $G(\Zp)$-socle induces an isomorphism for all $w\in W$,
\begin{multline*}
\Hom_{G(\Qp)}\Big(\Ind_{B^-(\Qp)}^{G(\Qp)}\!w^{-1}(\chi_{\rhobar})\cdot\omega^{-1}\circ\theta,\Pi(\rhobar)_{C_{\rhobar},w_{C_{\rhobar}}}\Big)\buildrel\sim\over\rightarrow \\
\Hom_{G(\Zp)}\big(F(\lambda_{w^{-1}(\chi_{\rhobar})}-\theta),\Pi(\rhobar)_{C_{\rhobar},w_{C_{\rhobar}}}\big).
\end{multline*}
But from the socle filtration of $\Pi(\rhobar)_{C_{\rhobar},w_{C_{\rhobar}}}$, we know that the left-hand side is nonzero (of dimension $1$) if and only if $w=w_{C_{\rhobar}}$.

Finally, to see the equality of (\ref{eq:2}) and (\ref{eq:3}), note that if $\rhobar \cong \o{\rho_1}$ are both ordinary, then by the Bruhat decomposition there is a $w \in W$ such that $\dot w^{-1} \rhobar \dot w$ is ordinary (which implies $w\in W_{C_{\rhobar}}$) and $w^{-1}(\chi_{\rhobar}) = \chi_{\o{\rho_1}}$.
\end{proof}

\begin{ex}\label{exord}
Assume $G$ is ${\GL}_n$, $T$ is the torus of diagonal matrices and $B$ the upper triangular matrices, then we can identify $X(T)$ (resp.\ $X_1(T)$) with $n$-tuples $\lambda:=(\lambda_1,\dots,\lambda_n)\in \Z^n$ (resp.\ $n$-tuples $\lambda:=(\lambda_1,\dots,\lambda_n)\in \Z^n$ such that $0\leq \lambda_i-\lambda_{i+1}\leq p-1$). Choose $\theta:=(n-1,n-2,\dots,0)\in X(T)$ as a twisting element. Then the set of ordinary Serre weights of $\rhobar$ (inertially generic ordinary) in the sense of Definition \ref{ordser} coincides with the set of $F(\lambda)$ for $\lambda\in X_1(T)$ such that $\rhobar$ has a conjugate $\overline{\rho_1}:\gp\rightarrow \widehat{B}(k_E)=B(k_E)$ with
\begin{eqnarray}\label{lift}
\overline{\rho_1}\vert_{\ip} = \begin{pmatrix}\omega^{\lambda_1+(n-1)}&*&\cdots &*\\
0&\omega^{\lambda_2+(n-2)}&\varddots &\vdots\\
\vdots & \varddots & \varddots & *\\
0&\dots & 0 & \omega^{\lambda_n}\end{pmatrix}.
\end{eqnarray}
Alternatively, $F(\lambda)$ for $\lambda \in X_1(T)$ is an ordinary Serre weight of $\rhobar$ if and only if a conjugate of $\rhobar$ admits an upper-triangular
crystalline lift $\rho$ such that $\wh\chi_\rho\vert_{\ip} = \diag(\varepsilon^{\lambda_1+(n-1)}, \cdots, \varepsilon^{\lambda_n})$ (use \cite[Lem.\ 3.1.5]{GG}).
\end{ex}

\begin{rem}
For $\rhobar$ ordinary, semi-simple and sufficiently generic, the set of ordinary Serre weights of $\rhobar$ is also the set of Serre weights $F(\lambda)$ of \cite[Prop.\ 6.28]{He1} for which $\rhobar\vert_{\ip}\cong \tau(1,\lambda+\theta)$ in the notation of {\it loc.\ cit.} (note that $G$ is assumed here to be ${\GL}_n$ but the statement can easily be extended to $G$ as in the present paper).
\end{rem}

\subsection{Some questions}\label{ques}

We state some questions on the representations $\Pi(\rho)^{\ord}$, $\Pi(\rhobar)^{\ord}$, $(\LL\vert_{\widehat{B}_{C_{\rho}}})^{\ord}$ and $(\LLbar\vert_{\widehat{B}_{C_{\rhobar}}})^{\ord}$.

Start first with $\rho:\gp\longrightarrow \widehat{B}_{C_{\rho}}(E)\subseteq \widehat{B}(E)\subseteq \widehat{G}(E)$ such that the closed subset $C_{\rho}\subseteq R^{+\vee}$ is minimal under conjugation by $\widehat{B}(E)$ (i.e.\ $\rho$ is a good conjugate) and such that $\rho$ is generic, and let $\Pi(\rho)^{\ord}$ be as in \S\ref{cons}. Conjecture \ref{folk} and the genericity of $\rho$ imply that, for $w_{C_{\rho}}\in W_{C_{\rho}}$, the irreducible constituents of $\Pi(\rho)_{C_{\rho},w_{C_{\rho}}}$ should be exactly the principal series
\begin{equation}\label{const}
\Big(\Ind_{B^-(\Qp)}^{G(\Qp)}\Big(\big(\prod_{\alpha\in I^\vee}s_{\alpha}\big)w_{C_{\rho}}\Big)^{-1}(\chi_{\rho})\cdot(\varepsilon^{-1}\circ\theta)\Big)^{{\mathcal C}^0}
\end{equation}
for $I$ running among the subsets of $w_{C_{\rho}}(S^\vee)\cap {C_{\rho}}$ of pairwise orthogonal roots. Note that this is indeed true if $(\overline{\chi_{\rho}\circ \alpha^{\vee})\cdot\varepsilon^{-1}}$ is never the trivial character of $\Qp^{\times}$ in $k_E^{\times}$ for all $w\in W_{\rho}$ and all $\alpha\in w(S)$ (see (\ref{wrho}), Remark \ref{add}(i) and Theorem \ref{classic}(ii)). More precisely, one can conjecture the following (compare with Lemma \ref{cw}):

\begin{conj}\label{conjloc}
There exists a unique admissible unitary continuous representation $\Pi(\rho)_{C_{\rho},w_{C_{\rho}}}$ of $G(\Qp)$ over $E$ with socle filtration $0={\rm Fil}_{-1}\Pi(\rho)_{C_{\rho},w_{C_{\rho}}}\subsetneq {\rm Fil}_0\Pi(\rho)_{C_{\rho},w_{C_{\rho}}}\subseteq \cdots$ such that for $j\in \Z_{\geq 0}$,
\begin{multline*}
{\rm Fil}_{j}\Pi(\rho)_{C_{\rho},w_{C_{\rho}}}/{\rm Fil}_{j-1}\Pi(\rho)_{C_{\rho},w_{C_{\rho}}}\cong \\
\bigoplus_{\substack{I\subseteq w_{C_{\rho}}(S^\vee)\cap C_{\rho}\\ \vert I\vert=j}}\Big(\Ind_{B^-(\Qp)}^{G(\Qp)}\Big(\big(\prod_{\alpha\in I^\vee}s_{\alpha}\big)w_{C_{\rho}}\Big)^{-1}(\chi_{\rho})\cdot(\varepsilon^{-1}\circ\theta)\Big)^{{\mathcal C}^0}
\end{multline*}
for $I$ running among the subsets of $w_{C_{\rho}}(S^\vee)\cap {C_{\rho}}$ of pairwise orthogonal roots.
\end{conj}

Hauseux has very recently announced a proof of this conjecture (assuming the irreducibility of the representations in (\ref{const})), see \cite{Ha2}. One can also ask the stronger question if there exists a unique admissible unitary continuous representation of $G(\Qp)$ over $E$ with socle $(\Ind_{B^-(\Qp)}^{G(\Qp)}w_{C_{\rho}}^{-1}(\chi_{\rho})\cdot(\varepsilon^{-1}\circ\theta))^{{\mathcal C}^0}$ and (multiplicity free) constituents given by (\ref{const}) (that is, we only fix the socle and the constituents instead of the whole socle filtration). Let us also mention the following recent theorem due to Hauseux (which was raised as an open question in the first version of this paper).

\begin{thm}[\cite{Ha1}]\label{hau}
Let $\chi:T(\Qp)\rightarrow \oE^{\times}\subseteq E^{\times}$ unitary continuous. If $\chi':T(\Qp)\rightarrow \oE^{\times}\subseteq E^{\times}$ is unitary continuous, we have 
$${\rm Ext}^1_{G(\Qp)}\Big(\big(\Ind_{B^-(\Qp)}^{G(\Qp)}\chi'\cdot(\varepsilon^{-1}\circ\theta)\big)^{{\mathcal C}^0}, \big(\Ind_{B^-(\Qp)}^{G(\Qp)}{\chi}\cdot(\varepsilon^{-1}\circ\theta)\big)^{{\mathcal C}^0}\Big)\ne 0$$
if and only if $\chi'=\chi$ or $\chi'=s_{\alpha}(\chi)$ for $\alpha\in S$. If moreover $\chi\circ \alpha^{\vee}\ne 1$ for every $\alpha\in R^+$, then the above ${\rm Ext}^1_{G(\Qp)}$ with $\chi'=s_{\alpha}(\chi)$ \($\alpha\in S$\) has dimension $1$.
\end{thm}

\begin{rem}\label{rememer}
Emerton informed us that, at least assuming $\chi\circ \alpha^{\vee}\ne 1$ for every $\alpha\in R^+$, he always expects a nonzero element in 
$${\rm Ext}^{\ell(w)}_{G(\Qp)}\Big(\big(\Ind_{B^-(\Qp)}^{G(\Qp)}w(\chi)\cdot(\varepsilon^{-1}\circ\theta)\big)^{{\mathcal C}^0}, \big(\Ind_{B^-(\Qp)}^{G(\Qp)}{\chi}\cdot(\varepsilon^{-1}\circ\theta)\big)^{{\mathcal C}^0}\Big),$$ 
where $\ell(w)$ is the length of $w$ in the Weyl group $W$ (\cite[\S II.1.5]{Ja}). Note that $\ell(w)=1$ when $w=s_{\alpha}$ if and only if $\alpha\in S$. 
\end{rem}

The underlying hope behind this paper is that there may exist a direct ``functorial'' \ link \ between \ the \ $G(\Qp)$-representation \ $\Pi(\rho)^{\ord}$ \ of \ \S\ref{cons} \ and \ the $\gp$-representation
$$(\LL\vert_{\widehat{B}_{C_{\rho}}})^{\ord}\circ \rho,$$
where $(\LL\vert_{\widehat{B}_{C_{\rho}}})^{\ord}$ is the algebraic $\widehat{B}_{C_{\rho}}$-representation of \S\ref{C}. More precisely, one can ask whether, for any choice of fundamental weights $(\widehat\lambda_{\alpha})_{\alpha\in S}\in X(\widehat T)^{|S|}$ and $(\lambda_{\alpha})_{\alpha\in S}\in X(T)^{|S|}$ as in \S\ref{fundamental}, there exists a covariant additive functor $F$, generalizing that of \cite{Co} in the case of $G={\GL}_2$, from the abelian category of finite length admissible unitary continuous representations of $G(\Qp)$ over $E$ to the abelian category of finite-dimensional continuous representations of $\gp$ over $E$ satisfying (at least) the following properties:
\begin{enumerate}
\item for any unitary continuous $\chi:T(\Qp)\rightarrow \oE^{\times}\subseteq E^{\times}$ one has
$$F\Big(\big(\Ind_{B^-(\Qp)}^{G(\Qp)}{\chi}\cdot(\varepsilon^{-1}\circ\theta)\big)^{{\mathcal C}^0}\Big)=\big(\sum_{\alpha\in S}\wh\lambda_{\alpha}\big)\circ\widehat\chi,$$
where $\widehat\chi$ is the character of $\gp$ corresponding to $\chi$ in \S\ref{prel} and $\theta:=\sum_{\alpha\in S}\lambda_{\alpha}$;
\item $F$ is exact when restricted to the full subcategory of representations such that all their irreducible constituents are subquotients of unitary continuous principal series of $G(\Qp)$ over $E$;
\item $F(\Pi(\rho)^{\rm ord})=(\LL\vert_{\widehat{B}_{C_{\rho}}})^{\ord}\circ \rho$.
\end{enumerate}

Assume that $p$ is a good prime for $G$ and consider $\rhobar:\gp\longrightarrow \widehat{B}_{C_{\rhobar}}(k_E)\subseteq \widehat{B}(k_E)\subseteq \widehat{G}(k_E)$ such that the closed subset $C_{\rhobar}\subseteq R^{+\vee}$ is minimal under conjugation by $\widehat{B}(k_E)$ and such that $\rhobar$ is generic. One can ask all the previous questions in the setting of \S\ref{variant2} replacing $\Pi(\rho)^{\ord}$ (resp.\ $\Pi(\rho)_{C_{\rho},w_{C_{\rho}}}$) by $\Pi(\rhobar)^{\ord}$ (resp.\ $\Pi(\rhobar)_{C_{\rhobar},w_{C_{\rhobar}}}$) and $(\LL\vert_{\widehat{B}_{C_{\rho}}})^{\ord}$ by $(\LLbar\vert_{\widehat{B}_{C_{\rhobar}}})^{\ord}$. Note that Theorem \ref{hau} also holds over $k_E$ (see the proof of \cite[Thm.\ 5.2.3]{Ha1}) and that one can prove the statement in Remark \ref{rememer} over $k_E$ modulo a conjecture of Emerton (\cite[Conj.\ 3.7.2]{Em3}): see \cite[Thm.\ 5.3.2]{Ha1}. Finally see \cite{Ha2} for a proof of the mod $p$ analogue of Conjecture \ref{conjloc} and \cite{Br7} for the construction of a functor $F$ on mod $p$ representations which satisfies the mod $p$ analogue of properties (i)--(iii) above.

\section{Local-global compatibility results and conjectures}\label{mainsection}

We conjecture that the representations $\Pi(\rho)^{\ord}$ or $\Pi(\rhobar)^{\ord}$ of $G(\Qp)$ defined in \S\ref{cons} and \S\ref{variant2} occur (under suitable conditions) in spaces of automorphic forms for unitary groups that are compact at infinity and split at places above $p$. Over $k_E$ we prove a weak form of this conjecture.

\subsection{The global setting}\label{set}

We define the relevant spaces of automorphic forms and state some of their properties.

We let $F^+$ be a finite totally real extension of $\Q$ with ring of integers $\oFF$ and $F$ a totally imaginary quadratic extension of $F$ with ring of integers $\oF$. We denote by $c$ the non-trivial element of $\Gal(F/F^+)$. If $v$ (resp.\ $w$) is a finite place of $F^+$ (resp.\ $F$), we let $F_v^+$ (resp.\ $F_w$) be the completion of $F^+$ (resp.\ $F$) at $v$ (resp.\ $w$) and $\oFFv$ (resp.\ $\oFw$) the ring of integers of $F_v^+$ (resp.\ $F_w)$. If $v$ splits in $F$ and $w,w^c$ are the two places of $F$ above $v$, we have $\oFFv=\oFw\buildrel c \over \simeq {\mathcal O}_{\!F_{w^c}}$, where the last isomorphism is induced by $c$. We let $\oFFp:=\oFF\otimes_{\Z}\Zp=\prod_{v\vert p}\oFFv$ and ${\mathbb A}_{F^+}^{\infty}$ (resp.\ ${\mathbb A}_{F^+}^{\infty,p}$) denote the finite ad\`eles of $F^+$ (resp.\ the finite ad\`eles of $F^+$ outside $p$). Finally we assume that all places of $F^+$ above $p$ split in $F$ and, for each $v\vert p$ in $F^+$, we choose one place $\tilde v\in \{w,w^c\}$ of $F$ above $v$ (this choice won't be important).

We let $n\in \Z_{>1}$, $N$ a positive integer prime to $p$ and $G$ a connected reductive algebraic group over $\oFF[1/N]$ satisfying the following conditions:
\begin{enumerate}
\item there is an isomorphism $\iota:G\times_{\oFF[1/N]} \oF[1/N]\buildrel\sim\over\longrightarrow {\GL_n}_{/\oF[1/N]}$;
\item $G\times_{\oFF[1/N]} F^+$ is an outer form of ${\GL_n}_{/F^+}$;
\item $G\times_{\oFF[1/N]} F^+$ is quasi-split at all finite places of $F^+$;
\item $G\times_{\oFF[1/N]} F^+$ is isomorphic to ${\rm U}_n(\R)$ at all infinite places of $F^+$.
\end{enumerate}

It is easy to see that such groups exist (cf.\ e.g.\ \cite[\S7.1.1]{EGH}). Condition (i) implies that if $v$ is any finite place of $F^+$ that splits in $F$ and if $w\vert v$ in $F$ the isomorphism $\iota$ induces $\iota_w:G(F_v^+)\buildrel\sim\over\rightarrow {\GL_n}(F_w)$ which restricts to an isomorphism still denoted by $\iota_w:G(\oFFv)\buildrel\sim\over\rightarrow {\GL_n}(\oFw)$ if $v$ doesn't divide $N$. Condition (ii) implies that $c\circ \iota_w:G(F_v^+)\buildrel\sim\over\rightarrow {\GL_n}(F_{w^c})$ (resp.\ $c\circ \iota_w:G(\oFFv)\buildrel\sim\over\rightarrow {\GL_n}({\mathcal O}_{\!F_{w^c}})$ if $v$ doesn't divide $N$) is conjugate in ${\GL_n}(F_{w^c})$ (resp.\ in ${\GL_n}({\mathcal O}_{\!F_{w^c}})$) to $\tau^{-1}\circ \iota_{w^c}$, where $\tau$ is the transpose in ${\GL_n}(F_{w^c})$ (resp.\ in ${\GL_n}({\mathcal O}_{\!F_{w^c}})$). Conditions (iii) and (iv) force $n[F^+:\Q]$ to be divisible by $4$ when $n$ is even.

If $U$ is any compact open subgroup of $\GAp\times G(\oFFp)$ and $M$ any $\oE$-module endowed with an $\oE$-linear action of $G(\oFFp)$, we let $S(U,M)$ be the $\oE$-module of functions
$$f:G(F^+)\backslash \GA\longrightarrow M$$
such that $f(gu)=u_p^{-1}(f(g))$, where $g\in \GA$ and $u_p\in G(\oFFp)$ is the projection of $u\in U\subseteq \GAp\times G(\oFFp)$. If $M$ is a finite type $\oE$-module, then so is $S(U,M)$. We also define
\begin{equation*}
S(U^p,M):=\ilim{U_p}{S(U^pU_p,M)},
\end{equation*}
where $U^p$ is a (fixed) compact open subgroup of $\GAp$ and where $U_p$ runs among compact open subgroups of $G(\oFFp)$. We can identify $S(U^p,M)$ with functions $f:G(F^+)\backslash \GA/U^p\rightarrow M$ for which there exists a compact open subgroup $U_p$ of $G(\oFFp)$ such that $f(gu)=u_p^{-1}(f(g))$ for $g\in \GA$ and $u\in U_p$. When the action of $G(\oFFp)$ on $M$ is trivial, we endow  $S(U^p,M)$ with a linear left action of $G(F^+\otimes_{\Q}\Qp)$ by $(hf)(g):=f(gh)$ ($h\in G(F^+\otimes_{\Q}\Qp)$, $g\in \GA$).

If $U$ is any compact open subgroup of $\GAp\times G(\oFFp)$, following \cite[\S7.1.2]{EGH} we say that $U$ is {\it unramified} at a finite place $v$ of $F^+$ which splits in $F$ and doesn't divide $N$ if we have $U=U^{v}\times G(\oFFv)$, where $U^{v}$ is a compact open subgroup of $\GAv$. Note that a compact open subgroup of $\GAp\times G(\oFFp)$ is unramified at all but a finite number of finite places of $F^+$. If $U$ is a compact open subgroup of $\GAp\times G(\oFFp)$ and $\Sigma$ a finite set of finite places of $F^+$ containing the set of places of $F^+$ that split in $F$ and divide $pN$ and the set of places of $F^+$ that split in $F$ at which $U$ is {\it not} unramified, we denote by $\TT^\Sigma:=\oE[T^{(j)}_w]$ the commutative polynomial $\oE$-algebra generated by formal variables $T^{(j)}_w$ for $j\in \{1,\dots,n\}$ and $w$ a place of $F$ lying above a finite place of $F^+$ that splits in $F$ and {\it doesn't} belong to $\Sigma$. The algebra $\TT^\Sigma$ acts on $S(U,M)$ by making $T^{(j)}_w$ act by the double coset
$$\iota_w^{-1}\left[{\GL_n}(\oFw)\begin{pmatrix}{\bf 1}_{n-j}&\\& \varpi_w{\bf 1}_{j}\end{pmatrix}{\GL_n}(\oFw)\right],$$
where $\varpi_w$ is a uniformizer in $\oFw$. Explicitly, if we write
$${\GL_n}(\oFw)\smat{{\bf 1}_{n-j}&\\& \varpi_w{\bf 1}_{j}}{\GL_n}(\oFw)=\coprod_{i}g_i\smat{{\bf 1}_{n-j}&\\& \varpi_w{\bf 1}_{j}}{\GL_n}(\oFw),$$
we have for $f\in S(U,M)$ and $g\in \GA$,
$$(T^{(j)}_wf)(g):=\sum_if\bigg(g\iota_w^{-1}\Big(g_i\smat{{\bf 1}_{n-j}&\\& \varpi_w{\bf 1}_{j}}\Big)\bigg).$$
One checks that $T^{(j)}_{w^c} = (T_w^{(n)})^{-1}T^{(n-j)}_{w}$ on $S(U,M)$. If $S$ is any $\TT^\Sigma$-module and $I$ any ideal of $\TT^\Sigma$, we set $S[I]:=\{x\in S : Ix=0\}$.

For a compact open subgroup $U^p$ of $\GAp$, we now focus on the spaces
$$S(U^p,\oE),\ \ \ S(U^p,\oE/\pE^n),\ \ \ S(U^p,k_E).$$
In particular, $S(U^p,\oE)$ is a free $\oE$-module and $S(U^p,\oE)\otimes_{\oE}\oE/\pE^n=S(U^p,\oE/\pE^n)$. These spaces are admissible smooth representations of $G(F^+\otimes_{\Q}\Qp)$ since their subspaces of $U_p$-invariant vectors for any $U_p$ are $S(U^pU_p,\oE)$, $S(U^pU_p,\oE/\pE^n)$ or $S(U^pU_p,k_E)$. We also consider the $p$-adic completion of $S(U^p,\oE)$, that is
\begin{equation*}
\widehat S(U^p,\oE):=\plim{n}{\big(S(U^p,\oE)/\pE^n\big)}\cong \plim{n}{S(U^p,\oE/\pE^n)},
\end{equation*}
which has the induced action of $G(F^+\otimes_{\Q}\Qp)$. If we tensor $\widehat S(U^p,\oE)$ by $E$, we obviously get an admissible unitary continuous representation $\widehat S(U^p,E)$ of $G(F^+\otimes_{\Q}\Qp)$ over $E$ with $\widehat S(U^p,\oE)$ as unit ball.

If $\Sigma$ a finite set of finite places of $F^+$ containing the set of places of $F^+$ that split in $F$ and divide $pN$ and the set of places of $F^+$ that split in $F$ and at which $U^p$ is not unramified, the algebra $\TT^\Sigma$ acts on $S(U^pU_p,\oE)$, $S(U^pU_p,\oE/\pE^n)$, $S(U^pU_p,k_E)$ for any $U_p$ and thus also on $S(U^p,\oE)$, $S(U^p,\oE/\pE^n)$, $S(U^p,k_E)$, $\widehat S(U^p,\oE)$ and $\widehat S(U^p,E)$. This action commutes with that of $G(F^+\otimes_{\Q}\Qp)$. Therefore, if $I$ is any ideal of $\TT^\Sigma$, $\widehat S(U^p,E)[I]$ is an admissible unitary continuous representation of $G(F^+\otimes_{\Q}\Qp)$ over $E$ which is a closed subrepresentation of $\widehat S(U^p,E)$.

If $\m^\Sigma$ is a maximal ideal of $\TT^{\Sigma}$ with residue field $k_E$, we can define the localized subspaces $S(U^pU_p,k_E)_{\m^\Sigma}$ and their inductive limit
$$\ilim{U_p}{S(U^pU_p,k_E)_{\m^\Sigma}}=S(U^p,k_E)_{\m^\Sigma},$$
which inherits an induced (admissible smooth) action of $G(F^+\otimes_{\Q}\Qp)$. We have $S(U^pU_p,k_E)[\m^{\Sigma}]\!\subseteq S(U^pU_p,k_E)_{\m^\Sigma}\subseteq S(U^pU_p,k_E)$ and thus inclusions of admissible smooth $G(F^+\otimes_{\Q}\Qp)$-representations
$$S(U^p,k_E)[\m^{\Sigma}]\subseteq S(U^p,k_E)_{\m^\Sigma}\subseteq S(U^p,k_E).$$
Moreover, as representations of $G(F^+\otimes_{\Q}\Qp)$, $S(U^p,k_E)_{\m^\Sigma}$ is a direct summand of $S(U^p,k_E)$ (the maximal vector subspace on which the elements of $\m^{\Sigma}$ act nilpotently). Finally, if ${U'}^p\subseteq U^p$ is a smaller compact open subgroup, if $\Sigma'\supseteq \Sigma$ is as above with respect to ${U'}^p$ and if $\m^{\Sigma'}:=\m^{\Sigma}\cap \TT^{\Sigma'}$ is the unique maximal ideal of $\TT^{\Sigma'}\subseteq \TT^{\Sigma}$ with residue field $k_E$ that is contained in $\m^{\Sigma}$, then we have $S(U^p,k_E)\subseteq S({U'}^p,k_E)$, $S(U^p,k_E)_{\m^\Sigma}\subseteq S({U'}^p,k_E)_{\m^{\Sigma'}}$ and $S(U^p,k_E)[\m^\Sigma]\subseteq S({U'}^p,k_E)[\m^{\Sigma'}]$.

\subsection{The conjectures}\label{globconj}

We state our local-global compatibility conjectures. We keep the setting and notation of \S\ref{set} and assume moreover that $p$ splits in $F^+$.

We start with the $p$-adic case.

Let $r:\gF\rightarrow {\GL}_n(E)$ be a continuous representation. We assume:
\begin{enumerate}
\item $r$ is ramified only at a finite number of places of $F$;
\item $r^c\cong r^\vee\otimes\varepsilon^{1-n}$ (where $r^c(g):=r(cgc)$ for $g\in \gF$);
\item $r$ is an absolutely irreducible representation of $\gF$.
\end{enumerate}
Let $U^p\subseteq \GAp$ be a compact open subgroup and $\Sigma$ a finite set of finite places of $F^+$ containing the set of places of $F^+$ that split in $F$ and divide $pN$, the set of places of $F^+$ that split in $F$ at which $U^p$ is not unramified {\it and} the set of places of $F^+$ that split in $F$ at which $r$ is ramified. We associate to $r$ and $\Sigma$ the prime ideal $\p^\Sigma$ in $\TT^\Sigma$ generated by all elements
$$\Big((-1)^j{\rm Norm}(w)^{j(j-1)/2}T_w^{(j)}-a^{(j)}_w\Big)_{j,w},$$
where $j\in \{1,\dots,n\}$, $w$ is a place of $F$ lying above a finite place of $F^+$ that splits in $F$ and doesn't belong to $\Sigma$, ${\rm Norm}(w)$ is the cardinality of the residue field at $w$ and where $X^n+a_w^{(1)}X^{n-1}+\cdots + a_w^{(n-1)}X+a_w^{(n)}$ is the characteristic polynomial of $r(\Frob_w)$ (an element of $\oE[X]$).

For a finite place $w$ of $F$, we denote by $r_w$ the restriction of $r$ to a decomposition subgroup at $w$. We assume finally:
\begin{enumerate}
\setcounter{enumi}{3}
\item $r_w$ is generic ordinary (in the sense of \S\ref{Ch} and Definition \ref{gen}) for all places $w$ of $F$ above $p$.
\end{enumerate}
Recall this means that $r_w$ is upper triangular (up to conjugation) and its diagonal characters $(\chi_{w,j})_{1\leq j\leq n}$ satisfy $\chi_{w,i}\chi_{w,j}^{-1}\notin \{1,\varepsilon,\varepsilon^{-1}\}$ for $i\ne j$. By the relation $r^c\cong r^\vee\otimes\varepsilon^{1-n}$ this is equivalent to $r_{\tilde v}$ being generic ordinary for all places $\tilde v$ where $v\vert p$ in $F^+$. In particular we can define the $G(F_v^+)$-representation $\Pi(r_{\tilde v})^{\ord}$, where $G(F_v^+)$ acts via $\iota_{\tilde v}:G(F_v^+)\buildrel\sim\over\rightarrow {\GL_n}(F_{\tilde v})=\GL_n(\Qp)$ and where we choose $\theta:=(n-1,n-2,\dots,0)$ as twisting element and the upper triangular matrices in ${\GL_n}(F_{\tilde v})$ as Borel subgroup (as in Example \ref{exord}).

\begin{lem}\label{trans}
  Suppose that $\rho : \gp \to \GL_n(E)$ is generic ordinary.  Then the $\GL_n(\Qp)$-representation $\Pi(\rho^\vee)^{\ord}\otimes (\varepsilon^{n-1}\circ \det)$ is isomorphic to $\Pi(\rho)^{\ord}$ but where $\GL_n(\Qp)$
  acts via the inverse transpose on the latter.
\end{lem}
\begin{proof}
Let $w_0\in W$ be the longest element, $J\subseteq S$ a subset of pairwise orthogonal roots and $G_J\subseteq {\GL_n}_{/\Qp}$ as in \S\ref{prel}. The inverse transpose $g\mapsto \tau^{-1}(g)$ preserves $G_J$. If $\Pi_J$ is an admissible unitary continuous representation of $G_J(\Qp)$ over $E$, we let $\Pi_J^{\star}$ be the admissible unitary continuous representation of $G_{-w_0(J)}(\Qp)$ with the same underlying Banach space as $\Pi_J$ and where $g\in G_{-w_0(J)}(\Qp)$ acts by $\tau(\dot{w}_0g \dot{w}_0^{-1})^{-1}=\dot{w}_0\tau(g)^{-1} \dot{w}_0^{-1}\in G_J(\Qp)$ (note that $\dot{w}_0G_J\dot{w}_0^{-1}=G_{w_0(J)}=G_{-w_0(J)}$). If $\chi:T(\Qp)\rightarrow \oE^{\times}\subseteq E^{\times}$ is a unitary continuous character, one checks that
\begin{equation}\label{dualtrans}
\Big(\big(\Ind_{B^-(\Qp)\cap G_J(\Qp)}^{G_J(\Qp)}\chi\big)^{{\mathcal C}^0}\Big)^{\star}\cong\big(\Ind_{B^-(\Qp)\cap G_{-w_0(J)}(\Qp)}^{G_{-w_0(J)}(\Qp)}w_0(\chi^{-1})\big)^{{\mathcal C}^0}.
\end{equation}
Without loss of generality, $\rho$ is a good conjugate, and we let $\rho':=\dot w_0 \rho^\vee \dot w_0^{-1}$ (a good conjugate of $\rho^\vee$). It is easy to see that $C_{\rho'}=-w_0(C_{\rho})$, $W_{C_{\rho'}}=w_0W_{C_{\rho}}w_0^{-1}$ and $\chi_{\rho'}=w_0(\chi_{\rho}^{-1})$. Since $w_0((\varepsilon^{-1}\circ \theta)^{-1})=(\varepsilon^{-1}\circ \theta)\cdot (\varepsilon^{n-1}\circ\det)$, this implies by Proposition \ref{piI} and (\ref{dualtrans}) that
$$\widetilde\Pi(\rho)_I^{\star}\cong \widetilde\Pi(\rho')_{-w_0(I)}\otimes (\varepsilon^{n-1}\circ\det),$$
where $I$ is as in Proposition \ref{piI}. Finally, if $\Pi$ is an admissible unitary continuous representation of $\GL_n(\Qp)$ over $E$ and if we denote by $\Pi^{\star}$ the admissible unitary continuous representation of $\GL_n(\Qp)$ with the same underlying vector space as $\Pi$ but where $g\in \GL_n(\Qp)$ acts by $\tau(g)^{-1}$, one checks that
\begin{multline*}
\Big(\big(\Ind_{B^-(\Qp)G_J(\Qp)}^{\GL_n(\Qp)}\widetilde\Pi(\rho)_I\big)^{{\mathcal C}^0}\Big)^{\star}\cong \big(\Ind_{B^-(\Qp)G_{-w_0(J)}(\Qp)}^{\GL_n(\Qp)}\widetilde\Pi(\rho)_I^{\star}\big)^{{\mathcal C}^0}\cong\\
\big(\Ind_{B^-(\Qp)G_{-w_0(J)}(\Qp)}^{\GL_n(\Qp)}\widetilde\Pi(\rho')_{-w_0(I)}\big)^{{\mathcal C}^0}\otimes (\varepsilon^{n-1}\circ\det)
\end{multline*}
from which it follows that $\big(\Pi(\rho)^{\ord}\big)^{\star}\cong \Pi(\rho^\vee)^{\ord}\otimes (\varepsilon^{n-1}\circ \det)$ (see \S\ref{cons}).
\end{proof}

Since $r_{\tilde v^c}\cong r^\vee_{\tilde v}\otimes \varepsilon^{1-n}$, we have $\Pi(r_{\tilde v^c})^{\ord}\cong \Pi(r_{\tilde v}^\vee)^{\ord}\otimes (\varepsilon^{1-n}\circ \det)$ and Lemma \ref{trans} implies that the $\GL_n(\Qp)$-representation $\Pi(r_{\tilde v^c})^{\ord}\otimes (\varepsilon^{2(n-1)}\circ \det)$ is isomorphic to $\Pi(r_{\tilde v})^{\ord}$ but where $\GL_n(\Qp)$ acts via the inverse transpose on the latter. Equivalently, the $\GL_n(\Qp)$-representation $\Pi(r_{\tilde v^c})^{\ord}\otimes (\varepsilon^{n-1}\circ \det)$ is isomorphic to $\Pi(r_{\tilde v})^{\ord}\otimes (\varepsilon^{n-1}\circ \det)$ but where $\GL_n(\Qp)$ acts via the inverse transpose. As $c\circ \iota_{\tilde v^c}$ is conjugate to $\tau^{-1}\circ \iota_{\tilde v}$ (see \S\ref{set}), we see that the $G(F_v^+)$-representation $\Pi(r_{\tilde v})^{\ord}\otimes (\varepsilon^{n-1}\circ \det)$ ultimately only depends on $v\vert p$ and not on the choice of $\tilde v$ above $v$.

If $\Pi$ is an admissible unitary continuous representation over $E$ of $G(F^+\otimes_{\Q}\Qp)$ and if $\Pi',\Pi''\subseteq \Pi$ are two closed invariant subspaces such that all the irreducible subquotients of both $\Pi'$ and $\Pi''$ are isomorphic to irreducible subquotients of (unitary continuous) principal series as in (\ref{para}), then the same is true for the closed invariant subspace $\Pi'+\Pi''\subseteq \Pi$ (since any irreducible subquotient of $\Pi'+\Pi''$ appears either in $\Pi'$ or in $\Pi''$). Therefore one can define $\Pi^{\ord}\subseteq \Pi$ as the maximal closed $G(F^+\otimes_{\Q}\Qp)$-subrepresentation such that all its irreducible subquotients are isomorphic to irreducible subquotients of (unitary continuous) principal series of $G(F^+\otimes_{\Q}\Qp)$.

\begin{conj}\label{theconj}
Let $r:\gF\rightarrow {\GL}_n(E)$ be a continuous representation that satisfies conditions \(i\) to \(iv\) above and assume that there exist a compact open subgroup $U^p\subseteq \GAp$ and a finite set $\Sigma$ of finite places of $F^+$ as above such that $\widehat S(U^p,E)[\p^{\Sigma}]\ne 0$. Then there is an integer $d\in \Z_{>0}$ depending only on $U^p$ and $r$ such that we have an isomorphism of admissible unitary continuous representations of $G(F^+\otimes_{\Q}\Qp)=\prod_{v\vert p}G(F_v^+)$ over $E$,
\begin{equation}\label{beta0}
\bigg(\mathop{\widehat\bigotimes}\limits_{v\vert p}\Big(\Pi(r_{\tilde v})^{\ord}\otimes (\varepsilon^{n-1}\circ\det)\Big)\bigg)^{\oplus d}\congto \widehat S(U^p,E)[\p^{\Sigma}]^{\ord}.
\end{equation}
\end{conj}

\begin{rem}\label{comments}
Note that $\widehat S(U^p,E)[\p^{\Sigma}]\ne 0$ implies assumption (i) above on $r$ and that both sides of (\ref{beta0}) are admissible unitary continuous representations. It should also be true that $\widehat S(U^p,E)[\p^{\Sigma}]$, and thus $\widehat S(U^p,E)[\p^{\Sigma}]^{\ord}$, don't depend on $\Sigma$ as above (for the latter, this is implied by the conjecture).
\end{rem}

\begin{rem}\label{HT}
One can formulate a slight generalization of Conjecture \ref{theconj} where we assume in (iv) above that $r_{\tilde v}$ is generic ordinary for some {\it subset} $\Sigma_p$ of the places $v\vert p$ only and $r_{\tilde v}$ is potentially semi-stable with distinct Hodge--Tate weights at the other $v\vert p$. In (\ref{beta0}), one then has to replace the completed tensor product over all places $v$ of $F^+$ dividing $p$ on the left-hand side by the completed tensor product over the places of $\Sigma_p$ and $\widehat S(U^p,E)[\p^{\Sigma}]^{\ord}$ on the right-hand side by
$$\Hom_{V_p}\bigg(\bigotimes_{\substack{v\vert p\\ v\notin\Sigma_p}}L(\lambda_{\tilde v}),\widehat S(U^p,E)[\p^{\Sigma}]\bigg)^{\ord}.$$
Here $V_p$ is a suitable compact open subgroup of $\prod_{\substack{v\vert p\\ v\notin\Sigma_p}}G(\oFFv)$ and $L(\lambda_{\tilde v})$ is the irreducible algebraic representation of $G\times_{\oFF[1/N]}F^+_v\buildrel \iota\over\simeq {\GL_n}_{/F_{\tilde v}}$ over $E$ of highest weight $\lambda_{\tilde v}=(\lambda_{{\tilde v},1},\dots,\lambda_{{\tilde v},n})$, where $\lambda_{{\tilde v},1}> \lambda_{{\tilde v},2}-1> \cdots> \lambda_{{\tilde v},n-1}-(n-2)> \lambda_{{\tilde v},n}-(n-1)$ are the Hodge--Tate weights of $r_{\tilde v}$ (the Hodge--Tate weight of $\varepsilon$ being $1$ by definition).
\end{rem}

We now state the mod $p$ conjecture.

Let $\rbar:\gF\rightarrow {\GL}_n(k_E)$ be a continuous representation. We assume:
\begin{enumerate}
\item $\rbar^c\cong \rbar^\vee\otimes\omega^{1-n}$;
\item $\rbar$ is an absolutely irreducible representation of $\gF$.
\end{enumerate}
Let $U^p\subseteq \GAp$ be a compact open subgroup and $\Sigma$ a finite set of finite places of $F^+$ containing the set of places of $F^+$ that split in $F$ and divide $pN$, the set of places of $F^+$ that split in $F$ at which $U^p$ is not unramified and the set of places of $F^+$ that split in $F$ at which $\rbar$ is ramified. We associate to $\rbar$ and $\Sigma$ the maximal ideal $\m^\Sigma$ in $\TT^\Sigma$ with residue field $k_E$ generated by $\pE$ and all elements
$$\Big((-1)^j{\rm Norm}(w)^{j(j-1)/2}T_w^{(j)}-a^{(j)}_w\Big)_{j,w},$$
where $j\in \{1,\dots,n\}$, $w$ is a place of $F$ lying above a finite place of $F^+$ that splits in $F$ and doesn't belong to $\Sigma$, $X^n+\overline a_w^{(1)}X^{n-1}+\cdots + \overline a_w^{(n-1)}X+\overline a_w^{(n)}$ is the characteristic polynomial of $\rbar(\Frob_w)$ (an element of $k_E[X]$) and where $a^{(j)}_w$ is any element in $\oE$ lifting $\overline a^{(j)}_w$.

For a finite place $w$ of $F$, we denote by $\rbar_w$ be the restriction of $\rbar$ to a decomposition subgroup at $w$. We assume also:
\begin{enumerate}
 \setcounter{enumi}{2}
\item $\rbar_w$ is generic ordinary in the sense of Definition \ref{genbar} for all places $w$ of $F$ above $p$.
\end{enumerate}
Again, by the relation (i) above, (iii) is equivalent to $\rbar_{\tilde v}$ being generic ordinary for all places $\tilde v$ where $v\vert p$ in $F^+$. In particular we can define the $G(F_v^+)$-representation $\Pi(\rbar_{\tilde v})^{\ord}$ where $G(F_v^+)$ acts via $\iota_{\tilde v}:G(F_v^+)\buildrel\sim\over\rightarrow {\GL_n}(F_{\tilde v})$ (and where we choose $\theta:=(n-1,n-2,\dots,0)$ as twisting element). By the same proof as for Lemma \ref{trans}, the $G(F_v^+)$-representation $\Pi(\rbar_{\tilde v})^{\ord}\otimes (\omega^{n-1}\circ \det)$ only depends on $v\vert p$ and not on the choice of $\tilde v$ above $v$.

If $\Pi$ is an admissible smooth representation of $G(F^+\otimes_{\Q}\Qp)$ over $k_E$, we define $\Pi^{\ord}\subseteq \Pi$ as the maximal $G(F^+\otimes_{\Q}\Qp)$-subrepresentation such that all its irreducible subquotients are isomorphic to irreducible subquotients of principal series of $G(F^+\otimes_{\Q}\Qp)$ over $k_E$.

\begin{conj}\label{theconjbar}
Let $\rbar:\gF\rightarrow {\GL}_n(k_E)$ be a continuous representation that satisfies conditions \(i\) to \(iii\) above and assume that there exist a compact open subgroup $U^p\subseteq \GAp$ and a finite set $\Sigma$ of finite places of $F^+$ as above such that $S(U^p,k_E)[\m^{\Sigma}]\ne 0$. Then there is an integer $d\in \Z_{>0}$ depending only on $U^p$ and $\rbar$ such that we have an isomorphism of admissible smooth representations of $G(F^+\otimes_{\Q}\Qp)=\prod_{v\vert p}G(F_v^+)$ over $k_E$,
\begin{equation}\label{beta1}
\bigg({\bigotimes_{v\vert p}}\Big(\Pi(\rbar_{\tilde v})^{\ord}\otimes (\omega^{n-1}\circ\det)\Big)\bigg)^{\oplus d}\congto  S(U^p,k_E)[\m^{\Sigma}]^{\ord}.
\end{equation}
\end{conj}

\begin{rem}\label{caution}
As in Remark \ref{comments}, note that $S(U^p,k_E)[\m^{\Sigma}]\ne 0$ implies assumption (i) above on $\rbar$ and that both sides of (\ref{beta1}) are admissible representations. Also, Conjecture \ref{theconjbar} in particular implies that $S(U^p,k_E)[\m^{\Sigma}]^{\ord}$ doesn't depend on $\Sigma$, but in fact it shouldn't be too hard to prove that $S(U^p,k_E)[\m^{\Sigma}]$ (and thus $S(U^p,k_E)[\m^{\Sigma}]^{\ord}$) is independent of $\Sigma$ for any $p$, see e.g.\ \cite[Lem.\ 4.6]{BDJ} for an analogous result.
\end{rem}

\subsection{A local result}\label{locres}

In this section, we prove the main local theorems which our local-global compatibility results (in \S\ref{theresult}) will rely on. We keep the same notation and assumptions as in \S\ref{variant2}. 

We start with some preliminaries. If $\sigma$ is a Serre weight for $G(\Fp)$ (seen as a representation of $G(\Zp)$), we denote by $\ind_{G(\Zp)}^{G(\Qp)}\sigma$ the usual smooth representation that is compactly induced from $\sigma$ and by ${\mathcal H}_{G}(\sigma):=\End_{G(\Qp)}\big(\ind_{G(\Zp)}^{G(\Qp)}\sigma\big)$ the (Hecke) $k_E$-algebra of its endomorphisms (see e.g.\ \cite[\S2.1]{He2}). The $k_E$-algebra ${\mathcal H}_{G}(\sigma)$ is commutative and of finite type, and for any smooth representation $\pi$ of $G(\Qp)$ over $k_E$, the vector space $\Hom_{G(\Zp)}(\sigma,\pi\vert_{G(\Zp)})$ is naturally a (right) ${\mathcal H}_{G}(\sigma)$-module via Frobenius reciprocity $\Hom_{G(\Zp)}(\sigma,\pi\vert_{G(\Zp)})\cong \Hom_{G(\Qp)}\big(\ind_{G(\Zp)}^{G(\Qp)}\sigma,\pi\big)$. Moreover, ${\mathcal H}_{G}(\sigma)$ can be identified with the $k_E$-vector space of compactly supported functions $\varphi:G(\Qp)\rightarrow \End_{k_E}(\sigma)$ such that $\varphi(h_1gh_2)=h_1\circ\varphi(g)\circ h_2$ for $h_1,h_2\in G(\Zp)$ and $g\in G(\Qp)$ endowed with the multiplication given by convolution.

Let $P=MU_M$ be a standard parabolic subgroup of $G$ over $\Zp$, where $U_M$ (resp.\ $M$) is the unipotent radical (resp.\ the Levi subgroup), $P^-$ the opposite parabolic subgroup of $P$ and $U^-_{M}$ the unipotent radical of $P^-$. Replacing $G$ by $M$ and $\sigma$ by the subspace of coinvariants $\sigma_{U_M^-(\Zp)}$ (which is a Serre weight for $M(\Fp)$, see e.g.\ \cite[Lem.\ 2.3]{He2}), one can define in an analogous way the (commutative noetherian) algebra ${\mathcal H}_{M}(\sigma_{U_M^-(\Zp)})$. When $M=T$, we write $U^-$ instead of $U_T^-$.

\begin{ex}\label{plustard}
Let $G={\GL}_n$, $T$ the torus of diagonal matrices and $B$ the upper triangular matrices. For $1\leq j\leq n$ let $M_j:={\GL}_{n-j}\times {\GL}_{j}$ and $P_j:=M_jU_{M_j}$. One has ${\mathcal H}_{G}(\sigma)=k_E[T_{\sigma,1},\dots,T_{\sigma,n-1},T_{\sigma,n},T_{\sigma,n}^{-1}]$, where $T_{\sigma,j}$, $1\leq j\leq n$, corresponds to $\varphi_j:G(\Qp)\rightarrow \End_{k_E}(\sigma)$ with support on $G(\Zp)\smat{{\bf 1}_{n-j}&\\& p{\bf 1}_{j}}G(\Zp)$ sending $\smat{{\bf 1}_{n-j}&\\& p{\bf 1}_{j}}$ to the endomorphism $\sigma\twoheadrightarrow \sigma_{U_{M_j}^-(\Zp)}\buildrel\sim\over\leftarrow \sigma^{U_{M_j}(\Zp)}\hookrightarrow \sigma$.
\end{ex}

By \cite[\S2]{He2} the map $\varphi\mapsto \big(m\mapsto p_{U_M^-}\circ \big(\sum_{u\in U_M^-(\Zp)\backslash U_M^-(\Qp)}\varphi(um)\big)\big)$ (where $m\in M(\Qp)$ and $p_{U_M^-}$ is the projection $\sigma\twoheadrightarrow \sigma_{U_{M}^-(\Zp)}$) induces an injective homomorphism of $k_E$-algebras ${\mathcal S}_M:{\mathcal H}_{G}(\sigma)\hookrightarrow {\mathcal H}_{M}(\sigma_{U_M^-(\Zp)})$ which is a localization map (\cite[Prop.\ 2.12]{He2}). It follows that if $\eta:{\mathcal H}_{G}(\sigma)\rightarrow k_E$ is a morphism of $k_E$-algebras that factors as a morphism of $k_E$-algebras ${\mathcal H}_{G}(\sigma)\buildrel{{\mathcal S}_M}\over\hookrightarrow {\mathcal H}_{M}(\sigma_{U_M^-(\Zp)})\rightarrow k_E$, then the morphism ${\mathcal H}_{M}(\sigma_{U_M^-(\Zp)})\rightarrow k_E$ is unique and we denote it by ${\mathcal S}_M^{-1}(\eta)$.

\begin{definit}\label{deford}
We say that a morphism of $k_E$-algebras ${\mathcal H}_{G}(\sigma)\rightarrow k_E$ is an \emph{ordinary character} (of ${\mathcal H}_{G}(\sigma)$) if it factors as a morphism of $k_E$-algebras ${\mathcal H}_{G}(\sigma)\buildrel{\mathcal S}_T\over\hookrightarrow {\mathcal H}_{T}(\sigma_{U^-(\Zp)})\rightarrow k_E$.
\end{definit}

If $\eta:{\mathcal H}_{G}(\sigma)\rightarrow k_E$ is any morphism of $k_E$-algebras, $\Hom_{G(\Zp)}(\sigma,\pi\vert_{G(\Zp)})[\eta]$ will denote the maximal $k_E$-vector subspace of $\Hom_{G(\Zp)}(\sigma,\pi\vert_{G(\Zp)})$ on which ${\mathcal H}_{G}(\sigma)$ acts by the character $\eta$.

The following theorem (a special case of \cite[Thm.\ 3.1]{He2}) will be crucial.

\begin{thm}[\cite{He2}]\label{he}
Let $\lambda\in X_1(T)$ such that $1\leq \langle \lambda,\alpha^\vee\rangle \leq p-1$ for all $\alpha\in S$ and let $\eta:{\mathcal H}_{G}(F(\lambda))\rightarrow k_E$ be an ordinary character. Then we have a canonical isomorphism of smooth $G(\Qp)$-representations
\begin{equation}\label{isocomp}
\big(\ind_{G(\Zp)}^{G(\Qp)}F(\lambda)\big)\otimes_{{\mathcal H}_{G}(F(\lambda)),\eta}k_E\buildrel\sim\over\longrightarrow  \Ind_{B^-(\Qp)}^{G(\Qp)}\chi,
\end{equation}
where $\chi:B^-(\Qp)\twoheadrightarrow T(\Qp)\rightarrow k_E^{\times}$ is the smooth character giving the action of $T(\Qp)$ on the $1$-dimensional vector space
$$\big(\ind_{T(\Zp)}^{T(\Qp)}F(\lambda)_{U^-(\Zp)}\big)\otimes_{{\mathcal H}_{T}(F(\lambda)_{U^-(\Zp)}),{\mathcal S}_T^{-1}(\eta)}k_E.$$
\end{thm}

Note that we necessarily have $F(\lambda)\hookrightarrow \soc_{G(\Zp)}\big(\Ind_{B^-(\Qp)}^{G(\Qp)}\chi\big)$, in particular $\chi\vert_{T(\Zp)}$ gives the action of $T(\Zp)$ on $F(\lambda)_{U^-(\Zp)}$. Conversely, for any character $\chi:T(\Qp)\rightarrow k_E^{\times}$ such that $\chi\vert_{T(\Zp)}\cong F(\lambda)_{U^-(\Zp)}$, it is easy to see that there is a unique morphism of $k_E$-algebras $\eta':{\mathcal H}_{T}(F(\lambda)_{U^-(\Zp)})\rightarrow k_E$ such that (\ref{isocomp}) holds with $\eta:=\eta'\circ{\mathcal S}_T$.

\begin{ex}\label{plustard2}
Keep the setting and notation of Example \ref{plustard} and let $\sigma:=F(\lambda)$ with $\lambda:=(\lambda_1,\dots,\lambda_n)\in \Z^n$ such that $1\leq \lambda_i-\lambda_{i+1}\leq p-1$. Then
$$\eta:{\mathcal H}_{\GL_n}(F(\lambda))=k_E[T_{F(\lambda),1},\dots,T_{F(\lambda),n-1},T_{F(\lambda),n},T_{F(\lambda),n}^{-1}]\longrightarrow k_E$$
is ordinary if and only if $\eta_{\lambda,j}:=\eta(T_{F(\lambda),j})\in k_E^{\times}$ for all $j$ and we then have
\begin{multline*}
\big(\ind_{\GL_n(\Zp)}^{\GL_n(\Qp)}F(\lambda)\big)\otimes_{{\mathcal H}_{\GL_n}(F(\lambda)),\eta}k_E\cong \\ \Ind_{B^-(\Qp)}^{\GL_n(\Qp)}\big(\omega^{\lambda_1}\nr\big({\scriptstyle\frac{\eta_{\lambda,n-1}}{\eta_{\lambda,n}}}\big)\otimes \cdots \otimes \omega^{\lambda_{n-1}}\nr\big({\scriptstyle\frac{\eta_{\lambda,1}}{\eta_{\lambda,2}}}\big)\otimes \omega^{\lambda_n}\nr\big({\scriptstyle\frac{1}{\eta_{\lambda,1}}}\big)\big).
\end{multline*}
\end{ex}

\begin{cor}\label{prat}
Keep the same assumptions as in Theorem \ref{he} and let $\pi$ be a smooth representation of $G(\Qp)$ over $k_E$. Then restriction to $F(\lambda)$ induces an isomorphism
$$\Hom_{G(\Qp)}\big(\Ind_{B^-(\Qp)}^{G(\Qp)}\chi,\pi\big)\buildrel\sim\over\longrightarrow \Hom_{G(\Zp)}\big(F(\lambda),\pi\vert_{G(\Zp)}\big)[\eta].$$
\end{cor}
\begin{proof}
By Frobenius reciprocity we have
\begin{eqnarray*}
\Hom_{G(\Zp)}\big(F(\lambda),\pi\vert_{G(\Zp)}\big)[\eta]&=&\Hom_{G(\Qp)}\!\big(\!\ind_{G(\Zp)}^{G(\Qp)}F(\lambda),\pi\big)[\eta]\\
&=&\Hom_{G(\Qp)}\!\big(\big(\!\ind_{G(\Zp)}^{G(\Qp)}F(\lambda)\big)\otimes_{{\mathcal H}_{G}(F(\lambda)),\eta}k_E,\pi\big)
\end{eqnarray*}
and the statement follows from Theorem \ref{he}.
\end{proof}

The following proposition and its two corollaries are entirely due to Pa{\v{s}}k{\=u}nas.

\begin{prop}\label{pask1}
Let $\sigma$ be a Serre weight for $G(\Fp)$, $\eta:{\mathcal H}_{G}(\sigma)\rightarrow k_E$ a morphism of $k_E$-algebras and $\pi(\sigma,\eta):=\big(\ind_{G(\Zp)}^{G(\Qp)}\sigma\big)\otimes_{{\mathcal H}_{G}(\sigma),\eta}k_E$. Let $\Pi$ be an admissible smooth representation of $G(\Qp)$ over $k_E$ such that $\Pi\vert_{G(\Zp)}$ is an injective object in the category of smooth representations of $G(\Zp)$ over $k_E$. If $\Hom_{G(\Qp)}\big(\pi(\sigma,\eta),\Pi\big)=0$ then ${\rm Ext}^1_{G(\Qp)}\big(\pi(\sigma,\eta),\Pi\big)=0$ \(in the category of smooth representations of $G(\Qp)$ over $k_E$\).
\end{prop}
\begin{proof}
Let $\mathcal E$ be an extension $0\rightarrow \Pi\rightarrow {\mathcal E}\rightarrow \pi(\sigma,\eta)\rightarrow 0$ in the category of smooth representations of $G(\Qp)$ over $k_E$. It is enough to prove that the functor $\Hom_{G(\Qp)}(\pi(\sigma,\eta),\cdot)$ is exact on that sequence. By Frobenius reciprocity as in the proof of Corollary \ref{prat}, this functor is $\Hom_{G(\Zp)}(\sigma,\cdot)[\eta]$. Since $\Pi\vert_{G(\Zp)}$ is injective, the extension splits when restricted to $G(\Zp)$, hence the sequence remains exact after applying $\Hom_{G(\Zp)}(\sigma,\cdot)$. If we denote with the subscript $\eta$ the generalized eigenspace for the action of the commutative algebra ${\mathcal H}_{G}(\sigma)$ corresponding to the eigencharacter $\eta$, the sequence is thus still exact after applying $\Hom_{G(\Zp)}(\sigma,\cdot)_{\eta}$. But the assumption implies $\Hom_{G(\Zp)}(\sigma,\Pi)_{\eta}=0$, hence $\Hom_{G(\Zp)}(\sigma,{\mathcal E})_{\eta}\buildrel\sim\over\rightarrow \Hom_{G(\Zp)}(\sigma,\pi(\sigma,\eta))_{\eta}$ and thus $\Hom_{G(\Zp)}(\sigma,{\mathcal E})[\eta]\buildrel\sim\over\rightarrow \Hom_{G(\Zp)}(\sigma,\pi(\sigma,\eta))[\eta]$. This finishes the proof.
\end{proof}

\begin{cor}\label{pask2}
Let $\Pi$ be as in Proposition \ref{pask1} and $\pi$ be a smooth representation of $G(\Qp)$ over $k_E$ which is of finite length such that its irreducible constituents are all principal series. Let $\pi_1\subseteq \pi$ be a subrepresentation such that $\Hom_{G(\Qp)}(C,\Pi)=0$ if $C$ is an irreducible constituent of $\pi/\pi_1$. Then restriction to $\pi_1$ induces an isomorphism 
$$\Hom_{G(\Qp)}(\pi,\Pi)\buildrel\sim\over\longrightarrow \Hom_{G(\Qp)}(\pi_1,\Pi).$$
\end{cor}
\begin{proof}
Note that we have $\Hom_{G(\Qp)}\big(\pi/\pi_1,\Pi\big)=0$ by d\'evissage. It follows from \cite[Lem.\ 2.5]{He2}, \cite[Lem.\ 2.14]{He2} and \cite[Thm.\ 3.1]{He2} that, if $C$ is an (irreducible) principal series, there exist $(\sigma,\eta)$ such that $C\cong \pi(\sigma,\eta)$ with $\pi(\sigma,\eta)$ as in the statement of Proposition \ref{pask1}. By Proposition \ref{pask1} we thus have ${\rm Ext}^1_{G(\Qp)}\big(C,\Pi\big)=0$ for all irreducible constituents of $\pi/\pi_1$, which implies ${\rm Ext}^1_{G(\Qp)}\big(\pi/\pi_1,\Pi\big)=0$ by d\'evissage. Applying $\Hom_{G(\Qp)}(\cdot,\Pi)$ to the exact sequence $0\rightarrow \pi_1\rightarrow \pi\rightarrow \pi/\pi_1\rightarrow 0$ then gives the result. 
\end{proof}

We now give a $p$-adic version of Corollary \ref{pask2}. We refer to Appendix A for the definition of an admissible unitary continuous representation $\pi$ of $G(\Qp)$ over $E$ which is residually of finite length, and denote by $\overline\pi^{\sss}$ the semi-simplification of its mod $\pE$ reduction.

\begin{cor}\label{pask3}
Let $\Pi$ be an admissible unitary continuous representation of $G(\Qp)$ over $E$ which has a unit ball $\Pi^0$ such that $(\Pi^0\otimes_{\oE}k_E)\vert_{G(\Zp)}$ is an injective object in the category of smooth representations of $G(\Zp)$ over $k_E$. Let $\pi$ be an admissible unitary continuous representation of $G(\Qp)$ over $E$ which is residually of finite length such that the irreducible constituents of $\overline\pi^{\sss}$ are principal series. Let $\pi_1\subseteq \pi$ be a closed subrepresentation such that $\Hom_{G(\Qp)}(C,\Pi^0\otimes_{\oE}k_E)=0$ if $C$ is an irreducible constituent of $\overline{\pi/\pi_1}^{\sss}$. Then restriction to $\pi_1$ induces an isomorphism 
$$\Hom_{G(\Qp)}(\pi,\Pi)\buildrel\sim\over\longrightarrow \Hom_{G(\Qp)}(\pi_1,\Pi).$$
\end{cor}
\begin{proof}
Let $\pi^0$ be a unit ball in $\pi$ and set $\pi_1^0:=\pi^0\cap \pi_1$. Note that $\pi^0/\pi_1^0$ is a unit ball in $\pi/\pi_1$. By the proof of Corollary \ref{pask2} applied to $\Pi^0\otimes_{\oE}k_E$, $\pi^0\otimes_{\oE}k_E$ and $\pi_1^0\otimes_{\oE}k_E$, we get
\begin{eqnarray*}
\Hom_{G(\Qp)}\big((\pi^0\otimes_{\oE}k_E)/(\pi^0_1\otimes_{\oE}k_E),\Pi^0\otimes_{\oE}k_E\big)&=&0\\
{\rm Ext}^1_{G(\Qp)}\big((\pi^0\otimes_{\oE}k_E)/(\pi^0_1\otimes_{\oE}k_E),\Pi^0\otimes_{\oE}k_E\big)&=&0.
\end{eqnarray*}
By \cite[Prop.\ B1]{Ha1} (and the lines before that proposition) we have 
$$\dim_E{\rm Ext}^1_{G(\Qp)}\big(\pi/\pi_1,\Pi\big)\leq \dim_{k_E}{\rm Ext}^1_{G(\Qp)}\big((\pi^0/\pi_1^0)\otimes_{\oE}k_E,\Pi^0\otimes_{\oE}k_E\big)=0$$
and likewise with Hom,  where the first ${\rm Ext}^1$ is in the category of admissible unitary continuous representation of $G(\Qp)$ over $E$. Thus $\Hom_{G(\Qp)}\big(\pi/\pi_1,\Pi\big)=0$ and ${\rm Ext}^1_{G(\Qp)}\big(\pi/\pi_1,\Pi\big)=0$. We conclude as for Corollary \ref{pask2}.
\end{proof}

We now fix a continuous homomorphism
$$\rhobar:\gp\longrightarrow \widehat{B}_{C_{\rhobar}}(k_E)\subseteq \widehat{B}(k_E)\subseteq \widehat{G}(k_E),$$
such that the closed subset $C_{\rhobar}\subseteq R^{+\vee}$ is minimal under conjugation by $\widehat{B}(k_E)$ and we assume that $\rhobar$ is inertially generic (Definition \ref{genbarfor}). We recall that this implies that $p$ is large enough so that Lemma~\ref{borelp} holds (see \S\ref{variant2}).

If $\sigma$ is a Serre weight for $G(\Fp)$ and $\pi$ an admissible smooth representation of $G(\Qp)$ over $k_E$, we write $\Hom_{G(\Zp)}(\sigma,\pi\vert_{G(\Zp)})^{\ord}$ for the maximal vector subspace of $\Hom_{G(\Zp)}(\sigma,\pi\vert_{G(\Zp)})$ on which the action of ${\mathcal H}_{G}(\sigma)$ extends to ${\mathcal H}_{T}(\sigma_{U^-(\Zp)})$.

\begin{thm}\label{main}
Let $\Pi$ be an admissible smooth representation of $G(\Qp)$ over $k_E$ such that $\Pi\vert_{G(\Zp)}$ is an injective object in the category of smooth representations of $G(\Zp)$ over $k_E$. Let $w_{C_{\rhobar}}\in W_{C_{\rhobar}}$ and assume
\begin{equation}\label{nul}
\Hom_{G(\Zp)}(F(\lambda_{w^{-1}(\chi_{\rhobar})}-\theta),\Pi\vert_{G(\Zp)})^{\ord}=0
\end{equation}
for all $w=\big(\prod_{\alpha\in I^\vee}s_{\alpha}\big)w_{C_{\rhobar}}$, where $I\subseteq w_{C_{\rhobar}}(S^\vee)\cap C_{\rhobar}$ runs among the \emph{non-empty} subsets of pairwise orthogonal roots and where $\lambda_{w^{-1}(\chi_{\rhobar})}$ is as in the proof of Proposition \ref{ordserex}. Then restriction to the $G(\Zp)$-socle induces an isomorphism
$$\Hom_{G(\Qp)}\big(\Pi(\rhobar)_{C_{\rhobar},w_{C_{\rhobar}}},\Pi\big)\buildrel\sim\over\longrightarrow \Hom_{G(\Zp)}\big(F(\lambda_{w_{C_{\rhobar}}^{-1}(\chi_{\rhobar})}-\theta),\Pi\vert_{G(\Zp)}\big)[\eta_{\rhobar,w_{C_{\rhobar}}}],$$
where $\eta_{\rhobar,w_{C_{\rhobar}}}$ is the ordinary character of ${\mathcal H}_{G}(F(\lambda_{w_{C_{\rhobar}}^{-1}(\chi_{\rhobar})}-\theta))$ associated to $\chi=w_{C_{\rhobar}}^{-1}(\chi_{\rhobar})\cdot\omega^{-1}\circ\theta$ in Theorem \ref{he}.
\end{thm}
\begin{proof}
First, the fact that $F(\lambda_{w_{C_{\rhobar}}^{-1}(\chi_{\rhobar})}-\theta)$ is the $G(\Zp)$-socle of $\Pi(\rhobar)_{C_{\rhobar},w_{C_{\rhobar}}}$ follows from the proof of Proposition \ref{ordserex} (and the inertial genericity of $\rhobar$). By Corollary \ref{prat} applied to $\lambda=\lambda_{w_{C_{\rhobar}}^{-1}(\chi_{\rhobar})}-\theta$ and $\eta=\eta_{\rhobar,w_{C_{\rhobar}}}$, it suffices to prove that the restriction to $\Pi(\rhobar)_{\varnothing} = \Ind_{B^-(\Qp)}^{G(\Qp)}w_{C_{\rhobar}}^{-1}(\chi_{\rhobar})\cdot(\omega^{-1}\circ\theta)$ induces an isomorphism
\begin{equation}
\Hom_{G(\Qp)}\big(\Pi(\rhobar)_{C_{\rhobar},w_{C_{\rhobar}}},\Pi\big)\buildrel\sim\over\longrightarrow \Hom_{G(\Qp)}\big(\Pi(\rhobar)_{\varnothing},\Pi\big).\label{eq:6}
\end{equation}
Recall from \S\ref{variant2} that the irreducible constituents of $\Pi(\rhobar)_{C_{\rhobar},w_{C_{\rhobar}}}/\Pi(\rhobar)_{\varnothing}$ are the principal series $\Ind_{B^-(\Qp)}^{G(\Qp)}w^{-1}(\chi_{\rhobar})\cdot(\omega^{-1}\circ\theta)$ where $w=\big(\prod_{\alpha\in I^\vee}s_{\alpha}\big)w_{C_{\rhobar}}$ with $I\subseteq w_{C_{\rhobar}}(S^\vee)\cap C_{\rhobar}$ running among the \emph{non-empty} subsets of pairwise orthogonal roots. The result follows from Corollary \ref{pask2} (with (\ref{nul}) and Corollary \ref{prat}).
\end{proof}

Taking the direct sum over $w_{C_{\rhobar}}\in W_{C_{\rhobar}}$ in Theorem \ref{main} yields the following corollary.

\begin{cor}\label{pascite}
We keep the notation of Theorem \ref{main} and assume
\begin{equation*}
\Hom_{G(\Zp)}(F(\lambda_{w^{-1}(\chi_{\rhobar})}-\theta),\Pi\vert_{G(\Zp)})^{\ord}=0
\end{equation*}
for all $w_{C_{\rhobar}}\in W_{C_{\rhobar}}$ and all $w=\big(\prod_{\alpha\in I^\vee}s_{\alpha}\big)w_{C_{\rhobar}}$, where $I\subseteq w_{C_{\rhobar}}(S^\vee)\cap C_{\rhobar}$ runs among the non-empty subsets of pairwise orthogonal roots. Then restriction to the $G(\Zp)$-socle induces an isomorphism
$$\Hom_{G(\Qp)}\big(\Pi(\rhobar)^{\ord},\Pi\big)\buildrel\sim\over\longrightarrow \bigoplus_{w_{C_{\rhobar}}\in W_{C_{\rhobar}}}\Hom_{G(\Zp)}\big(F(\lambda_{w_{C_{\rhobar}}^{-1}(\chi_{\rhobar})}-\theta),\Pi\vert_{G(\Zp)}\big)[\eta_{\rhobar,w_{C_{\rhobar}}}].$$
\end{cor}

If $\rho:\gp\longrightarrow \widehat{B}_{C_{\rho}}(E)\subseteq \widehat{B}(E)\subseteq \widehat{G}(E)$ is a continuous homomorphism such that the closed subset $C_{\rho}\subseteq R^{+\vee}$ is minimal under conjugation by $\widehat{B}(E)$, we let $\overline{\chi_\rho}:T(\Qp)\rightarrow k_E^\times$ be the character $\chi_\rho:T(\Qp)\rightarrow \oE^\times\subset E^\times$ of \S\ref{cons} composed with $\oE^\times\twoheadrightarrow k_E^\times$. Using Corollary \ref{pask3} instead of Corollary \ref{pask2}, we also have $p$-adic versions of Corollary \ref{pascite} (and Theorem \ref{main}) that we state without proof. 

\begin{cor}\label{mainpadic}
Let $\rho:\gp\longrightarrow \widehat{B}_{C_{\rho}}(E)\subseteq \widehat{B}(E)$ be as above such that $(\overline{\chi_{\rho}}\circ \alpha^{\vee})\vert_{\Zp^\times}\notin \{1,\omega,\omega^{-1}\}$ for all $\alpha\in R^{+}$. Let $\Pi$ be as in Corollary \ref{pask3} and assume
\begin{equation*}
\Hom_{G(\Zp)}(F(\lambda_{w^{-1}(\overline{\chi_{\rho}})}-\theta),(\Pi^0\otimes k_E)\vert_{G(\Zp)})^{\ord}=0
\end{equation*}
for all $w_{C_{\rho}}\in W_{C_{\rho}}$ and all $w=\big(\prod_{\alpha\in I^\vee}s_{\alpha}\big)w_{C_{\rho}}$, where $I\subseteq w_{C_{\rho}}(S^\vee)\cap C_{\rho}$ runs among the non-empty subsets of pairwise orthogonal roots. Then restriction to the $G(\Qp)$-socle induces an isomorphism
$$\Hom_{G(\Qp)}\big(\Pi(\rho)^{\ord},\Pi\big)\buildrel\sim\over\longrightarrow \bigoplus_{w_{C_{\rho}}\in W_{C_{\rho}}}\!\!\!\!\Hom_{G(\Qp)}\Big(\big(\Ind_{B^-(\Qp)}^{G(\Qp)}w_{C_{\rho}}^{-1}(\chi_{\rho})\cdot(\varepsilon^{-1}\circ\theta)\big)^{{\mathcal C}^0},\Pi\Big).$$
\end{cor}

\subsection{A global result}\label{theresult}

We apply the main results of \S\ref{locres} to prove a weak version of Conjecture \ref{theconjbar} and a special case of Conjecture \ref{theconj}. We keep the notation of \S\ref{set} and \S\ref{globconj}. We suppose from now on that $F/F^+$ is unramified at all finite places.

We say that a compact open subgroup $U\subseteq \GAp\times G(\oFFp)$ is sufficiently small if there exists a finite place $v$ of $F^+$ such that the projection of $U$ to $G(F_v^+)$ contains no element of finite order. We say a compact open subgroup $U^p\subseteq \GAp$ is sufficiently small if $U^pG(\oFFp)$ is sufficiently small. The following lemma is well-known (see e.g.\ \cite[\S7.1.2]{EGH}).

\begin{lem}\label{small}
Let $M$ be an $\oE$-module endowed with an $\oE$-linear action of $G(\oFFp)$ and let $U\subseteq \GAp\times G(\oFFp)$ be a sufficiently small compact open subgroup. Then one has a natural isomorphism $S(U,M)\buildrel\sim\over\rightarrow M^{\oplus r}$ for some integer $r\geq 0$ that only depends on $U$.
\end{lem}

\begin{lem}\label{tensor}
Let $M$ be a finite-dimensional smooth representation of $G(\oFFp)$ over $k_E$ and $U^p\subseteq \GAp$ a compact open subgroup. Then one has
$$S\big(U^pG(\oFFp),M\big)\buildrel\sim\over\longrightarrow \Hom_{G(\oFFp)}\big(M^\vee,S(U^p,k_E)\big).$$
\end{lem}
\begin{proof}
The proof is essentially that of \cite[Lem.\ 7.4.3]{EGH}, so we just define the map. We send $f:G(F^+)\backslash \GA\rightarrow M$ to
$$\ell\in  M^\vee \mapsto \big(g\mapsto \ell(f(g))\big).$$
Then $u_p\cdot \ell$ is sent to $\big(g\mapsto \ell(u_p^{-1}(f(g)))=\ell(f(gu_p))\big)=u_p\cdot \big(g\mapsto \ell(f(g))\big)$, where $u_p\in G(\oFFp)$, and the map is thus $G(\oFFp)$-equivariant. Since the action of $G(\oFFp)$ on $M^\vee$ is smooth, its image lies in $S(U^p,k_E)$.
\end{proof}

If $U^p\subseteq \GAp$ is a compact open subgroup, we denote by $\Sigma$ a finite set of finite places of $F^+$ containing the set of places of $F^+$ that split in $F$ and divide $pN$ and the set of places of $F^+$ that split in $F$ at which $U^p$ is not unramified.

\begin{prop}\label{injsigma}
Let $U^p\subseteq \GAp$ be a sufficiently small compact open subgroup, $\Sigma$ a set of places of $F^+$ as above and $\m^{\Sigma}$ a maximal ideal of $\TT^{\Sigma}$ \(see \S\ref{set}\) with residue field $k_E$. Then the admissible smooth $G(F^+\otimes_{\Q}\Qp)$-representations $S(U^p,k_E)$ and $S(U^p,k_E)_{\m^{\Sigma}}$ are injective objects in the category of smooth representations of $G(\oFFp)$ over $k_E$.
\end{prop}
\begin{proof}
The second representation being a direct summand of the first (see \S\ref{set}), it is enough to prove the statement for the first. By \cite[Prop.\ 2.1.9]{Em3}, it is enough to prove that $S(U^p,k_E)$ is injective in the category of {\it admissible} smooth representations of $G(\oFFp)$ over $k_E$. Let $j:\pi\hookrightarrow \pi'$ be an injection of admissible smooth $G(\oFFp)$-representations over $k_E$, we have to prove that
$$\Hom_{G(\oFFp)}\big(\pi',S(U^p,k_E)\big)\buildrel \cdot\circ j\over\longrightarrow \Hom_{G(\oFFp)}\big(\pi,S(U^p,k_E)\big)$$
is surjective. Write $\pi'=\cup_m \pi'_m$, where $(\pi'_m)_{m\in \Z_{>0}}$ is an increasing sequence of finite-dimensional vector subspaces preserved by $G(\oFFp)$ (recall $\pi'$ is admissible). It is enough to prove that all maps
\begin{equation}\label{surjm}
\Hom_{G(\oFFp)}\big(\pi'_m,S(U^p,k_E)\big)\buildrel \cdot\circ j\vert_{\pi\cap\pi'_m}\over\longrightarrow \Hom_{G(\oFFp)}\big(\pi\cap \pi'_m,S(U^p,k_E)\big)
\end{equation}
are surjective. Indeed, since $\Hom_{G(\oFFp)}\big(\pi'_m/(\pi\cap\pi'_m),S(U^p,k_E)\big)$ is finite-dimen\-sional ($S(U^p,k_E)$ being admissible), the Mittag-Leffler conditions are satisfied on the projective system $\big(\Hom_{G(\oFFp)}\big(\pi'_m/(\pi\cap\pi'_m),S(U^p,k_E)\big)\big)_{m\in \Z_{>0}}$ and the surjection survives the projective limit. The surjection in (\ref{surjm}) then follows from Lemma \ref{tensor} applied to $M=(\pi'_m)^{\vee}$ and $M=(\pi\cap\pi'_m)^\vee$ and Lemma \ref{small}.
\end{proof}

Let $\rbar:\gF\rightarrow {\GL}_n(k_E)$ be a continuous representation such that $\rbar^c\cong \rbar^\vee\otimes\omega^{1-n}$ and $\rbar$ is absolutely irreducible. If $U^p\subseteq \GAp$ is a compact open subgroup we denote by $\Sigma$ a finite set of finite places of $F^+$ containing the set of places of $F^+$ that split in $F$ and divide $pN$, the set of places of $F^+$ that split in $F$ at which $U^p$ is not unramified and the set of places of $F^+$ that split in $F$ at which $\rbar$ is ramified. Recall that we have associated to $\rbar$ and $\Sigma$ a maximal ideal $\m^{\Sigma}$ of $\TT^{\Sigma}$ with residue field $k_E$ in \S\ref{globconj}.

We denote by $G_p$ the restriction of scalars from $\oFFp$ to $\Zp$ of the algebraic group $G\times_{\oFF[1/N]}\oFFp$, so that $G_p(\Zp)=G(\oFFp)=\prod_{v\vert p}G(\oFFv)$ (the algebraic group $G_p$ is isomorphic to $\prod_{\widetilde v}\!{\GL_n}_{/\Zp}$). If $\sigma$ is any Serre weight for $G_p(\Fp)$ and if $U^p$ and $\Sigma$ are as above, then the commutative $k_E$-algebra
$$\TT^{\Sigma}\otimes_{\oE}{\mathcal H}_{G_p}(\sigma)=(\TT^{\Sigma}\otimes_{\oE}k_E)\otimes_{k_E}{\mathcal H}_{G_p}(\sigma)$$
(see \S\ref{locres} for ${\mathcal H}_{G_p}(\sigma)$) acts on $\Hom_{G(\oFFp)}(\sigma,S(U^p\!,k_E))=\Hom_{G_p(\Zp)}(\sigma,S(U^p\!,k_E))$ and also on $\Hom_{G(\oFFp)}(\sigma,S(U^p,k_E)_{\m^{\Sigma}})$. Identifying $G_p(\Zp)$ with $\prod_{v\vert p}\!\GL_n({\mathcal O}_{\!F_{\tilde v}})$ via $\prod_{v\vert p}\iota_{\tilde v}$, a Serre weight $\sigma$ for $G_p(\Fp)$ is of the form $\sigma=\otimes_{v\vert p}\sigma_{\tilde v}$, where $\sigma_{\tilde v}$ is a Serre weight for $\GL_n({\mathcal O}_{\!F_{\tilde v}})=\GL_n(\Zp)$. Using the proof of \cite[Lem.\ 8.2]{He2} we have ${\mathcal H}_{G_p}(\sigma)=\otimes_{v\vert p}{\mathcal H}_{\GL_n}(\sigma_{\tilde v})$. By Example \ref{plustard} we thus have an isomorphism
\begin{equation}\label{isov}
{\mathcal H}_{G_p}(\sigma)\cong k_E[T_{\sigma_{\tilde v},1},\dots,T_{\sigma_{\tilde v},n-1},T_{\sigma_{\tilde v},n}^{\pm 1},\ v\vert p].
\end{equation}
Recall that $\Hom_{G(\oFFp)}(\sigma,S(U^p,k_E)_{\m^{\Sigma}})^{\ord}\subseteq \Hom_{G(\oFFp)}(\sigma,S(U^p,k_E)_{\m^{\Sigma}})$ was defined in \S\ref{locres}. Note that it is compatible with base change for algebraic extensions of $k_E$.

\begin{prop}\label{prop:ordinarygal}
  Suppose that $U^p \subseteq \GAp$ is a compact open subgroup, $\Sigma$ is a finite set of places as above, and that $\rbar
  : \gF\rightarrow {\GL}_n(k_E)$ is absolutely irreducible.  Suppose that $\sigma = \otimes_{v | p} \sigma_{\tilde v} =
  \otimes_{v | p} F(\lambda_{\tilde v})$ is a Serre weight for $G(\oFFp)$ \(so $0\leq \lambda_{\tilde v,i}-\lambda_{\tilde
    v,i+1}\leq p-1$ for all $i$\) and that $\eta : {\mathcal H}_{G_p}(\sigma) \to k_E$ is an ordinary character. If
  $\Hom_{G(\oFFp)}\big(\sigma,S(U^p,k_E)_{\m^{\Sigma}}\big)[\eta]\ne 0$, then
  \begin{equation*}
    \rbar_{\tilde v}\cong \begin{pmatrix}\omega^{\lambda_{\tilde v,1}}\nr(u_{\tilde v,1})&*&\cdots &*\\
      0&\omega^{\lambda_{\tilde v,2}-1}\nr(\frac{u_{\tilde v,2}}{u_{\tilde v,1}})&\varddots &\vdots\\
      \vdots & \varddots & \varddots & *\\
      0&\dots & 0 & \omega^{\lambda_{\tilde v,n}-(n-1)}\nr(\frac{u_{\tilde v,n}}{u_{\tilde v,n-1}})\end{pmatrix},
  \end{equation*}
  where $u_{\tilde v, j} = \eta(T_{\sigma_{\tilde v},n}^{-1}T_{\sigma_{\tilde v},n-j}) \in k_E^\times$.
\end{prop}

\begin{proof}
  Step 1: We make some easy reductions.\\
  Let $\sigma$ be any Serre weight for $G_p(\Fp)$. Since the actions of $\TT^{\Sigma}$ and ${\mathcal H}_{G_p}(\sigma)$
  commute we have
  \begin{eqnarray*}
    \Hom_{G(\oFFp)}\big(\sigma,S(U^p,k_E)_{\m^{\Sigma}}\big)^{\ord}&=&\Hom_{G(\oFFp)}\big(\sigma,S(U^p,k_E)\big)^{\ord}_{\m^{\Sigma}}\\
    &=&S\big(U^pG(\oFFp),\sigma^\vee\big)^{\ord}_{\m^{\Sigma}},
  \end{eqnarray*}
  where the second equality follows from Lemma \ref{tensor}. Since $S(U^p,k_E)_{\m^\Sigma}\subseteq
  S({U'}^p,k_E)_{\m^{\Sigma'}}$ if ${U'}^p\subseteq U^p$ and $\Sigma'\supseteq \Sigma$ (see \S\ref{set}), we have an analogous
  inclusion with the $\Hom_{G(\oFFp)}(\sigma,\cdot)^{\ord}$. Thus we can assume $U^p$ sufficiently small.

  \noindent
  Step 2: We check some compatibilities with \cite{Ge} and \cite{GG}.\\
  Let $T_p\subseteq B_p\subseteq G_p$ be the split maximal torus and the Borel subgroup in $G_p$ which, under $\iota$,
  correspond respectively to diagonal and upper triangular matrices in $\prod_{\widetilde v}\!{\GL_n}_{/\Zp}$. We have
  $\sigma^\vee\cong F(\lambda')$, where $\lambda'\in X_1(T_p)$ satisfies $\lambda'_{\tilde v, i} = -\lambda_{\tilde v,n-i+1}$
  for all $\tilde v$, $i$. Let $L(\lambda')_{/{\oE}}=\otimes_{v\vert p}L(\lambda'_{\tilde v})_{/{\oE}}$ be the algebraic
  representation of $G_p\times_{\Zp}\oE$ given by $({\rm ind}_{B^-_p}^{G_p}\lambda')_{/\oE}$ (algebraic induction functor of
  \cite[\S I.3.3]{Ja}). We consider $L(\lambda')$ as a representation of $G_p(\Zp)=G(\oFFp)$ over $\oE$ and we have an
  injection of smooth finite-dimensional $G(\oFFp)$-representations over $k_E$: $F(\lambda')\hookrightarrow
  L(\lambda')\otimes_{\oE}k_E$ (\cite[\S II.2]{Ja}). Let $U^p\subseteq \GAp$ be a sufficiently small compact open subgroup and
  $\Sigma$ a set of finite places as above, then $S(U^pG(\oFFp),L(\lambda')\otimes_{\oE}k_E)\cong
  S(U^pG(\oFFp),L(\lambda'))\otimes_{\oE}k_E$ by Lemma \ref{small} and we deduce an injection of $\TT^{\Sigma}$-modules
  $S(U^pG(\oFFp),F(\lambda'))\hookrightarrow S(U^pG(\oFFp),L(\lambda'))\otimes_{\oE}k_E$ which induces an injection of
  $\TT^{\Sigma}$-modules
  $$S(U^pG(\oFFp),F(\lambda'))_{\m^{\Sigma}}\hookrightarrow S(U^pG(\oFFp),L(\lambda'))_{\m^{\Sigma}}\otimes_{\oE}k_E.$$
  By \cite[Def.\ 2.3.2]{Ge} or \cite[\S6.1]{GG}
  there is an action on $S(U^pG(\oFFp),L(\lambda'))$ (and thus on $S(U^pG(\oFFp),L(\lambda'))_{\m^{\Sigma}}$) of the
  ``weighted'' double cosets
  \begin{eqnarray}\label{dcp}
    p^{-\sum_{i=1}^j\lambda'_{\tilde v,n-i+1}}\iota_{\tilde v}^{-1}\left[{\GL_n}({\mathcal O}_{\!F_{\tilde v}})
      \begin{pmatrix}{\bf 1}_{n-j}&\\& p{\bf 1}_{j}\end{pmatrix}{\GL_n}({\mathcal O}_{\!F_{\tilde v}})\right]
  \end{eqnarray}
  for $v\vert p$ and $j\in \{1,\dots,n\}$ whose reduction modulo $\pE$ preserve the subspace $S(U^pG(\oFFp),F(\lambda'))$
  (and thus $S(U^pG(\oFFp),F(\lambda'))_{\m^{\Sigma}}$). Moreover, by a variant \ of \ \cite[Prop.\ 4.4.2]{EGH}, \ one \ can \ check \ that
  \ the \ action \ of \ (\ref{dcp}) \ on $S(U^pG(\oFFp),\!F(\lambda'))\!\cong \!\Hom_{G(\oFFp)}(\sigma,S(U^p,k_E))$ coincides with the
  action of the operator $T_{\sigma_{\tilde v},n}^{-1}T_{\sigma_{\tilde v},n-j}$ of (\ref{isov}). (In {\it loc.\ cit.}, it
  coincides with the action of $T_{\sigma_{\tilde v},j}$ on $(\sigma^\vee\otimes S(U^p,k_E))^{G(\oFFp)}$ as defined in
  \cite[\S2.2]{EGH}, but this is the same as the action of $T_{\sigma_{\tilde v},n}^{-1}T_{\sigma_{\tilde v},n-j}$ on
  $\Hom_{G(\oFFp)}(\sigma,S(U^p,k_E))$ as defined in \S\ref{locres}, see \cite[\S2.3]{He2}.)

  \noindent
  Step 3: We conclude.\\
  Replacing $E$ by a finite extension if necessary, by \cite[Lem.\ 4.5.1]{EGH} and Step 2, the Hecke eigenvalues
  $u_{\tilde v, j} \in k_E^\times$ of $T_{\sigma_{\tilde v},n}^{-1}T_{\sigma_{\tilde v},n-j}$ on $S(U^pG(\oFFp),F(\lambda'))_{\m^{\Sigma}}$ lift to
  eigenvalues $\widetilde u_{\tilde v, j} \in \oE^{\times}$ of (\ref{dcp}) on the larger space $S(U^pG(\oFFp),L(\lambda'))_{\m^{\Sigma}}
  \otimes_{\oE}k_E$.
  Consider the embedding $S(U^pG(\oFFp),L(\lambda'))_{\m^{\Sigma}}\subseteq
  S(L(\lambda'))\otimes_{\oE}\overline E$, where the latter is a semi-simple $\GA$-representation (see \S\ref{set} and
  \cite[Lem.\ 2.2.5]{Ge}). If $\pi$ is any irreducible constituent of $S(L(\lambda'))\otimes_{\oE}\overline E$ such that
  $\pi\cap S(U^pG(\oFFp),L(\lambda'))_{\m^{\Sigma}}\ne 0$, then its base change to ${\GL_n}_{/F}$ (which exists by
  \cite[Lem.\ 2.2.5]{Ge} and \cite[Cor.\ 5.3]{La}) is a cuspidal automorphic representation of $\GL_n({\mathbb A}_F)$ since
  its associated $p$-adic representation of $\gF$ is irreducible (as it reduces to the irreducible $\rbar$).
  Now assume that the eigenvalues of the operators (\ref{dcp}) on $\pi\cap
  S(U^pG(\oFFp),L(\lambda'))_{\m^{\Sigma}}=\pi^{U^pG(\oFFp)}\cap S(U^pG(\oFFp),L(\lambda'))_{\m^{\Sigma}}$ are all in
  $\oE^{\times}$. We claim the proof of \cite[Cor.\ 2.7.8(1)]{Ge} applies to give the desired result. In our situation, the level
  is prime to $p$, so we do not need the regularity condition on $\lambda$. The characteristic polynomial in the proof becomes
  $\sum_j (-1)^j p^{j(j-1)/2 + \sum_{i=1}^j \lambda'_{\tilde v, n-i+1}} \widetilde u_{\tilde v, j} X^j$ whose roots
  $\alpha_j$ have (distinct) valuations $j-1+\lambda'_{n-j+1}$. Finally observe that in the notation of \cite[Cor.\ 2.7.8(1)]{Ge},
  $\psi_{\tilde v, j}(\mathrm{Art}_{F_{\tilde v}}(p)) = \prod_{i = 1}^j (\alpha_i p^{-(i-1)-\lambda'_{n-i+1}})$ is congruent to
  $\widetilde u_{\tilde v, j}$ modulo $\varpi_E$.
   \end{proof}

Following \cite[\S6]{GG} we say that $\rbar$ (continuous, absolutely irreducible) is {\it modular and ordinary} if there exist a sufficiently small compact open subgroup
$U^p = \prod_{v \nmid p} U_v\subseteq \GAp$ such that $U_v$ is a hyperspecial maximal compact subgroup of $G(F_v^+)$ for all places $v$ of $F^+$ that are inert in $F$,
a finite set of finite places $\Sigma$ as above and a Serre weight $\sigma$ for $G_p(\Fp)$ such that
$$\Hom_{G(\oFFp)}\big(\sigma,S(U^p,k_E)_{\m^{\Sigma}}\big)^{\ord}\ne 0.$$
(Note that the modularity assumption implies $\rbar^c\cong \rbar^\vee\otimes\omega^{1-n}$.) By Proposition~\ref{prop:ordinarygal} if $\rbar$ is modular and ordinary then $\rbar_w$ is ordinary for all $w\vert p$ in $F$.

\begin{prop}\label{Nul}
  Let $\rbar:\gF\rightarrow {\GL}_n(k_E)$ be a continuous representation. We assume that $\rbar$ satisfies the following
  assumptions:
  \begin{enumerate}
  \item $\rbar|_{{\Gal}(\overline F/F(\sqrt[p]{1}))}$ is absolutely irreducible;
  \item $\rbar$ is modular and ordinary;
  \item  $p > 2n+2$ and $\zeta_p \not\in F$.
  \end{enumerate}
  Suppose that for all places $v \vert p$ we have
  \begin{equation}\label{eq:1}
    \rbar_{\tilde v}\cong \begin{pmatrix}\omega^{\lambda_{\tilde v,1}}\nr(u_{\tilde v,1})&*&\cdots &*\\
      0&\omega^{\lambda_{\tilde v,2}-1}\nr(\frac{u_{\tilde v,2}}{u_{\tilde v,1}})&\varddots &\vdots\\
      \vdots & \varddots & \varddots & *\\
      0&\dots & 0 & \omega^{\lambda_{\tilde v,n}-(n-1)}\nr(\frac{u_{\tilde v,n}}{u_{\tilde v,n-1}})\end{pmatrix},
  \end{equation}
  where $1\leq \lambda_{\tilde v,i}-\lambda_{\tilde v,i+1}\leq p-1$ and $u_{\tilde v,i} \in k_E^\times$ for all $i$. Let $\sigma = \otimes_{v | p} \sigma_{\tilde v} := \otimes_{v | p} F(\lambda_{\tilde v})$, a Serre weight for $G(\oFFp)$.

  Then there exist a sufficiently small compact open subgroup $U^p\subseteq \GAp$ and a finite set of finite places $\Sigma$
  as above such that
  \begin{equation*}
    \Hom_{G(\oFFp)}\big(\sigma,S(U^p,k_E)_{\m^{\Sigma}}\big)^{\ord}\ne 0.
  \end{equation*}

  If moreover the integers $\lambda_{\tilde v,i}-(i-1)$ are distinct modulo $p-1$ \(for any $v \vert p$\), then for any such $(U^p,\Sigma)$ and any finite
extension $k'_E$ of $k_E$, there is only one ordinary character $\eta$ of ${\mathcal H}_{G_p}(\sigma)$ such that
$\Hom_{G(\oFFp)}\big(\sigma,S(U^p,k'_E)_{\m^{\Sigma}}\big)[\eta]\ne 0$ and this character sends $T_{\sigma_{\tilde v},j}$ in
\eqref{isov} to $u_{\tilde v,n}^{-1}u_{\tilde v,n-j}\in k_E^{\times}$ in \eqref{eq:1} for $v\vert p$ and $1\leq j\leq n$
\(where $u_{\tilde v,0}:=1$ for all $v\vert p$\).
\end{prop}
\begin{proof}
  By \cite[Lem.\ 3.1.5]{GG} it follows that for all $v \vert p$, $\rbar_{\tilde v}$ has an ordinary crystalline lift of
  Hodge--Tate weights $(\lambda_{\tilde v,1}, \lambda_{\tilde v,2}-1,\dots, \lambda_{\tilde v,n}-(n-1))$. The first part of
  the proposition now follows from \cite[Thm.\ 6.1.6]{GG} under assumptions (i$'$) $\rbar({\Gal}(\overline F/F(\sqrt[p]{1})))$
  big and (iii$'$) ${\overline F}^{\Ker({\rm ad}(\rbar))}$ does not contain $F(\sqrt[p]{1})$ (instead of (i) and (iii)).
  (Note that the extra condition on the level $U$ at places $v \in S_a$ in \cite[\S6.1]{GG} is only used to guarantee that $U$ is
  sufficiently small.) To conclude, we replace the application of \cite[Thm.\ 5.1.1]{GG} in the proof of \cite[Thm.\ 6.1.6]{GG}
  by the proof of \cite[Thm.\ 4.4.1]{blggt2} and Thorne's improved modularity theorem \cite[Thm.\ 2.3.1]{blggt2} (to guarantee
  that $\pi'$ can be chosen to be ordinary of level prime to $p$). (Note that crystalline ordinary local Galois
  representations are potentially diagonalizable and that if $\rho_1 \sim \rho_2$ in the notation of
  \cite[\S1.4]{blggt2}, then $\rho_1$ is crystalline ordinary if and only if $\rho_2$ is crystalline ordinary; see
  \cite[\S1.4]{blggt2}.) The last part follows from Proposition~\ref{prop:ordinarygal}.
\end{proof}

Note that the congruence condition on the $\lambda_{\tilde v,i}-(i-1)$ is satisfied whenever $\rbar_{\tilde v}$ is inertially generic.

Assume that, for all $w\vert p$ in $F$, $\rbar_w$ is ordinary (i.e.\ upper triangular up to conjugation) and generic in the sense of Definition \ref{genbar} (equivalently, assume $\rbar_{\tilde v}$ is ordinary and generic at all places $\tilde v$ where $v\vert p$ in $F^+$). We have associated to $\rbar_{\tilde v}\otimes \omega^{n-1}$ in Definition \ref{ordser} a set of ordinary Serre weights for $\iota_{\tilde v}(G(\oFFv))=\GL_n({\mathcal O}_{\!F_{\tilde v}})$ defined as the set of irreducible constituents of the $\GL_n({\mathcal O}_{\!F_{\tilde v}})$-socle of $\Pi(\rbar_{\tilde v})^{\ord}\otimes (\omega^{n-1}\circ\det)$. By Lemma \ref{trans} (more precisely its variant for $\rbar_{\tilde v}$) and the discussion that follows, the set of ordinary Serre weights of $\rbar_{\tilde v}\otimes \omega^{n-1}$ considered as a set of $G(\oFFv)$-representations via $\iota_{\tilde v}$ is the same as the set of ordinary Serre weights of $\rbar_{\tilde v^c}\otimes \omega^{n-1}$ considered as a set of $G(\oFFv)$-representations via $\iota_{\tilde v^c}$. We say that a Serre weight $\sigma=\otimes_{v\vert p}\sigma_{\tilde v}$ for $G_p(\Fp)$ is an {\it ordinary Serre weight of $\rbar\otimes \omega^{n-1}$} if, for every $v\vert p$, $\sigma_{\tilde v}$ is an ordinary Serre weight of $\rbar_{\tilde v}\otimes \omega^{n-1}$. This doesn't depend on the choice of $\tilde v$ above $v$.

\begin{cor}\label{cor:Nul}
  Suppose that $\rbar:\gF\rightarrow {\GL}_n(k_E)$ satisfies assumptions \(i\)--\(iii\) of Proposition~\ref{Nul}, and suppose in addition
  that
  \begin{enumerate}
    \setcounter{enumi}{3}
  \item for all $w\vert p$ in $F$, $\rbar_w$ is \(ordinary\) inertially generic \(Definition \ref{genbarfor}\).
  \end{enumerate}
  Then there exist a sufficiently small compact open subgroup $U^p\subseteq \GAp$ and a finite set of finite places $\Sigma$
  as above such that, for any Serre weight $\sigma$ for $G(\oFFp)$, we have
  \begin{equation*}
    \Hom_{G(\oFFp)}\big(\sigma,S(U^p,k_E)_{\m^{\Sigma}}\big)^{\ord}\ne 0
  \end{equation*}
  if and only if $\sigma$ is an ordinary Serre weight of $\rbar\otimes\omega^{n-1}$.
\end{cor}

\begin{proof}
  This follows immediately from Propositions \ref{prop:ordinarygal} and \ref{Nul}, and Example~\ref{exord}.
\end{proof}

The following statement is our main result and can be seen as a weak form of Conjecture \ref{theconjbar}. Recall that an injection $\pi\hookrightarrow \Pi$ between smooth representations of $G(F^+\otimes_{\Q}\Qp)$ over $k_E$ is said to be {\it essential} if, for any nonzero subrepresentation $\pi'\subseteq \Pi$, we have $\pi\cap \pi'\ne 0$ in $\Pi$.

\begin{thm}\label{ouf!}
Let $\rbar:\gF\rightarrow {\GL}_n(k_E)$ be a continuous representation that satisfies conditions \(i\) to \(iv\) of Proposition \ref{Nul} and Corollary \ref{cor:Nul} \(so in particular $p>2n+2$\). Fix a good conjugate of each $\rbar_{\tilde v}$ and, for every ordinary Serre weight $\sigma_{\tilde v}$ of $\rbar_{\tilde v}\otimes \omega^{n-1}$, let $w_{\sigma_{\tilde v}}\in W_{C_{\rbar_{\tilde v}}}$ be the unique element corresponding to $\sigma_{\tilde v}\otimes (\omega^{1-n}\circ \det)$ in Proposition \ref{ordserex} applied to $\rhobar=\rbar_{\tilde v}$. Fix $U^p\subseteq \GAp$ a sufficiently small compact open subgroup and $\Sigma$ a finite set of finite places as in Corollary \ref{cor:Nul}. Then, for each ordinary Serre weight $\sigma=\otimes_{v\vert p}\sigma_{\tilde v}$ of $\rbar\otimes \omega^{n-1}$, there is an integer $d_{\sigma}>0$ such that we have an essential injection of admissible smooth representations of $G(F^+\otimes_{\Q}\Qp)=\prod_{v\vert p}G(F_v^+)$ over $k_E$,
\begin{equation}\label{beta2}
\bigoplus_{\sigma=\otimes\sigma_{\tilde v}} \bigg({\bigotimes_{v\vert p}}\Big(\Pi(\rbar_{\tilde v})_{C_{\rbar_{\tilde v}},w_{\sigma_{\tilde v}}}\otimes (\omega^{n-1}\circ\det)\Big)\bigg)^{\oplus d_{\sigma}}\hookrightarrow S(U^p,k_E)[\m^{\Sigma}]^{\ord}.
\end{equation}
\end{thm}
\begin{proof}
For each ordinary Serre weight $\sigma$ of $\rbar\otimes\omega^{n-1}$ let
\begin{eqnarray*}
d_{\sigma}&:=&\dim_{k_E}\Hom_{G(\oFFp)}\big(\sigma,S(U^p,k_E)[\m^{\Sigma}]\big)[\eta],
\end{eqnarray*}
where $\eta$ is the unique ordinary character of ${\mathcal H}_{G_p}(\sigma)$ in Proposition \ref{Nul}. We have in particular
$d_{\sigma}>0$. Write $\sigma_{\tilde v}=F(\lambda_{\tilde v})$ with $\lambda_{\tilde v}=(\lambda_{\tilde
  v,1},\dots,\lambda_{\tilde v,n})\in \Z^n$ ($1\leq \lambda_{\tilde v,i}-\lambda_{\tilde v,i+1}\leq p-1$) and set
$\eta_{\tilde v}:=\eta\vert_{{\mathcal H}_{\GL_n}(\sigma_{\tilde v})}$ (recall ${\mathcal H}_{G_p}(\sigma)=\otimes_{v\vert
  p}{\mathcal H}_{\GL_n}(\sigma_{\tilde v})$). By definition of $w_{\sigma_{\tilde v}}$ we have $w_{\sigma_{\tilde v}}^{-1}
(\wh \chi_{\rbar_{\tilde v} \otimes \omega^{n-1}}) = \diag(\omega^{\lambda_{\tilde v,1}+n-1}\nr(u_{\tilde v,1}),\linebreak[1] \omega^{\lambda_{\tilde
    v,2}+n-2}\nr({\scriptstyle\frac{u_{\tilde v,2}}{u_{\tilde v,1}}}),\linebreak[1] \dots,\linebreak[1] \omega^{\lambda_{\tilde
    v,n}}\nr({\scriptstyle\frac{u_{\tilde v,n}}{u_{\tilde v,n-1}}}))$ \ for \ some \ $u_{\tilde v,i} \in k_E^\times$, \ and \ by
\ Proposition~\ref{Nul} \ we \ have $\eta(T_{\sigma_{\tilde v}, j}) = u_{\tilde v,n}^{-1} u_{\tilde v,n-j}$.
Thus we get from Example \ref{plustard2} that
\begin{multline*}
\big(\ind_{\GL_n({\mathcal O}_{\!F_{\tilde v}})}^{\GL_n(F_{\tilde v})}\sigma_{\tilde v}\big)\otimes_{{\mathcal H}_{\GL_n}(\sigma_{\tilde v}),\eta_{\tilde v}}k_E\cong \\
\Ind_{B^-(F_{\tilde v})}^{\GL_n(F_{\tilde v})}\omega^{\lambda_{\tilde v,1}}\nr(u_{\tilde v,1})\otimes \omega^{\lambda_{\tilde v,2}}\nr({\scriptstyle\frac{u_{\tilde v,2}}{u_{\tilde v,1}}})\otimes \cdots \otimes \omega^{\lambda_{\tilde v,n}}\nr({\scriptstyle\frac{u_{\tilde v,n}}{u_{\tilde v,n-1}}})
\end{multline*}
and hence that
\begin{multline}\label{explicit}
\big(\ind_{\GL_n({\mathcal O}_{\!F_{\tilde v}})}^{\GL_n(F_{\tilde v})}\sigma_{\tilde v}\big)\otimes_{{\mathcal H}_{\GL_n}(\sigma_{\tilde v}),\eta_{\tilde v}}k_E\cong \Ind_{B^-(F_{\tilde v})}^{\GL_n(F_{\tilde v})} w_{\sigma_{\tilde v}}^{-1}(\chi_{\rbar_{\tilde v}\otimes \omega^{n-1}})\cdot(\omega^{-1}\circ\theta),
\end{multline}
where $\chi_{\rbar_{\tilde v}\otimes \omega^{n-1}}=\chi_{\rbar_{\tilde v}}\otimes (\omega^{n-1}\circ\det)$ is as in \S\ref{variant2}.
Let $T_p$ and $B_p$ be as in Step 2 of the proof of Proposition \ref{prop:ordinarygal}. By Proposition \ref{injsigma}, Corollary \ref{cor:Nul} and (\ref{explicit}) we can apply Theorem \ref{main} to $(G_p,B_p,T_p)$, $\Pi:=S(U^p,k_E)_{\m^{\Sigma}}$, $\rhobar:=\oplus_{v\vert p}(\rbar_{\tilde v}\otimes \omega^{n-1})$, the Serre weight $\sigma$ and the character $\eta_{\rhobar,w_{C_{\rhobar}}}:=\eta$. By the analogue over $k_E$ of Remark \ref{rema3} and Corollary \ref{prat}, we get that the restriction to the $G(\oFFp)$-socle induces an isomorphism of $\TT^\Sigma$-modules
\begin{multline*}
\Hom_{G(F^+\otimes_{\Q}\Qp)}\bigg({\bigotimes_{v\vert p}}\Big(\Pi(\rbar_{\tilde v})_{C_{\rbar_{\tilde v}},w_{\sigma_{\tilde v}}}\otimes (\omega^{n-1}\circ\det)\Big),S(U^p,k_E)_{\m^{\Sigma}}\bigg)\buildrel\sim\over\longrightarrow \\
\Hom_{G(\oFFp)}\big(\sigma,S(U^p,k_E)_{\m^{\Sigma}}\big)[\eta].
\end{multline*}
Taking the $\m^{\Sigma}$-eigenspaces on both sides and using the fact that all constituents of $\otimes_{v\vert p}\Pi(\rbar_{\tilde v})_{C_{\rbar_{\tilde v}},w_{\sigma_{\tilde v}}}$ are principal series, we deduce
\begin{multline}\label{close}
\Hom_{G(F^+\otimes_{\Q}\Qp)}\bigg({\bigotimes_{v\vert p}}\Big(\Pi(\rbar_{\tilde v})_{C_{\rbar_{\tilde v}},w_{\sigma_{\tilde v}}}\otimes (\omega^{n-1}\circ\det)\Big),S(U^p,k_E)[\m^{\Sigma}]^{\ord}\bigg)\buildrel\sim\over\longrightarrow \\
\Hom_{G(\oFFp)}\big(\sigma,S(U^p,k_E)[\m^{\Sigma}]\big)[\eta].
\end{multline}
Let $f_1,\dots,f_{d_{\sigma}}$ be a $k_E$-basis of the right-hand side of (\ref{close}) and $F_1,\dots,F_{d_{\sigma}}$ the corresponding basis of the left-hand side, then $\oplus_{i=1}^{d_{\sigma}}F_i$ induces a $G(F^+\otimes_{\Q}\Qp)$-equivariant map
$$ \bigg({\bigotimes_{v\vert p}}\Big(\Pi(\rbar_{\tilde v})_{C_{\rbar_{\tilde v}},w_{\sigma_{\tilde v}}}\otimes (\omega^{n-1}\circ\det)\Big)\bigg)^{\oplus d_{\sigma}}\longrightarrow S(U^p,k_E)[\m^{\Sigma}]^{\ord},$$
which is injective as it is injective on the $G(\oFFp)$-socle. Summing over all $\sigma$, we get by the same argument an injection as in (\ref{beta2}). It remains to prove that it is essential, and, since $S(U^p,k_E)[\m^{\Sigma}]^{\ord}\otimes_{k_E}\overline k_E\hookrightarrow S(U^p,\overline k_E)[\m^{\Sigma}]^{\ord}$, it is enough to prove this replacing $k_E$ by an algebraic closure $\overline k_E$. Assume the injection is not essential and let $\Pi'$ be a nonzero subrepresentation of $S(U^p,\overline k_E)[\m^{\Sigma}]^{\ord}$ which has zero intersection with the left-hand side of (\ref{beta2}) (tensored by $\overline k_E$). Since $\Pi'$ is admissible, it satisfies the descending chain condition and, replacing $\Pi'$ by an irreducible subrepresentation, we can assume $\Pi'$ irreducible. Since it lies in $S(U^p,\overline k_E)[\m^{\Sigma}]^{\ord}$ it is thus an irreducible subquotient of a principal series of $G_p(\Qp)$ over $\overline k_E$. Let $\sigma'\subseteq \Pi'$ be a Serre weight for $G_p(\Fp)$, then by \cite[Cor.\ 9.13(i)]{He2} (more precisely the last sentence in {\it loc.\ cit.}) the action of ${\mathcal H}_{G_p}(\sigma')$ on $\Hom_{G_p(\Zp)}(\sigma',\Pi')$ factors through ${\mathcal S}_{T_p}$ (see \S\ref{locres} for ${\mathcal S}_M$) and thus there is an injection $\sigma'\hookrightarrow \Pi'\hookrightarrow S(U^p,\overline k_E)[\m^{\Sigma}]^{\ord}$ which gives an element in $\Hom_{G(\oFFp)}(\sigma',S(U^p,\overline k_E)[\m^{\Sigma}])[\eta']$ for some ordinary character $\eta'$ of ${\mathcal H}_{G_p}(\sigma')$. But Corollary \ref{cor:Nul} and the isomorphism
$$\Hom_{G(\oFFp)}\big(\sigma',S(U^p,k_E)[{\m^{\Sigma}}]\big)\otimes_{k_E}\overline k_E\buildrel\sim\over\longrightarrow \Hom_{G(\oFFp)}\big(\sigma',S(U^p,\overline k_E)[\m^{\Sigma}]\big)$$
(recall that $\sigma'$ is defined over $k_E$) imply that $\sigma'$ must be an ordinary Serre weight $\sigma$ of $\rbar\otimes\omega^{n-1}$ and that $\eta'$ must be $\eta$, so that $\sigma'$ can't have a zero intersection with the left-hand side of (\ref{beta2}) by construction. This proves we must have $\Pi'=0$.
\end{proof}

The integers $d_{\sigma}$ in Theorem \ref{ouf!} {\it a priori} depend on all the data, that is, on $\rbar$, $U^p$, $\Sigma$ and $\sigma$. Note that if all $d_{\sigma}$ are equal, and if $d$ is their common value, then one recovers an essential injection as in (\ref{beta1}).

We end this paper with some evidence for the $p$-adic case, i.e.\ Conjecture \ref{theconj}. If $r:\gF\rightarrow {\GL}_n(E)$ is a continuous representation, we denote by $\overline r$ the semi-simplification of its mod $\pE$ reduction. If $r$ is ramified only at a finite number of places of $F$, we recall that $r$ is modular if there exist a compact open subgroup $U^p\subseteq \GAp$, a finite set $\Sigma$ of finite places of $F^+$ (containing the set of places of $F^+$ that split in $F$ and divide $pN$, the set of places of $F^+$ that split in $F$ at which $U^p$ is not unramified and the set of places of $F^+$ that split in $F$ at which $r$ is ramified) and irreducible algebraic representations $L(\lambda_{\tilde v})$ of $G\times_{\oFF[1/N]}F^+_v\simeq {\GL_n}_{/F_{\tilde v}}$ over $E$ for $v\vert p$ of highest weight $\lambda_{\tilde v}=(\lambda_{{\tilde v},1}\geq \dots \geq \lambda_{{\tilde v},n})$ such that $\Hom_{U_p}\big(\bigotimes_{v\vert p}L(\lambda_{\tilde v}),\widehat S(U^p,E)[\p^{\Sigma}]\big)\ne 0$ for $U_p$ a small enough compact open subgroup of $\prod_{v\vert p}G(\oFFv)$ (see \cite[\S7.1.4]{EGH} and note that $\Hom_{U_p}\big(\bigotimes_{v\vert p}L(\lambda_{\tilde v}),\widehat S(U^p,E)\big)=S\big(U^pU_p,\otimes_{v\vert p}L(\lambda'_{\tilde v})\big)$ where $\lambda'_{\tilde v, i} := -\lambda_{\tilde v,n-i+1}$; also recall that the prime ideal $\p^{\Sigma}$ of $\TT^\Sigma$ was defined in \S\ref{globconj}). Moreover $r$ modular implies that $r_{\tilde v}$ is potentially semi-stable for all $v\vert p$ and that $\lambda_{{\tilde v},1}> \lambda_{{\tilde v},2}-1> \cdots > \lambda_{{\tilde v},n}-(n-1)$ are its Hodge--Tate weights (see \cite[Thm.\ 7.2.1]{EGH} together with our convention at the end of Remark \ref{HT}).

\begin{thm}\label{ouf!2}
Let $r:\gF\rightarrow {\GL}_n(E)$ be a continuous representation that satisfies the following assumptions:
 \begin{enumerate}
  \item $r$ is modular;
  \item $\rbar|_{{\Gal}(\overline F/F(\sqrt[p]{1}))}$ is absolutely irreducible;
  \item $r_w$ is generic ordinary for all $w\vert p$ in $F$;
  \item $(\rbar)_w$ is \(ordinary\) inertially generic with $C_{(\rbar)_w}$ maximal for all $w\vert p$ in $F$;
  \item $p > 2n+2$ and $\zeta_p \not\in F$.
 \end{enumerate}
Fix $U^p\subseteq \GAp$ a sufficiently small compact open subgroup and a finite set $\Sigma$ of finite places of $F^+$ as above such that $\Hom_{U_p}\big(\bigotimes_{v\vert p}L(\lambda_{\tilde v}),\widehat S(U^p,E)[\p^{\Sigma}]\big)\ne 0$ for $U_p$ small enough. Then there is an integer $d\in \Z_{>0}$ such that we have an injection of admissible unitary continuous representations of $G(F^+\otimes_{\Q}\Qp)=\prod_{v\vert p}G(F_v^+)$ over $E$,
\begin{equation}\label{beta3}
\bigg(\mathop{\widehat\bigotimes}\limits_{v\vert p}\Big(\Pi(r_{\tilde v})^{\ord}\otimes (\varepsilon^{n-1}\circ\det)\Big)\bigg)^{\oplus d}\hookrightarrow \widehat S(U^p,E)[\p^{\Sigma}]^{\ord}.
\end{equation}
\end{thm}
\begin{proof}
Note that assumptions (i) and (iii) imply that $\rbar\otimes\omega^{n-1}$ is modular and ordinary. Assumption (iv) implies $W_{C_{(\rbar)_w}}=\{1\}$, or equivalently that $(\rbar)_w$ is upper triangular and ``as indecomposable as possible'' (see \S\ref{Ch}), or equivalently again by Corollary \ref{cor:Nul} that $\rbar\otimes\omega^{n-1}$ has only one ordinary Serre weight. We also get from assumptions (iii) and (iv) that $W_{C_{r_w}}=\{1\}$, or equivalently that $\Pi(r_{w})^{\ord}$ has an irreducible socle. Moreover it follows from assumptions (i) and (iii) that $r_w$ is potentially crystalline for all $w\vert p$ in $F$ with associated Weil--Deligne representation corresponding to an irreducible smooth principal series of ${\GL_n}(F_w)$ over $E$. By \cite[Thm.\ 7.2.1(iv)]{EGH} and the well-known description of the locally algebraic vectors of $\widehat S(U^p,E)$ (see e.g.\ \cite[Prop.\ 5.1]{Br5}), we get a $\prod_{v\vert p}G(F_v^+)$-equivariant injection
$$\mathop{\bigotimes}\limits_{v\vert p}\Big(\Ind_{B^-(F_{\tilde v})}^{\GL_n(F_{\tilde v})} (\chi_{r_{\tilde v}\otimes \varepsilon^{n-1}})\cdot(\varepsilon^{-1}\circ\theta)\Big)^{\alg}\hookrightarrow \widehat S(U^p,E)[\p^{\Sigma}]$$
where ``alg'' means the locally algebraic vectors of the irreducible unitary continuous principal series $\big(\Ind_{B^-(F_{\tilde v})}^{\GL_n(F_{\tilde v})} (\chi_{r_{\tilde v}\otimes \varepsilon^{n-1}})\cdot(\varepsilon^{-1}\circ\theta)\big)^{{\mathcal C}^0}$. Either by an application of Emerton's functor of ordinary parts or by the fact that $\big(\Ind_{B^-(F_{\tilde v})}^{\GL_n(F_{\tilde v})} (\chi_{r_{\tilde v}\otimes \varepsilon^{n-1}})\cdot(\varepsilon^{-1}\circ\theta)\big)^{{\mathcal C}^0}$ is the universal unitary completion of $\big(\Ind_{B^-(F_{\tilde v})}^{\GL_n(F_{\tilde v})} (\chi_{r_{\tilde v}\otimes \varepsilon^{n-1}})\cdot(\varepsilon^{-1}\circ\theta)\big)^{\alg}$, we have
\begin{multline*}
\Hom_{G(F^+\otimes_{\Q}\Qp)}\Big( \mathop{\bigotimes}\limits_{v\vert p} \big(\Ind_{B^-(F_{\tilde v})}^{\GL_n(F_{\tilde v})} (\chi_{r_{\tilde v}\otimes \varepsilon^{n-1}})\cdot(\varepsilon^{-1}\circ\theta)\big)^{\alg}, \widehat S(U^p,E)[\p^{\Sigma}]\Big)=\\\Hom_{G(F^+\otimes_{\Q}\Qp)}\Big( \mathop{\widehat\bigotimes}\limits_{v\vert p} \big(\Ind_{B^-(F_{\tilde v})}^{\GL_n(F_{\tilde v})} (\chi_{r_{\tilde v}\otimes \varepsilon^{n-1}})\cdot(\varepsilon^{-1}\circ\theta)\big)^{{\mathcal C}^0}, \widehat S(U^p,E)[\p^{\Sigma}]\Big).
\end{multline*}
Let $d$ be the (finite positive) dimension of this space of homomorphisms, then, as in the proof of Theorem \ref{ouf!}, one applies Corollary \ref{mainpadic} to $\Pi:=\widehat S(U^p,E)_{\p^{\Sigma}}$ and $\rho:=\oplus_{v\vert p}(r_{\tilde v}\otimes \varepsilon^{n-1})$, and concludes as in the proof of Theorem \ref{ouf!}.
\end{proof}

\begin{rem}
(i) With some more work, one should be able to prove that the injection (\ref{beta3}) is essential for $d$ as in the above proof, at least in some cases. First, it follows from Corollary \ref{prat} and assumption (v) in the statement of Theorem \ref{ouf!2} that an irreducible constituent in the $G(F^+\otimes_{\Q}\Qp)$-socle of $\widehat S(U^p,E)[\p^{\Sigma}]^{\ord}$ is always a {\it full} principal series (because it reduces to a full principal series). Secondly (at least in some cases), one should be able to replace the control of the $G(\oFFp)$-socle of $S(U^p,k_E)[\m^{\Sigma}]^{\ord}$ given by Corollary \ref{cor:Nul} by a control of the locally analytic vectors of the $G(F^+\otimes_{\Q}\Qp)$-socle of $S(U^p,E)[\p^{\Sigma}]^{\ord}$ given by \cite[Prop.\ 8.1]{Br5} and \cite[\S9]{Br5}, and get that any principal series in this socle should always be (one copy of) $\mathop{\widehat\bigotimes} \big(\Ind_{B^-(F_{\tilde v})}^{\GL_n(F_{\tilde v})} (\chi_{r_{\tilde v}\otimes \varepsilon^{n-1}})\cdot(\varepsilon^{-1}\circ\theta)\big)^{{\mathcal C}^0}$. See also the recent \cite{BC} for stronger results on Conjecture \ref{theconj} when $n=3$.\\
(ii) When $G={\rm GSp}_4$, analogous results to those of Theorem \ref{ouf!} and Theorem \ref{ouf!2} should follow by the same method from recent results on ordinary Serre weights for ${\rm GSp}_4$ (see \cite{HT} and \cite[\S7]{GG}) and the transfer from compact mod centre ${\rm GSp}_4$ to ${\rm GL}_4$ (see for example \cite[Thm.\ B]{So}, as well as \cite[Thm.\ 12.1]{GT}), showing that the representation $\Pi(\rho)^{\rm ord}$ hopefully remains relevant beyond ${\rm GL}_n$.\\
(iii) Finally, it is natural to wonder what happens with respect to $L$-packets when the classical Langlands correspondence is involved, that is, when $\rho$ is potentially semi-stable with distinct Hodge--Tate weights. Let us assume for simplicity that $\rho$ is crystalline (and generic ordinary as in Definition \ref{gen}). In that case, the Weil representation associated to $\rho$ is the same as the Weil representation $W(\widehat\chi_{\rho})$ associated to $\widehat\chi_{\rho}:\gp\rightarrow \widehat{T}(E)$ as in \S\ref{cons}. The image of $W(\widehat\chi_{\rho})$ in $\widehat{T}(E)$ is generated by one (semi-simple) element $t$ since $W(\widehat\chi_{\rho})$ is unramified. As the center of $G$ is connected, a classical result of Steinberg (\cite[\S3.9]{spr-st}) implies that the centralizer $Z(t)$ of $t$ in $\widehat G$ is a connected reductive algebraic group. In particular, the associated $L$-packet is a singleton (an irreducible unramified principal series of $G(\Qp)$). Therefore, at least in this case, no (classical) endoscopic phenomenon occurs.
\end{rem}

\appendix

\section{Some results on unitary continuous representations I}

In this appendix, we give technical results on admissible unitary continuous representations of $p$-adic analytic groups that are used in the text. We didn't try to reach the greatest generality.

We refer to \cite[\S2]{Sc1} and \cite[\S2]{Em2} for the definition and basic properties of the abelian category of admissible unitary continuous representations of a $p$-adic analytic group $G$ on $p$-adic Banach spaces over $E$ (more precisely, with the notation of \cite[\S2]{Em2}, it is the category ${\rm Mod}_G^{\pE-{\rm adm}}(\oE)^{\rm fl}\otimes \Q$, that is, the category ${\rm Mod}_G^{\pE-{\rm adm}}(\oE)^{\rm fl}$ of \cite[Prop.\ 2.4.10]{Em2} up to isogeny). If $\Pi,\Pi'$ are two admissible unitary continuous representations of $G$ over $E$, we denote by ${\rm Ext}^1_G(\Pi',\Pi)$ the $E$-vector spaces of Yoneda extensions in this abelian category. 

We didn't seek to be optimal in the statements that follow as they suffice for our purpose (thus some of them might still be true under weaker assumptions).

\begin{lem0}\label{proj}
Let $(\cdots \rightarrow M_n\rightarrow M_{n-1}\rightarrow \cdots)$ be a projective system of $\oE/\pE^n$-modules ($n\in \Z_{>0}$) such that $M:= \plim{n}{M_n}$ is of finite type over $\oE$. Let $\Pi$ be a $p$-adic Banach space over $E$ and $\Pi^0\subseteq \Pi$ be a unit ball. Then there is a topological isomorphism
$$M\otimes_{\oE}\Pi\cong \big(\plim{n}({M_n}\otimes_{\oE}\Pi^0)\big)\otimes_{\oE}E,$$
where ${M_n}\otimes_{\oE}\Pi^0={M_n}\otimes_{\oE}\Pi^0/\pE^n$ is endowed with the discrete topology and $\plim{n}({M_n}\otimes_{\oE}\Pi^0)$ with the projective limit topology.
\end{lem0}
\begin{proof}
Since $E$ is discretely valued, the Banach space $\Pi$ admits an orthonormalizable basis $(e_i)_{i\in I}$ by \cite[Prop.\ 10.1]{Sc2} and we can take $\Pi^0$ to be the unit ball for that basis. Then $\plim{n}({M_n}\otimes_{\oE}\Pi^0)=\plim{n}({M_n}\otimes_{\oE}(\oplus_{i\in I}\oE e_i))$ which can be identified with the $\oE$-module $\widehat M:=\{(m_i)_{i\in I} : m_i\in M, m_i\rightarrow 0{\ \rm when\ }i\rightarrow \infty\}$, where the convergence condition means that for any $n\geq 1$ there exists a finite subset $I_n\subseteq I$ such that $m_i\in \Ker(M\rightarrow M_n)$ if $i\notin I_n$. Since $M= \plim{}{M_n}$, the $\oE$-submodules $\Ker(M\rightarrow M_n)$ form a basis of open neighbourhoods of $0$ in $M$ for the natural profinite (i.e.\ $\pE$-adic) topology on $M$, and this convergence condition is equivalent to asking that for any $n\geq 1$ there exists a finite subset $I_n\subseteq I$ such that $m_i\in \pE^nM$ if $i\notin I_n$. In other terms $\widehat M$ together with its projective limit topology is the $\pE$-adic completion of $M\otimes _{\oE}(\oplus_{i\in I}\oE e_i)$. Since $M$ is of finite type over $\oE$, this is just $M\otimes_{\oE}\Pi^0$.
\end{proof}

\begin{lem0}\label{tf}
Let $M$ be an $\oE$-module such that each element is killed by a power of $\pE$, it has no nonzero divisible element and the submodule of elements killed by $\pE$ is a finite-dimensional $k_E$-vector space. Then $M$ is of finite type over $\oE$.
\end{lem0}
\begin{proof}
For $n\in \Z_{>0}$ let $M[\pE^n]\subseteq M$ be the submodule of elements killed by $\pE^n$ and let $r$ be the dimension of $M[\pE]$. Then $M[\pE^n]$ is of finite type (as follows from the exact sequence $0\rightarrow M[\pE]\rightarrow M[\pE^n]\buildrel\pE\over\rightarrow M[\pE^{n-1}]$ and a straightforward induction on $n$) and isomorphic to $\oplus_{i=1}^r\oE/(\pE^{d_i(n)})$ for some $d_i(n)\in \Z_{>0}$. Since $M[\pE^{n-1}]=M[\pE^{n}][\pE^{n-1}]$, we see that, up to a permutation on the set $\{d_1(n),\dots ,d_r(n)\}$, we have $d_i(n-1)\leq d_i(n)$. We claim that there is $D\in \Z_{>0}$ such that $d_i(n)\leq D$ for all $i$ and all $n$. Indeed, if not, then for each $d\in \Z_{>0}$, $M[\pE]$ contains a nonzero element which is in $\pE^dM$. Since $M[\pE]$ is a {\it finite} set ($k_E$ being finite), we see that $M$ must contain a nonzero divisible element which is impossible. Since the $d_i(n)$ are bounded and increasing, we have $M[\pE^{n}]=M[\pE^{n+1}]$ for $n$ large enough which finishes the proof.
\end{proof}

The following lemma will often be tacitly used in the sequel.

\begin{lem0}\label{penible1}
Let $G_1$, $G_2$ be two $p$-adic analytic groups and $\Pi_1$, $\Pi_2$ two admissible unitary continuous representations of respectively $G_1$ and $G_2$ over $E$. Then the completed tensor product $\Pi_1\widehat \otimes_E \Pi_2$ is an admissible unitary continuous representation of the $p$-adic analytic group $G_1\times G_2$.
\end{lem0}
\begin{proof}
For $i\in \{1,2\}$ let $H_i\subseteq G_i$ be a compact open subgroup and $\Pi_i^0\subseteq \Pi_i$ a unit ball. By \cite[Thm.\ 2.3]{Sc1} we have to prove that $\Hom_{\oE}(\Pi_1^0\widehat\otimes_{\oE}\Pi_2^0,\oE)$ (the $\oE$-dual) is an $\oE[[H_1\times H_2]]$-module of finite type, where $\oE[[H_1\times H_2]]$ is the Iwasawa algebra of $H_1\times H_2$. Let $(H_{i,m})_{m\in \Z_{\geq 1}}$ be a decreasing sequence of normal compact open subgroups of $H_i$ so that $\oE[[H_i]]\cong \plim{m}{\oE[H_i/H_{i,m}]}$. Since $\Pi_i^0/\pE^n$ is (smooth) admissible, $(\Pi_i^0/\pE^n)^{H_{i,m}}$ is a finite type $\oE$-module for all $n,m\in \Z_{\geq 1}$ (and even a finite $\oE$-module as $\oE/\pE^n$ is finite). Since $\Pi_1^0/\pE^n\otimes_{\oE}\Pi_2^0/\pE^n$ is smooth, we have a topological isomorphism (for the profinite topology)
\begin{multline}\label{profi}
\Hom_{\oE}\big(\Pi_1^0/\pE^n\otimes_{\oE}\Pi_2^0/\pE^n,\oE/\pE^n\big)\cong \\
\plim{m}{\Big(\Hom_{\oE}\big((\Pi_1^0/\pE^n)^{H_{1,m}},\oE/\pE^n\big)\otimes_{\oE} \Hom_{\oE}\big((\Pi_1^0/\pE^n)^{H_{1,m}},\oE/\pE^n\big)\Big)}.
\end{multline}
Set $M_i:=\Hom_{\oE}(\Pi_i^0,\oE)$ which is an $\oE[[H_i]]$-module of finite type by assumption. Since
\begin{multline*}
(\Pi_i^0/\pE^n)^{H_{i,m}}=\Hom_{\oE}(M_i/\pE^n,\oE/\pE^n)^{H_{i,m}}=\\
\Hom_{\oE}\big(M_i/\pE^n\otimes_{\oE[[H_i]]}\oE[[H_i/H_{i,m}]],\oE/\pE^n\big),
\end{multline*}
we have by biduality,
\begin{equation}\label{bid}
M_i/\pE^n\otimes_{\oE[[H_i]]}\oE[[H_i/H_{i,m}]]\cong \Hom_{\oE}\big((\Pi_i^0/\pE^n)^{H_{i,m}},\oE/\pE^n\big)
\end{equation}
and thus an isomorphism of $\oE[[H_1/H_{1,m}]]\otimes \oE[[H_2/H_{2,m}]]=\oE[[H_1/H_{1,m}\times H_2/H_{2,m}]]$-modules of finite type,
\begin{multline*}
(M_1/\pE^n\otimes_{\oE}M_2/\pE^n)\otimes_{\oE[[H_1]]\otimes_{\oE}\oE[[H_2]]}\oE[[H_1/H_{1,m}\times H_2/H_{2,m}]] \cong \\
\Hom_{\oE}\big((\Pi_1^0/\pE^n)^{H_{1,m}},\oE/\pE^n\big)\otimes_{\oE} \Hom_{\oE}\big((\Pi_1^0/\pE^n)^{H_{1,m}},\oE/\pE^n\big).
\end{multline*}
Any finite type $\oE/\pE^n[[H_1\times H_2]]$ module $M$ satisfies
$$\plim{m}{\big(M\otimes_{\oE[[H_1\times H_2]]} \oE[[H_1/H_{1,m}\times H_2/H_{2,m}]]\big)}\cong M$$
(use $\pi\cong \ilim{m}{\pi^{H_{1,m}\times H_{2,m}}}$ and (\ref{bid}), where $\pi$ is the smooth representation of $H_1\times H_2$ corresponding to $M$). Taking the projective limit over $m$ and using (\ref{profi}) thus yields an isomorphism of $\oE/\pE^n[[H_1\times H_2]]$-modules of finite type,
\begin{multline*}
(M_1/\pE^n\otimes_{\oE}M_2/\pE^n)\otimes_{\oE[[H_1]]\otimes_{\oE}\oE[[H_2]]}\oE[[H_1\times H_2]]\cong \\
\Hom_{\oE}\big(\Pi_1^0/\pE^n\otimes_{\oE}\Pi_2^0/\pE^n,\oE/\pE^n\big).
\end{multline*}
Taking now the projective limit over $n$ yields an isomorphism of $\oE[[H_1\times H_2]]$-modules of finite type,
\begin{multline*}
(M_1\otimes_{\oE}M_2)\otimes_{\oE[[H_1]]\otimes_{\oE}\oE[[H_2]]}\oE[[H_1\times H_2]]\cong\\
\plim{n}{\Hom_{\oE}\big(\Pi_1^0/\pE^n\otimes_{\oE}\Pi_2^0/\pE^n,\oE/\pE^n\big)}\cong \Hom_{\oE}\big(\Pi_1^0\widehat\otimes_{\oE}\Pi_2^0,\oE\big),
\end{multline*}
which finishes the proof.
\end{proof}

We say that an admissible unitary continuous representation $\Pi$ of a $p$-adic analytic group $G$ is {\it residually of finite length} if for some (or equivalently any) unit ball $\Pi^0\subseteq \Pi$ preserved by $G$, the admissible smooth $G$-representation $\Pi^0\otimes_{\oE}k_E$ is of finite length.

\begin{lem0}\label{penible2}
Let $G_1$, $G_2$ be two $p$-adic analytic groups and $\Pi_1$, $\Pi'_1$ \(resp.\ $\Pi_2$\) some admissible unitary continuous representation\(s\) of $G_1$ \(resp.\ $G_2$\) over $E$. Assume that $\Pi'_1$ is residually of finite length and $\dim_E{\Hom}_{G_1}(\Pi'_1,\Pi_1)<\infty$. Then there is a canonical equivariant isomorphism of admissible unitary continuous representations of $G_2$ over $E$,
$${\Hom}_{G_1}(\Pi'_1,\Pi_1)\otimes_E\Pi_2\buildrel\sim\over\longrightarrow {\Hom}_{G_1}(\Pi'_1,\Pi_1\widehat\otimes_E\Pi_2).$$
\end{lem0}
\begin{proof}
Let $\Pi_1^0$, ${\Pi'_1}^{\!0}$ (resp.\ $\Pi_2^0$) be invariant unit balls in $\Pi_1$, $\Pi'_1$ (resp.\ $\Pi_2$), then we have ${\Hom}_{G_1}({\Pi'_1},\Pi_1\widehat\otimes_{E}\Pi_2)\cong E\otimes_{\oE}{\Hom}_{G_1}({\Pi'_1}^{\!0},\Pi_1^0\widehat\otimes_{\oE}\Pi_2^0)$. We have obvious isomorphisms
\begin{eqnarray}\label{limit}
{\Hom}_{G_1}({\Pi'_1}^{\!0},\Pi_1^0)&\buildrel\sim\over\rightarrow &\plim{n}{{\Hom}_{G_1}({\Pi'_1}^{\!0}/\pE^n,\Pi_1^0/\pE^n)}\\
\label{limit'}{\Hom}_{G_1}({\Pi'_1}^{\!0},\Pi_1^0\widehat\otimes \Pi_2^0)&\buildrel\sim\over\rightarrow & \plim{n}{{\Hom}_{G_1}({\Pi'_1}^{\!0}/\pE^n,\Pi_1^0/\pE^n\otimes \Pi_2^0/\pE^n)}.
\end{eqnarray}
Writing $\Pi_2^0/\pE^n=\ilim{i}{M_i}$, where $(M_i)_{i\in I}$ is an increasing sequence of free $\oE/\pE^n$-submodules of $\Pi_2^0/\pE^n$ of finite rank, we have isomorphisms
\begin{eqnarray}\label{homiso}
\nonumber {\Hom}_{G_1}({\Pi'_1}^{\!0}/\pE^n,\Pi_1^0/\pE^n\otimes \Pi_2^0/\pE^n)\!&\!\cong \!&\!{\Hom}_{G_1}\big({\Pi'_1}^{\!0}/\pE^n,\ilim{i}{(\Pi_1^0/\pE^n\otimes M_i)}\big)\\
\nonumber \!&\!\cong \!& \!\ilim{i}{{\Hom}_{G_1}\big({\Pi'_1}^{\!0}/\pE^n,\Pi_1^0/\pE^n\otimes M_i\big)}\\
\nonumber \!&\!\cong \!& \!\ilim{i}{\big({\Hom}_{G_1}\big({\Pi'_1}^{\!0}/\pE^n,\Pi_1^0/\pE^n\big)\otimes M_i \big)}\\
\!&\!\cong\!& \!{\Hom}_{G_1}\!\big({\Pi'_1}^{\!0}/\pE^n,\Pi_1^0/\pE^n\big) \!\otimes\!\Pi_2^0/\pE^n,
\end{eqnarray}
where we use that ${\Pi'_1}^{\!0}/\pE^n$ is a finite length representation of $G_1$ for the second isomorphism. Taking the projective limit over $n$ and using (\ref{limit'}), (\ref{limit}) together with Lemma \ref{proj} gives the result.
\end{proof}

Note that the assumption $\dim_E{\Hom}_{G_1}(\Pi'_1,\Pi_1)<\infty$ in Lemma \ref{penible2} is automatically \ satisfied \ if \ {\it also} \ $\Pi_1$ \ is \ residually \ of \ finite \ length \ (using that ${\Hom}_{G_1}({\Pi'_1}^{\!0},\Pi_1^0)/\pE$ embeds into ${\Hom}_{G_1}({\Pi'_1}^{\!0}/\pE,\Pi_1^0/\pE)$ which is finite-dimensional over $k_E$ as both ${\Pi'_1}^{\!0}/\pE$ and $\Pi_1^{0}/\pE$ are admissible smooth representations of $G_1$ of finite length).

\begin{lem0}\label{penible2bis}
Let $G_1$, $G_2$ be two $p$-adic analytic groups and $\Pi_1$, $\Pi'_1$ \(resp.\ $\Pi_2$, $\Pi'_2$\) be two admissible unitary continuous representations of $G_1$ \(resp.\ $G_2$\) over $E$. Assume that $\Pi_1$ and $\Pi'_1$ are residually of finite length and $\dim_E{\rm Ext}^1_{G_1}(\Pi'_1,\Pi_1)<\infty$. Let $\Pi$ be an extension of $\Pi'_1\widehat\otimes \Pi'_2$ by $\Pi_1\widehat\otimes \Pi_2$ \(in the category of admissible unitary continuous representations of $G_1\times G_2$ over $E$\). Then we have an exact sequence of admissible unitary continuous representations of $G_2$ over $E$,
\begin{multline*}
0\rightarrow \Hom_{G_1}(\Pi'_1,\Pi_1)\otimes_E \Pi_2\rightarrow \Hom_{G_1}(\Pi'_1,\Pi)\rightarrow \End_{G_1}({\Pi'_1})\otimes_E \Pi'_2\rightarrow \\
{\rm Ext}^1_{G_1}(\Pi'_1,\Pi_1)\otimes_E \Pi_2.
\end{multline*}
\end{lem0}
\begin{proof}
Note first that ${\Hom}_{G_1}(\Pi'_1,\Pi_1)$ and ${\End}_{G_1}(\Pi'_1)$ are finite dimensional $E$-vector spaces, as we saw above. Let $\Pi_1^0$, ${\Pi'_1}^{\!0}$, $\Pi_2^0$, ${\Pi'_2}^{\!0}$ be invariant unit balls in $\Pi_1$, $\Pi'_1$, $\Pi_2$, $\Pi'_2$ and $\Pi^0$ an invariant unit ball in $\Pi$ such that we have an exact sequence of $\oE[G_1\times G_2]$-modules,
$$0\longrightarrow \Pi_1^0\widehat\otimes \Pi_2^0\longrightarrow \Pi^0\longrightarrow {\Pi'_1}^{\!0}\widehat\otimes {\Pi'_2}^{\!0}\longrightarrow 0.$$
Let us first prove that it is enough to have an exact sequence for every $n\geq 1$,
\begin{multline}\label{n}
0\rightarrow \Hom_{G_1}({\Pi'_1}^{\!0}/\pE^n,\Pi_1^0/\pE^n)\otimes \Pi_2^0/\pE^n\rightarrow \Hom_{G_1}({\Pi'_1}^{\!0}/\pE^n,\Pi^0/\pE^n)\rightarrow \\
\End_{G_1}({\Pi'_1}^{\!0}/\pE^n)\otimes {\Pi'_2}^{\!0}/\pE^n\rightarrow {\rm Ext}^1_{G_1}({\Pi'_1}^{\!0}/\pE^n,\Pi_1^0/\pE^n)\otimes \Pi_2^0/\pE^n,
\end{multline}
where the ${\rm Ext}^1$ is in the abelian category of smooth representations of $G_1$ over $\oE/\pE^n$-modules. The Mittag-Leffler condition on the projective system
$$\big(\Hom_{G_1}({\Pi'_1}^{\!0}/\pE^n,\Pi_1^0/\pE^n)\otimes \Pi_2^0/\pE^n\big)_n$$
is satisfied since $\Pi_2^0/\pE^n\rightarrow \Pi_2^0/\pE^{n-1}$ is surjective and $\Hom_{G_1}({\Pi'_1}^{\!0}/\pE^n,\Pi_1^0/\pE^n)$ is a finite type $\oE/\pE^n$-module (as follows from the assumptions of residual finite length and a d\'evissage, see above). Thus the sequence remains exact after taking $\plim{}$ (decomposing a projective system of exact sequences $0\rightarrow A_n \rightarrow B_n\rightarrow C_n\rightarrow D_n$ as $0\rightarrow A_n \rightarrow B_n\rightarrow {\rm Im}(B_n)\rightarrow 0$ and $0\rightarrow {\rm Im}(B_n) \rightarrow C_n\rightarrow D_n$ and noting that the second sequence always remains exact at the limit since there is no surjectivity to check). One easily checks that the $\oE$-module $M:=\plim{}{\rm Ext}^1_{G_1}({\Pi'_1}^{\!0}/\pE^n,\Pi_1^0/\pE^n)$ is a submodule of ${\rm Ext}^1_{\oE[G_1]}({\Pi'_1}^{\!0},\Pi_1^0)$ (:= extensions as linear representations of $G_1$ on $\oE$-modules) and satisfies $E\otimes M\cong {\rm Ext}^1_{G_1}(\Pi'_1,\Pi_1)$. The natural surjection
$$\Hom_{G_1}({\Pi'_1}^{\!0},\Pi_1^0/\pE)\twoheadrightarrow {\rm Ext}^1_{\oE[G_1]}({\Pi'_1}^{\!0},\Pi_1^0)[\pE]$$
(see the notation in the proof of Lemma \ref{tf}) coming from the exact sequence $0\rightarrow \Pi_1^0\buildrel\pE\over\rightarrow \Pi_1^0\rightarrow \Pi_1^0/\pE\rightarrow 0$ shows that ${\rm Ext}^1_{\oE[G_1]}({\Pi'_1}^{\!0},\Pi_1^0)[\pE]$ is finite-dimensional over $k_E$. If $M_{\rm tors}$ is the torsion part of $M$, then {\it a fortiori} $M_{\rm tors}[\pE]=M[\pE]$ is also finite-dimensional. Moreover $M_{\rm tors}$ has no nonzero divisible element as it is contained in a projective limit of $\oE/\pE^n$-modules. Applying Lemma \ref{tf}, we see that $M_{\rm tors}$ is of finite type over $\oE$. Since $E\otimes M\cong {\rm Ext}^1_{G_1}(\Pi'_1,\Pi_1)$ is finite-dimensional over $E$, we get that $M$ is a finite type $\oE$-module. Using isomorphisms like (\ref{limit}) and Lemma \ref{proj} (which is were we use that $M$ is of finite type), we then conclude that taking the projective limit of (\ref{n}) and tensoring by $E$ gives the result. Let us now prove (\ref{n}). Applying $\Hom_{G_1}({\Pi'_1}^{\!0}/\pE^n,\cdot)$ to the exact sequence $0\rightarrow \Pi_1^0/\pE^n\otimes \Pi_2^0/\pE^n\rightarrow \Pi^0/\pE^n\rightarrow {\Pi'_1}^{\!0}/\pE^n\otimes {\Pi'_2}^{\!0}/\pE^n\rightarrow 0$ and using isomorphisms like (\ref{homiso}) we get an exact sequence
\begin{multline}\label{nn}
0\rightarrow \Hom_{G_1}({\Pi'_1}^{\!0}/\pE^n,\Pi_1^0/\pE^n)\otimes \Pi_2^0/\pE^n\rightarrow \Hom_{G_1}({\Pi'_1}^{\!0}/\pE^n,\Pi^0/\pE^n)\rightarrow \\
\End_{G_1}({\Pi'_1}^{\!0}/\pE^n)\otimes {\Pi'_2}^{\!0}/\pE^n\rightarrow {\rm Ext}^1_{G_1}({\Pi'_1}^{\!0}/\pE^n,\Pi_1^0/\pE^n\otimes \Pi_2^0/\pE^n)
\end{multline}
(the ${\rm Ext}^1$ are still in the category of smooth $G_1$-representations over $\oE/\pE^n$). Let $f\in \End_{G_1}({\Pi'_1}^{\!0}/\pE^n)\otimes {\Pi'_2}^{\!0}/\pE^n\cong \Hom_{G_1}({\Pi'_1}^{\!0}/\pE^n,{\Pi'_1}^{\!0}/\pE^n\otimes {\Pi'_2}^{\!0}/\pE^n)$, then $f$ defines a natural $G_1$-equivariant morphism ${\Pi'_1}^{\!0}/\pE^n\rightarrow {\Pi'_1}^{\!0}/\pE^n\otimes {\Pi'_2}^{\!0}/\pE^n$ and the image of $f$ in the ${\rm Ext}^1$ on the right is given by $0\rightarrow \Pi_1^0/\pE^n\otimes \Pi_2^0/\pE^n\rightarrow V\rightarrow {\Pi'_1}^{\!0}/\pE^n\rightarrow 0$, where $V$ is the fiber product
$$\begin{matrix} \Pi^0/\pE^n & \twoheadrightarrow &{\Pi'_1}^{\!0}/\pE^n\otimes {\Pi'_2}^{\!0}/\pE^n \\
\uparrow &&\uparrow f\\
V&\rightarrow &{\Pi'_1}^{\!0}/\pE^n\end{matrix}.$$
Lift in $V$ a finite set of generators of the $G_1$-representation ${\Pi'_1}^{\!0}/\pE^n$ (recall it is of finite length). Since $\Pi^0/\pE^n$ is a smooth representation of $G_2$, this finite set and the $G_1$-representation it generates both lie in
$$V \cap \big((\Pi^0/\pE^n)^{H_2}\times {\Pi'_1}^{\!0}/\pE^n\big)$$
for a sufficiently small compact open subgroup $H_2$ of $G_2$. Moreover we have an exact sequence of $G_1$-representations,
$$0\rightarrow \Pi_1^0/\pE^n\otimes (\Pi_2^0/\pE^n)^{H_2}\rightarrow V \cap \big((\Pi^0/\pE^n)^{H_2}\times {\Pi'_1}^{\!0}/\pE^n\big) \rightarrow {\Pi'_1}^{\!0}/\pE^n\rightarrow 0,$$
which gives back $V$ by pushout along $\Pi_1^0/\pE^n\otimes (\Pi_2^0/\pE^n)^{H_2}\hookrightarrow \Pi_1^0/\pE^n\otimes \Pi_2^0/\pE^n$. If $V$ is split, then any $G_1$-equivariant section ${\Pi'_1}^{\!0}/\pE^n\hookrightarrow V$ lies in $V \cap \big((\Pi^0/\pE^n)^{H_2}\times {\Pi'_1}^{\!0}/\pE^n\big)$ for $H_2$ sufficiently small (as follows again from the fact that ${\Pi'_1}^{\!0}/\pE^n$ is of finite length) and thus $V \cap \big((\Pi^0/\pE^n)^{H_2}\times {\Pi'_1}^{\!0}/\pE^n\big)$ is also split. Conversely, if $V \cap \big((\Pi^0/\pE^n)^{H_2}\times {\Pi'_1}^{\!0}/\pE^n\big)$ is split, then so is the pushout $V$. All this shows that one can replace ${\rm Ext}^1_{G_1}({\Pi'_1}^{\!0}/\pE^n,\Pi_1^0/\pE^n\otimes \Pi_2^0/\pE^n)$ in (\ref{nn}) by the inductive limit
\begin{equation*}
\ilim{H_2}{\rm Ext}^1_{G_1}\big({\Pi'_1}^{\!0}/\pE^n,\Pi_1^0/\pE^n\otimes (\Pi_2^0/\pE^n)^{H_2}\big).
\end{equation*}
Since $\Pi_2^0/\pE^n$ is admissible, this inductive limit can also be computed replacing the increasing sequence of submodules $(\Pi_2^0/\pE^n)^{H_2}$ by an increasing sequence of free $\oE/\pE^n$-submodules of $\Pi_2^0/\pE^n$ of finite rank, the union of which is $\Pi_2^0/\pE^n$. Thus we see that it is isomorphic to ${\rm Ext}^1_{G_1}({\Pi'_1}^{\!0}/\pE^n,\Pi_1^0/\pE^n)\!\otimes \Pi_2^0/\pE^n$, giving at last (\ref{n}).
\end{proof}

\begin{lem0}\label{penible3}
Let $G_1$, $G_2$ be two $p$-adic analytic groups and $\Pi_1$, $\Pi'_1$ \(resp.\ $\Pi_2$, $\Pi'_2$\) be two admissible unitary continuous representations of $G_1$ \(resp.\ $G_2$\) over $E$. Assume that $\Pi_1$ is residually of finite length, $\dim_E{\rm Ext}^1_{G_1}(\Pi'_1,\Pi_1)<\infty$, ${\Hom}_{G_1}(\Pi'_1,\Pi_1)=0$ and ${\End}_{G_1}(\Pi'_1)=E$.\\
\textup{(i)} Assume that $\Pi'_1$ is residually of finite length and either ${\Hom}_{G_2}(\Pi'_2,\Pi_2)=0$ or ${\rm Ext}^1_{G_1}(\Pi'_1,\Pi_1)=0$. Then we have
$${\rm Ext}^1_{G_1\times G_2}(\Pi'_1\widehat\otimes_E \Pi'_2,\Pi_1\widehat\otimes_E \Pi_2)=0.$$
\textup{(ii)} \ Assume \ that \ $\Pi_2$ \ is \ residually \ of \ finite \ length, \ $\dim_E{\rm Ext}^1_{G_1}(\Pi'_1,\Pi_1)=1$, $\dim_E{\rm Ext}^1_{G_2}(\Pi_2,\Pi_2)<\infty$ and ${\End}_{G_2}(\Pi_2)=E$. Then we have
\begin{equation*}
\dim_E{\rm Ext}^1_{G_1\times G_2}(\Pi'_1\widehat\otimes_E \Pi_2,\Pi_1\widehat\otimes_E \Pi_2)=1,
\end{equation*}
where the corresponding unique non-split extension is realized by $V_1\widehat\otimes_E\Pi_2$, $V_1$ being the unique non-split extension of $\Pi'_1$ by $\Pi_1$.
\end{lem0}
\begin{proof}
(i) Let $\Pi$ be an extension of $\Pi'_1\widehat\otimes \Pi'_2$ by $\Pi_1\widehat\otimes \Pi_2$, applying $\Hom_{G_1}(\Pi'_1,\cdot)$ we get by Lemma \ref{penible2bis} an exact sequence of admissible unitary continuous representations of $G_2$,
$$0\rightarrow \Hom_{G_1}(\Pi'_1,\Pi)\rightarrow \Pi'_2\rightarrow {\rm Ext}^1_{G_1}(\Pi'_1,\Pi_1)\otimes \Pi_2.$$
If $\Hom_{G_2}(\Pi'_2,\Pi_2)=0$ or ${\rm Ext}^1_{G_1}(\Pi'_1,\Pi_1)=0$ the map on the right is zero and tensoring by $\Pi'_1$ we deduce a commutative $G_1\times G_2$-equivariant diagram,
$$\begin{matrix} \Pi'_1\widehat\otimes \Hom_{G_1}(\Pi'_1,\Pi)& \buildrel\sim\over\longrightarrow &\Pi'_1\widehat\otimes \Pi'_2\\ \downarrow && \Vert \\ \Pi&\longrightarrow & \Pi'_1\widehat\otimes \Pi'_2\end{matrix}$$
yielding a splitting of $\Pi$.\\
(ii) Let us first prove $\dim_E{\rm Ext}^1_{G_1\times G_2}(\Pi'_1\widehat\otimes \Pi_2,\Pi_1\widehat\otimes \Pi_2)\leq 1$. Let $\Pi$ be an extension of $\Pi'_1\widehat\otimes \Pi_2$ by $\Pi_1\widehat\otimes \Pi_2$, applying $\Hom_{G_2}(\Pi_2,\cdot)$ we get by Lemma \ref{penible2bis} an exact sequence of admissible unitary continuous representations of $G_1$,
$$0\rightarrow \Pi_1\rightarrow \Hom_{G_2}(\Pi_2,\Pi)\rightarrow \Pi'_1\rightarrow {\rm Ext}^1_{G_2}(\Pi_2,\Pi_2)\otimes \Pi_1.$$
Since $\Hom_{G_1}(\Pi'_1,\Pi_1)=0$, the map on the right is zero and thus $\Hom_{G_2}(\Pi_2,\Pi)$ gives an element of ${\rm Ext}^1_{G_1}(\Pi'_1,\Pi_1)$. Tensoring by $\Pi_2$ we deduce a short exact sequence,
$$0\rightarrow \Pi_1\widehat\otimes \Pi_2\rightarrow \Hom_{G_2}(\Pi_2,\Pi)\widehat\otimes \Pi_2\rightarrow \Pi'_1\widehat\otimes \Pi_2\rightarrow 0$$
yielding a canonical $G_1\times G_2$-equivariant isomorphism $\Hom_{G_2}(\Pi_2,\Pi)\widehat\otimes \Pi_2\buildrel\sim\over\rightarrow \Pi$. This implies that tensoring by $\Pi_2$ induces a canonical surjection,
$${\rm Ext}^1_{G_1}(\Pi'_1,\Pi_1)\twoheadrightarrow {\rm Ext}^1_{G_1\times G_2}(\Pi'_1\widehat\otimes \Pi_2,\Pi_1\widehat\otimes \Pi_2)$$
and in particular, we get
$$\dim_E{\rm Ext}^1_{G_1\times G_2}(\Pi'_1\widehat\otimes \Pi_2,\Pi_1\widehat\otimes \Pi_2)\leq \dim_E{\rm Ext}^1_{G_1}(\Pi'_1,\Pi_1)=1.$$
Now the representation $V_1\widehat\otimes \Pi_2$ yields an element in ${\rm Ext}^1_{G_1\times G_2}(\Pi'_1\widehat\otimes \Pi_2,\Pi_1\widehat\otimes \Pi_2)$. This element is nonzero since $\Hom_{G_1}(\Pi'_1,V_1\widehat\otimes \Pi_2)=\Hom_{G_1}(\Pi'_1,V_1)\otimes \Pi_2=0$ (as follows from ${\Hom}_{G_1}(\Pi'_1,\Pi_1)=0$, ${\End}_{G_1}(\Pi'_1)=E$, $V_1$ non-split and Lemma \ref{penible2}) and thus {\it a fortiori} $\Hom_{G_1\times G_2}(\Pi'_1\widehat\otimes \Pi_2,V_1\widehat\otimes \Pi_2)=0$. This finishes the proof.
\end{proof}

\section{Some results on unitary continuous representations II}

We recall here more or less well-known results concerning unitary continuous principal series of $\Gp$ or of a product of $\Gp$ that are used in the text.

\begin{prop0}\label{fu}
For $i\in \{1,\dots,n\}$ \($n\in \Z_{\geq 1}$\) let $\chi_{1,i},\chi_{2,i}:\Qp^{\times}\rightarrow \oE^{\times}\subseteq E^{\times}$ be unitary continuous characters such that $\chi_{1,i}\ne \chi_{2,i}$. View $\chi_{2,i}\otimes\chi_{1,i}$ as a character of $\smat{*&0\\ *&*}\subseteq \Gp$ by $\smat{x & 0\\ t & y}\mapsto \chi_2(x)\chi_1(y)$ and define
$$\Pi_i:=\big(\Ind_{\smat{*&0\\ *&*}}^{\Gp}\chi_{2,i}\otimes\chi_{1,i}\big)^{{\mathcal C}^0},$$
where the continuous parabolic induction is as in \S\ref{prel}. Then the unitary continuous representation $\Pi_1\widehat\otimes_E\cdots\widehat\otimes_E\Pi_n$ of the product group $\Gp\times\cdots\times\Gp$ \($n$ times\) is admissible, topologically irreducible, residually of finite length and has only scalar endomorphisms.
\end{prop0}
\begin{proof}
  Throughout this proof, irreducible means topologically irreducible. The admissibility follows from Lemma \ref{penible1}. For the finite length mod $\pE$ and the scalar endomorphisms, it is enough to prove that the reduction $\overline\Pi_1\otimes_{k_E}\cdots\otimes_{k_E}\overline\Pi_n$, where $\overline\Pi_i:=\Ind_{\smat{*&0\\ *&*}}^{\Gp}\overline\chi_{2,i}\otimes\overline\chi_{1,i}$ (smooth induction), is of finite length and has scalar endomorphisms. From \cite[Thm.\ 30]{BL} and the (classical) fact that the tensor product of two smooth irreducible representations over $k_E$ with scalar endomorphisms is an irreducible representation of the product group with scalar endomorphisms, we deduce that this tensor product over $k_E$ is of finite length, indecomposable and with distinct constituents. Thus its endomorphisms are just $k_E$. Let us indicate how one can prove the irreducibility. Note that this is straightforward if all the $\overline\Pi_i$ are irreducible. First each $\Pi_i$ is irreducible. Indeed, if $\overline\Pi_i$ is reducible, then it is a non-split extension of a twist of the Steinberg representation (which is irreducible) by a $1$-dimensional representation (see \cite[Thm.\ 30]{BL}). It implies that any nonzero strict invariant closed subspace $\Pi'_i\subsetneq \Pi_i$ has to be $1$-dimensional, which is easily checked to be impossible since we assumed $\chi_{1,i}\ne \chi_{2,i}$ (one can also use locally analytic vectors as in what follows or as in \cite[Prop.\ 5.3.4]{Em1}). Secondly, it is enough to prove that any nonzero closed invariant subspace of the completed tensor product contains an element of the form $v_1\otimes v_2\otimes\cdots\otimes v_n$. Indeed, by irreducibility of each $\Pi_i$ (and since it is a closed subspace) it will then contain the whole usual tensor product and hence the completed tensor product (again as it is closed). To prove this we use locally analytic vectors. As the completed tensor product of principal series is a principal series for the product group (in both continuous and locally analytic worlds) and as $\chi_{1,i},\chi_{2,i}$, being continuous {\it characters} of $\Qp^{\times}$, are automatically locally analytic, one easily checks that the subspace of locally analytic vectors of $\Pi_1\widehat\otimes_E\cdots\widehat\otimes_E\Pi_n$ is the representation $\Pi_1^{\rm an}\widehat\otimes_E\cdots\widehat\otimes_E\Pi_n^{\rm an}$, where $\Pi_i^{\rm an}$ denotes the locally analytic principal series $\big(\Ind_{\smat{*&0\\ *&*}}^{\Gp}\chi_{2,i}\otimes\chi_{1,i}\big)^{\rm an}$ (\cite[\S3]{Sc1}). By density of locally analytic vectors in continuous admissible representations (\cite[Thm.\ 4.2]{Sc1}), it is enough to prove that the irreducible constituents of $\Pi_1^{\rm an}\widehat\otimes_E\cdots\widehat\otimes_E\Pi_n^{\rm an}$ are the finitely many representations $C_1\widehat\otimes_E\cdots\widehat\otimes_EC_n$, where $C_i$ is an irreducible constituent of $\Pi_i^{\rm an}$. Indeed, any nonzero closed invariant subspace will contain such a constituent and thus {\it a fortiori} a vector $v_1\otimes v_2\otimes\cdots\otimes v_n$. In other words one has to prove that $C_1\widehat\otimes_E\cdots\widehat\otimes_EC_n$ is irreducible. Changing the numbering, we can assume (see \cite[\S4]{ST}) that $C_i$ is locally algebraic if $1\leq i\leq m$ (in the sense of the Prasad's appendix in \cite{ST}) and is an irreducible locally analytic principal series if $m+1\leq i\leq n$ (where $0\leq m\leq n$). We can then rewrite the completed tensor product as $(C_1\otimes\cdots\otimes C_m)\otimes(C_{m+1}\widehat\otimes\cdots\widehat\otimes C_n)$ and one easily proves that it is irreducible if the factor $C_{m+1}\widehat\otimes\cdots\widehat\otimes C_n$ is irreducible (as $C_1\otimes\cdots\otimes C_m$ is obviously irreducible and the tensor product in the middle is a usual - i.e.\ not completed - tensor product). But this factor is a locally analytic principal series for the group $\Gp\times\cdots\times\Gp$ ($n-m$ times), and the fact it is irreducible follows then from the fact that the ``corresponding'' Verma module over the enveloping algebra of this group is obviously irreducible (being the tensor product of irreducible Verma modules for ${\rm GL}_2$): see \cite[\S4.1]{OS} (together with \cite[\S4]{ST}).
\end{proof}

The following proposition is not new (see e.g.\ \cite[Rk.\ 3.3.20]{Em4}, see also \cite{Ha1}), however, due to its importance in this paper, we provide a quick proof.

\begin{prop0}\label{ch}
Let $\chi_1,\chi_2:\Qp^{\times}\rightarrow \oE^{\times}\subseteq E^{\times}$ be unitary continuous characters such that $\chi_1\chi_2^{-1}\notin \{\varepsilon,\varepsilon^{-1}\}$ and view $\chi_2\varepsilon^{-1}\otimes\chi_1$ and $\chi_1\varepsilon^{-1}\otimes\chi_2$ as characters of $\smat{*&0\\ *&*}$ as in Proposition \ref{fu}.\\
\textup{(i)} If $\chi_1\ne \chi_2$, we have \(in the category of unitary continuous representations of $\Gp$\)
$$\dim_E{\rm Ext}^1_{\Gp}\Big(\big(\Ind_{\smat{*&0\\ *&*}}^{\Gp}\chi_2\varepsilon^{-1}\otimes\chi_1\big)^{{\mathcal C}^0},\big(\Ind_{\smat{*&0\\ *&*}}^{\Gp}\chi_1\varepsilon^{-1}\otimes\chi_2\big)^{{\mathcal C}^0}\Big)=1.$$
Moreover, this unique non-split extension has a central character.\\
\textup{(ii)} We have \(in the category of unitary continuous representations of $\Gp$\)
$$\dim_E{\rm Ext}^1_{\Gp}\Big(\big(\Ind_{\smat{*&0\\ *&*}}^{\Gp}\chi_1\varepsilon^{-1}\otimes\chi_2\big)^{{\mathcal C}^0},\big(\Ind_{\smat{*&0\\ *&*}}^{\Gp}\chi_1\varepsilon^{-1}\otimes\chi_2\big)^{{\mathcal C}^0}\Big)\leq 5.$$
\end{prop0}
\begin{proof}
The assumption $\chi_1\chi_2^{-1}\notin\{\varepsilon,\varepsilon^{-1}\}$ makes the two continuous $\Gp$-representations in (i) topologically irreducible (use e.g.\ Proposition \ref{fu}) and the assumption $\chi_1\ne \chi_2$ makes them distinct (Theorem \ref{classic}(iii)). The assertion on the central character is then automatic. To prove (i) and (ii), we apply the exact sequence
\begin{multline*}
0\rightarrow {\rm Ext}^1_M(U,\Ord_P(V))\rightarrow {\rm Ext}^1_G(\Ind_{P^-}^GU,V)\rightarrow \Hom_M(U,R^1\Ord_P(V))\rightarrow \\
{\rm Ext}^2_M(U,\Ord_P(V))
\end{multline*}
of \cite[(3.7.5)]{Em3} to $A=\oE/\pE^n$, $G=\Gp$, $P=\smat{*&*\\ 0&*}$, $M=\smat{*&0\\ 0&*}$, $U=\chi_2\varepsilon^{-1}\otimes\chi_1$ mod $\pE^n$ and $V=\Ind_{P^-}^G\chi_1\varepsilon^{-1}\otimes\chi_2$ mod $\pE^n$, where extensions are in the abelian category of locally admissible smooth representations over $\oE/\pE^n$ of the relevant group (see \cite[Def.2.2.17]{Em2}, the ${\rm Ext}^1$ are the same as in the category of admissible smooth representations). We have $\Ord_P(V)=\chi_1\varepsilon^{-1}\otimes\chi_2$ by \cite[Cor.\ 4.2.10]{Em3}. This implies in particular that ${\rm Ext}^1_M(U,\Ord_P(V))$ is an $\oE/\pE^n$-module of finite type (as follows for instance from a d\'evissage and \cite[Lem.\ 4.3.10]{Em3} together with \cite[Rem.\ 4.3.11]{Em3}). Since $R^1\Ord_P(V)=U$ by \cite[Cor.\ 4.2.10]{Em3} together with the comment after \cite[Conj.\ 3.7.2]{Em3}, the $\oE/\pE^n$-module $\Hom_M(U,R^1\Ord_P(V))$ is free of rank $1$. When $\chi_1\ne \chi_2$, let $N\in \Z_{\geq 0}$ be the maximal integer such that $\chi_1=\chi_2$ mod $\pE^N$, the same argument as in \cite[Lem.\ 4.3.10]{Em3} shows that ${\rm Ext}^i_M(U,\Ord_P(V))$, $i=1,2$ is an $\oE/\pE^N$-module. In particular the image ${\rm Im}({\rm Ext}^1)$ of ${\rm Ext}^1_G(\Ind_{P^-}^GU,V)$ in $\Hom_M(U,R^1\Ord_P(V))\cong \oE/\pE^n$ for $n>N$ contains $\pE^N(\oE/\pE^n)\cong \oE/\pE^{n-N}$. Passing to the projective limit over $n$ on the short exact sequences
\begin{multline*}
0\longrightarrow {\rm Ext}^1_M(U,\Ord_P(V))\longrightarrow {\rm Ext}^1_G(\Ind_{P^-}^GU,V) \longrightarrow {\rm Im}({\rm Ext}^1)\longrightarrow 0
\end{multline*}
(which still yields a short exact sequence since, ${\rm Ext}^1_M(U,\Ord_P(V))$ being a finite set for all $n$, the Mittag-Leffler conditions are satisfied) and tensoring by $E$ easily gives (i). (ii) is proved in the same way using that $\dim_E{\rm Ext}^1_M(\chi_1\varepsilon^{-1}\otimes\chi_2,\chi_1\varepsilon^{-1}\otimes\chi_2)=4$ (in the category of admissible unitary continuous representations of $M$ over $E$) and $\dim_E\Hom_M(\chi_1\varepsilon^{-1}\otimes\chi_2,\chi_2\varepsilon^{-1}\otimes\chi_1)\le 1$.
\end{proof}

\begin{rem0}
Pushing further the proof of Proposition \ref{ch}(ii) (using again \cite{Em3}), one gets that the dimension is in fact $4$ (at least when $p>2$).
\end{rem0}

\providecommand{\germ}{\mathfrak}

\bibliography{fondamental}
\bibliographystyle{amsalpha} 

\end{document}